\documentclass[12pt]{amsart}
\usepackage{amssymb}
\usepackage{amsmath}
\usepackage{amsfonts}
\usepackage{epsfig}
\usepackage{color}
\usepackage{epstopdf}
\usepackage{graphicx}
\usepackage{srcltx}
\usepackage{amsthm}
\usepackage{enumerate}
\usepackage[mathscr]{eucal}
\usepackage{verbatim}
\usepackage[titletoc,toc,title]{appendix}
\usepackage{hyperref}
 \theoremstyle{plain}
\newtheorem{theorem}{Theorem}[section]
\newtheorem{lemma}[theorem]{Lemma}

\newtheorem{rmk}[theorem]{Remark}
\theoremstyle{plain}

\newcommand{\nnorm}[1]{\Vert#1\Vert}
\newcommand{\abs}[1]{\vert #1 \vert}

\newcommand{\Tchi}[1]{T \chi_{#1}}
\newcommand{\TTchi}[1]{T^\ast \chi_{#1}}

\newcommand{\WTchi}[1]{\widetilde{T}\chi_{#1}}
\newcommand{\WTTchi}[1]{\widetilde{T}^{\ast}\chi_{#1}}
\title[Convolution estimates for measures on some complex curves]{Convolution  estimates for  measures\\on some complex curves}

\author{Hyunuk Chung}
\author{Seheon Ham}
\keywords{convolution estimate, complex curves, affine arclength
measure} \subjclass[2010]{42B10}

\address {Department of Mathematics\\ Pohang University of Science and Technology
\\ Pohang 790-784, Korea}
\email{nuki@postech.ac.kr}
\address {School of Mathematics\\ Korea Institute for Advanced Study
\\ Seoul 130-722, Korea}
\email{hamsh@kias.re.kr}
\date{ \today}

\begin{document}
\begin{abstract}
We consider the convolution operator for a measure supported on complex curves. The measure which we consider here is an analogue of
the affine arclength measure for real curves.
By modifying a combinatorial argument called the band structure argument, we  prove  the (nearly) optimal Lorentz space estimates. This includes the optimal strong type estimates as special cases.
The complex curves we consider here are the ones considered for the Fourier restriction estimates for complex curves in \cite{BH}.
\end{abstract}
\maketitle

\section{Introduction}
Let $h (z) = (z,z^2,\dots,z^{d-1},\phi(z))$, $d\ge2$, be a complex curve of simple type in $\mathbb C^d$, where $\phi(z)$ is an analytic function defined on
a region $\Omega \subset \mathbb C$. This is regarded as a 2-dimensional surface in $\mathbb R^{2d}$ defined by the real mapping
\[
z=(x,y) \mapsto h(x,y) = (x,y,x^2-y^2,2xy,\dots,\textrm{Re}(\phi(z)),\textrm{Im}(\phi(z))).
\]

We consider a convolution operator for a measure supported on the range of $h$, defined by
\[
\mathcal A f( \xi ) = \int_{D} f(\xi -h(z) ) \, |\phi^{(d)}(z)|^{\frac{4}{d(d+1)}} \, d\mu(z).
\]
Here $D$ is the unit ball in $\mathbb R^2$, $f$ is a Borel function on $\mathbb R^{2d}$, and $d\mu(z)$ is the surface measure $d\mu(z) = d\mu(h(z)) =
dx dy$, where $z = x+ i y$.
We have $|\phi^{(d)}(z)|^{\frac{4}{d(d+1)}} d\mu(z) = c_d \, |\det(h'(z),h''(z),\dots,h^{(d)}(z))|^{\frac{4}{d(d+1)}}\, d\mu(z)$,
which is an analogue of the affine arclength measure in the case of real curves.

It is obvious that $\mathcal A : L^\infty (\mathbb R^{2d}) \mapsto L^\infty(\mathbb R^{2d}) $, and by duality $\mathcal A : L^1 (\mathbb R^{2d}) \mapsto L^1(\mathbb R^{2d}) $.
By interpolation between these trivial estimates, we see that $\mathcal A : L^p \mapsto L^p$ for all $1 \le p \le \infty$.
We are interested in determining how much the integrability can be improved via the measure $$|\phi^{(d)}(z)|^{\frac{4}{d(d+1)}} d\mu(z) = :d\sigma (z) .$$
In the nondegenerate case, i.e. when  $\phi(z) = z^d$, one can see that
$\mathcal A$ is not of restricted weak type $(p,q)$ outside of
the (closed) trapezoid $\mathcal R$ with vertices at $(0,0)$, $(1,1)$, $(1/p_d,1/q_d)$, and $(1/q'_d,  1/p'_d)$ in the plane.
Here $p_d = \frac{d+1}{2}$ and $q_d = \frac{d(d+1)}{2(d-1)}$. (We will discuss the necessary conditions in Appendix B.)
Hence, if we show the restricted weak type for $\mathcal A$ at $(p_d,q_d)$ (then by duality $\mathcal A$ is also of restricted weak type of $(q_d',p_d')$),
then $\mathcal A$ is of strong type $(p ,q)$ on $\mathcal R$
%(including the boundary)
except for the points $(1/p_d,1/q_d)$ and $(1/q'_d,  1/p'_d)$.
%it suffices to show that $\mathcal A$ is of strong type $(p_d,q_d)$.
%The boundedness from $L^{q'_d}$ to $L^{p'_d}$ follows by duality.
%Thus the aim is to obtain a strong type estimate $(p_d,q_d)$ (and also $(p_d', q_d')$) for $\mathcal A$. % at $(1/p_d,1/q_d)$, and $(1/q'_d,  1/p'_d)$.
%Moreover, we shall deal with some degenerate curves of simple type.
In addition to the restricted weak type at $(p_d,q_d)$, we will show the optimal boundedness of $\mathcal A$ on the scale of Lorentz spaces.

The following is our first result.

%For some degenerate cases, we claim that $\mathcal A$ is of strong type $(p,q)$ for all $(p,q) \in \mathcal R$.
%Our first result deals with degenerate monomial curve of simple type.
%It is well-known that the affine arclength measure of a curve plays a crucial role in obtaining a uniform estimate for many problems in harmonic analysis.

%We obtain strong type estimates for a monomial case.

\begin{theorem}\label{thm:monomial}
Let $d \ge 2$ and $\phi(z) = z^N$ for a nonnegative integer $N$. %Also let $q_d = \frac{d(d+1)}{2(d-1)}$ and $p_d = \frac{d+1}{2}$.
Then there exists a constant $C=C(d,N)$ such that
\begin{align}\label{main_inq:mono}
\| \mathcal A f \|_{L^{q_d,v} (\mathbb R^{2d})} \le C  \| f \|_{L^{p_d,u}(\mathbb R^{2d})}\quad\textrm{ for }\, u < q_d , \,  p_d < v , \,\textrm{and }\,  u<v .
%\intertext{or equivalently}
%\| \mathcal A f \|_{L^{p'_d,q'_d+\varepsilon}(\mathbb R^{2d})} %\le C(\varepsilon,N) \| f \|_{L^{q'_d}(\mathbb R^{2d})},
\end{align}
Here $C = C(d,N)$ depends only on $d,N$.
\end{theorem}
%These results contain the non-degenerate case, i.e. $\phi(z)= z^d$, which is one answer proposed in the remark at the end of \cite{Ch1}.
%Also, the range of boundedness of $\mathcal A : L^{p_d} \mapsto L^{q_d}$ is optimal when $\phi(z) = z^d$. We will show this at the end of this section.

This implies that $\mathcal A$ is of strong type in $\mathcal R$.
This result relies on an estimate for a lower bound of some Jacobian (see Section \ref{Sec:Jaco}) arisen from  the  Fourier restriction theorem for complex curves (see \cite{BH}), where a uniform Fourier restriction estimate for polynomial curves of simple type for $d=3$ was also obtained.
Using this, we thus get the following.
%Bak and the second author also obtained a uniform Fourier restriction estimate for polynomial curves of simple type when $d=3$.

%Especially in 3-dimensional case, we obtain a uniform estimate when $\phi(z)$ is an arbitrary polynomial of degree at most $N$.
\begin{theorem}\label{thm:poly}
When $d=3$ and $\phi(z)$ is a polynomial of degree at most $N$, \eqref{main_inq:mono} holds for $h(z) = (z,z^2,\phi(z))$ and a constant $C=C(N)$, which depends only on $N$.
\end{theorem}

For a two-dimensional surface $h(z) = (z,\phi(z))$ in $\mathbb R^4$, where $\phi(z)$ is not necessarily holomorphic, Drury and Guo \cite{DG} showed the
$L^{3/2}(\mathbb R^4) \rightarrow L^3(\mathbb R^4)$ boundedness for the convolution operator defined by an induced measure on the surface $h(z)$ under some nondegeneracy conditions on
$\phi(z)$.

Our approach basically follows the case of real curves. Let $\gamma(t) : I =[0,1] \mapsto \mathbb R^d$ be a smooth space curve in $\mathbb R^d$. Let us denote the
affine arclength measure along $\gamma(t)$ by
\[
w(t) dt =  [ \det \begin{pmatrix}\gamma'(t) , \gamma''(t),\dots,  \gamma^{(d)}\end{pmatrix} ]^{\frac{2}{d(d+1)}} dt.
\]
There has been much work about estimates for the convolution operator given by
\[
 \mathcal B f(x) = \int_I f(x-\gamma(t)) w(t) dt.
\]

% of $L^{p_d} - L^{q_d}$ estimate.
%(Recall that $p_d = \frac{d+1}{2}$ and $q_d = \frac{d(d+1)}{2(d-1)}$.)
When $\gamma(t)  = (t,t^2,\dots,t^d)$ i.e. $w(t)dt \sim dt$,
Littman \cite{L} ($d=2$) and Oberlin \cite{O,O2,O4,O5} established the strong type $(p_d,q_d)$ for $d=3$ and the restricted weak type $(p_d,q_d)$ for $d=4$.
The endpoint case in $d=2$, i.e. $L^{3/2,r}(\mathbb R^2) \rightarrow L^{3,r}(\mathbb R^2)$ estimate for all $p_2= 3/2 \le r \le 3 = q_2$ was shown by Bak, Oberlin, and Seeger \cite{BOS}.
The estimate $L^2(\mathbb R^3) \rightarrow L^{3/2,2}(\mathbb R^3)$ (which is the case $u = v = p_3 = 2$) was established by Bennett and Seeger in \cite{BS}, where the result was proved by analyzing the singularities of the phase function of a certain oscillatory integral operator.
When $d \ge 2$, Christ \cite{Ch1} proved the restricted weak type $(p_d,q_d)$ for $\mathcal B$ by using the band structure argument.
This argument was extended by Stovall \cite{S1} to establish the $L^{p_d,u} \rightarrow L^{q_d,v}$ estimates for  $ u < q_d$, $v > p_d$, and $u < v$.

%For degenerate curves, a class of monomial curve
Gressman \cite{G} proved the restricted weak type $(p_{\mathbf a},q_{\mathbf a})$ for a class of monomial curves $\gamma(t) = (t^{a_1},\dots,t^{a_d})$ for positive integers $a_1 < \dots <a_d$, and for $w(t) =|t|^b$, $-1 < b \le 0$.
Here, $(p_{\mathbf a},q_{\mathbf a})$ depends on the exponents $a_1,\dots, a_d$,  where ${\mathbf a} = (a_1,\dots, a_d)$.
See also \cite{P, O3, Choi2} for the $d=3$ case.

For more details on the history related to general classes of curves, we refer to \cite{S2, DLW}  and the references therein.
Here we will focus on some results related to our work.
Drury \cite{D} introduced the affine arclength measure to obtain optimal $L^{p_2} \rightarrow L^{q_2}$ estimates for $(t, p(t))$, where $p(t)$ satisfies some technical assumptions.
(See also \cite{Choi1}.)
Let $\gamma(t) = (P_1(t) ,\dots,P_d(t))$ for arbitrary polynomials $P_i(t)$, $1\le i \le d$.
When $d=2$, Oberlin \cite{O5} showed optimal $L^{p_2} \rightarrow L^{q_2}$ boundedness of  $\mathscr A $, where the constant depends only on the maximum degree of the polynomials.
Dendrinos, Laghi and Wright \cite{DLW} obtained the uniform boundedness of $\mathcal B$ in some Lorentz spaces when $d=3$,
which was an extension of the case $\gamma(t) = (t, P_1(t), P_2(t))$ established by Oberlin \cite{O5}.
For the general dimension, Stovall \cite{S2} proved the $L^{p_d,u} \mapsto L^{q_d,v}$ boundedness whenever $u < q_d$, $p_d < v$, and $ u < v$.

For curves with less regularity, more conditions on the torsion were needed. Oberlin \cite{O6} proved the sharp strong type boundedness for $\mathcal B$ in $d=2,3,4$ for certain flat curves of simple type, where the weight function is monotone and $\log$-concave. For the higher dimensional cases it seems to be difficult to construct the band structure. Recently, Dendrinos and Stovall \cite{DS} proved the restricted weak type estimates for certain curves (not necessarily simple) with low regularity under monotonicity and concavity assumptions on the affine arclength measure. They also obtained strong type $(p_d,q_d)$ for the monomial-like curves. By suitably ordering certain parameters, they efficiently avoid the band structure argument.

We basically follow the argument in \cite{S2}, which exploited geometric inequalities arisen from Fourier restriction estimates for polynomial curves (see \cite{DW}) and the band structure argument in \cite{Ch1, S1}.

%\chg{more references}

%As used in study of restriction of Fourier transform to space curves, the affine arclength measure $w(t) dt$ makes it possible to obtain a uniform estimate for a class of curves since $w(t)$ vanishes where the torsion of $\gamma$ vanishes.
%By this one can obtain a uniform estimate for a class of curves.

As in the study of the Fourier restriction estimates for space curves, properties of the mapping $(t_1,\dots,t_d) \mapsto \sum_{i=1}^d \gamma(t_i)$, such as finite generic multiplicity and a lower bound for the Jacobian, are required in this paper to perform a change of variables.
As in \cite{S2}, which relied on the uniform estimates for the Jacobian in \cite{DW}, we rely on an analogous (but weaker) result in \cite{BH}, related to the Fourier restriction estimates for complex curves.
In fact, we can decompose $\mathbb C$ into finitely many disjoint regions such that the holomorphic Jacobian of the mapping $(z_1,\dots,z_d) \mapsto \sum_{i=1}^d h(z_i)$ is bounded below on each region by the complex Vandermonde determinant and the arithmetic mean of $|\phi^{(d)}(z_1)|,\dots,|\phi^{(d)}(z_d)| $.
See Section \ref{Sec:Jaco} for more details.

We wish to point out here that the affine arclength measure was first used in the study of the Fourier restriction estimates for various classes of degenerate curves to allow the possibility of uniform estimates by mitigating the degeneracy of the torsion of the curve.
(See \cite{BOS1,BOS2,BOS3,DM,DW,D,DM1,DM2,S3} %[Why not include BOS1, American J. Math, here too, since it was almost the first paper in that direction??]
for the case of degenerate real curves, and see \cite{Ob,BH} for complex curves.)

The usual treatment of the affine arclength measure involved a lower bound of a certain Jacobian in terms of the geometric mean of the translates of the torsion. Here, following \cite{BOS3, BH}, we will exploit a stronger estimate involving the arithmetic mean (or, equivalently, the maximum) of the translates of the torsion (in Lemma \ref{jaco_poly} and \ref{jaco_mono}) rather than their geometric mean.
Using $\max_{1\le i \le d} |\phi^{(d)}(z_i)|$ instead of $\prod_{1 \le i \le d} |\phi^{(d)}(z_i)|^{\frac 1d}$, our computation can be simplified quite a bit, since it allows us to manipulate the exponents easily by choosing $c_1,\dots,c_d$ appropriately such that
\[
 \max_{1\le i \le d} |\phi^{(d)}(z_i)| \ge \prod_{1\le i \le d} |\phi^{(d)}(z_i)|^{c_i} \, \textrm{, where } \sum_{i=1}^d c_i =1.
 \]

%\chg{By this, the restricted strong type estimates at the endpoint for the Fourier restriction theorems for some space curves was possible. (See \cite{BOS3} or \cite{BH}.)}

If one can obtain Lemma \ref{jaco_poly} for complex polynomial curves in general dimensions, Theorem \ref{thm:monomial} may be extended to cover those curves by using the modified band structure argument for complex variables.
The idea is to consider balls in $\mathbb C$ in place of the distance between real variables, which we will explain in more detail in Section 3 (which was motivated by the proof of Lemma \ref{trilinear_3d}).
By this we can obtain an estimate on Lorentz spaces.% (hence strong type $(p_d,q_d)$) although it is not optimal.

For the optimal estimate we need further observations, which will be given in Section 4.
For the case $d=3$, the argument is simpler, since it does not need the band structure argument. This case will be discussed in Section 5 for the sake of completeness.
The standard method to obtain Theorem \ref{thm:monomial} and \ref{thm:poly} from the main lemmas in Sections 3 and 4 was established by Stovall \cite{S1}.
For the sake of completeness a detailed proof will be given in Appendix A.
The necessary conditions on the indexes $p_d, q_d, u, v$ will be discussed in Appendix B.

\section{Lower bounds for the Jacobian}\label{Sec:Jaco}
In this section, we recall the lower bounds for the Jacobian for $h(z) = (z,z^2, \phi (z))$, where $\phi(z)$ is an arbitrary polynomial.
 %and $h(z) = (z,z^2,\dots,z^{d-1},z^N)$. %Lemma 4.2 in \cite{BH}.
Let $J_{\mathbb C}(z_1,z_2,z_3)$ be the determinant of the holomorphic Jacobian of the mapping $(z_1,z_2,z_3) \mapsto \Phi_h(z_1,z_2,z_3) =
-h(z_1)+h(z_2)-h(z_3)$.
Then $\mathbb C$ is decomposed into a bounded number of regions, on which a lower bound of $J_{\mathbb C}(z_1,z_2,z_3)$ may be given as follows.
\begin{lemma}[Lemma 4.2 in \cite{BH}]\label{jaco_poly}
There exists a positive integer $M = M(N)$ and a collection of convex open sets $B_1,\dots,B_M$, which are pairwise disjoint, such that $\mathbb C =
\cup_{\ell=1}^M B_\ell$ ignoring null sets. Moreover, there exists a constant $c(N) >0$ such that for $1 \le \ell \le M$,
\begin{align*}
|J_{\mathbb C}(z_1,z_2,z_3) | \ge c(N) V (z_1,z_2,z_3) \max\{ |\phi'''(z_1)|,|\phi'''(z_2)|,|\phi'''(z_3)|\},
\end{align*}
whenever $z_1,z_2,z_3 \in B_\ell$.
Here $V(z_1,z_2,z_3) = |z_2-z_1||z_3-z_1||z_3-z_2|$ is the Vandermonde determinant.
\end{lemma}

Let us describe the set $B_\ell$ in brief. Fix a zero $u_1$ of $\phi'''(z)$ and denote the other $N-4$ zeros by $u_2$,\dots,$u_{N-3}$ such that $|u_2-u_1|\le
\dots \le |u_{N-3} -u_1|$. %By translation, we may assume $u_1=0$.
Set $S(u_1)= \{ z \in \mathbb C : |z-u_1| < |z-u_k| \textrm{ for } k=2,\dots,N-3 \}$. By translation, we may assume that $u_1=0$.
Then we define the gap annuli $G_k$ and the dyadic annuli $D_k$ by
\begin{align*}
G_k &= \{ z_1 \in \mathbb C : A |u_k| \le |z_1| \le A^{-1} |u_{k+1}| \} \textrm{ for } 1 \le k \le N-4,    \\
D_k &= \{ z_1 \in \mathbb C : A_1^{-1} |u_k| \le |z_1| \le A_1 |u_k| \} \textrm{ for } 2 \le k \le N-3,
\end{align*}
and $G_{N-3}  = \{ z_1 \in \mathbb C : A |u_{N-3}| \le |z_1| \}$.
Here $A$ and $A_1$ are appropriate constants. (See the proof of Lemma 4.2 in \cite{BH}.)
In addition, let us consider the collection of narrow sectors $\{\Delta\}$ centered at the origin with angle $\varepsilon$, which cover $\mathbb C$.
Since the operator $\mathcal A$ is invariant under an affine transformation, it suffices to consider one sector $\Delta = \{ z= r e^{i\theta} : ~0< r , \, \theta\in(0,\varepsilon) \}$.
Then $B_\ell$ is a convex subset of $S(u_1)\cap \Delta \cap E_k$ for some $E_k = G_k$ or $D_k$.
By the proof of Lemma 4.2 in \cite{BH}, we have that
\begin{equation}\label{assum:poly}
|\phi'''(z)| \approx \prod_{n=k+1}^{N-3} |u_n| |z|^k =: H_k |z|^k
\end{equation}
whenever $z \in B_\ell \subset S\cap \Delta\cap E_k$ with $E_k = G_k$ or $D_k$.

In fact, $|J_{\mathbb C} (z_1,z_2,z_3)|$ can be reduced to the determinant of the holomorphic Jacobian of the mapping $(z_1,z_2,z_3) \mapsto \Phi_{\tilde h}(z_1,z_2,z_3)$, where $\tilde h(z) = (z,z^2, z^{k+3})$ on $B_\ell$.
More precisely, it is known that
\begin{align}\label{jaco:poly-mono}
|J_{\mathbb C} (z_1,z_2,z_3) | \gtrsim H_k \cdot V(z_1,z_2,z_3) \cdot \max\{ |z_1|^k,|z_2|^k,|z_3|^k \}
\end{align}
if $z_1,z_2,z_3 \in B_\ell$ for some $B_\ell \subset S\cap \Delta\cap E_k$ with $E_k = G_k$ or $D_k$.

When $h(z)$ is a monomial curve of simple type, we have the following.
\begin{lemma}[Lemma 3.3 in \cite{BH}]\label{jaco_mono}
Let $h(z) = (z,z^2,\dots,z^{d-1},z^N)$ for an integer $N\ge d$ with $d \ge 2$. Set
\[
J_d(z_1,\dots,z_d) = J_{\mathbb C} (z_1,\dots,z_d) = \det(h'(z_1),\dots,h'(z_d))
\]
where $z_j \in \mathbb C$, $1\le j \le d$. Then $\mathbb C$ may be written as the union (ignoring a null-set) of $C(d,N)$ sectors $\Delta_\ell$ with vertex
at the origin such that for each $1 \le \ell \le C(d,N)$, we have
\begin{align*}
|J(z_1,\dots,z_d)| \ge c(d,N) \,\max_{1\le j \le d}|z_j|^{N-d} \prod_{1 \le i < j \le d} |z_j - z_i|
\end{align*}
where $z_j \in \Delta_\ell$. Here, $C(d,N)$ and $c(d,N)$ are positive constants depending only on $d$ and $N$.
\end{lemma}
In this case, it suffices to consider the case when $\Delta_{\ell} = \Delta = \{ re^{i \theta} : r > 0 \textrm{ and } \theta \in (0,\epsilon) \}$ for some
small $\epsilon$.
\section{The band structure for complex variables}

In this section, we consider a monomial curve $h(z) = (z,z^2,\dots,z^{d-1},z^N)$ for a nonnegative integer $N$. We may assume $N \ge d$, because $\mathcal A
f(x) =0$ for $N <d$. To handle the general dimensional case, we basically follow the `band structure' argument in \cite{Ch1}.
%As previously discussed, our goal is to obtain the following lemma, which gives a strong type estimate for $\mathcal A$.

By Lemma \ref{jaco_mono} it suffices to consider
\begin{align}\label{operator}
T f(x) = \int_{\Delta} f(x-h(z))d\sigma(z),
\end{align}
where $d\sigma(z) \sim |z|^{\frac{4(N-d)}{d(d+1)}} d\mu(z) =: |z|^{\frac{4K}{d(d+1)}} d\mu(z)$ and $\Delta = \{ re^{i \theta} : 0 < r \le 1 \textrm{ and } \theta
\in (0,\epsilon) \}$ for a small constant $\epsilon$ as in Lemma \ref{jaco_mono}.

The estimate \eqref{main_inq:mono} for $T$ follows from the propositions in the Appendix.
Those propositions assume the restricted weak type $(p_d,q_d)$, which we will now prove.

Let us define quantities $\alpha$ and $\beta$ by
\begin{align*}
\alpha = \frac{\left< \Tchi E , \chi_F \right>}{|F|} \quad \textrm{ and }\quad
\beta = \frac{\left< \TTchi F , \chi_E \right>}{|E|},
\end{align*}
for measurable sets $E, F$.
Here $T^\ast$ is the dual operator of $T$ given by $T^\ast f(x) = \int_\Delta f(x + h(z))d\sigma(z)$.
Note that $\alpha|F| = \beta|E|$.
Then
\begin{equation*}
\langle \Tchi{E}, \chi_F \rangle \lesssim |E|^{\frac{1}{p_d}} |F|^{\frac{1}{q_d'}}
\end{equation*}
is equivalent to
\begin{align}\label{Eab1}
|E| \gtrsim \alpha^{\frac{d(d+1)}{2}} \left(\frac{\beta}{\alpha}\right)^{d-1}
\end{align}
since $\alpha|F| = \beta |E| = \langle \Tchi{E}, \chi_{F} \rangle$.

The following lemma is a refinement of \eqref{Eab1}.
\begin{lemma}\label{d-trilinear}
Let $E_1,E_2,G \subset \mathbb R^{2d}$ be measurable sets with finite measure. Suppose that
\begin{align*}
\Tchi{E_1}(x) \ge \alpha_1 \textrm{   and   } \Tchi{E_2} (x) \ge \alpha_2
\end{align*}
for all $x \in G$ and $\alpha_1 \le \alpha_2$. Then
\begin{align*}
| E_2| \gtrsim \alpha_1^{\frac{d(d+1)}{2}} \,\left( \frac{\beta}{\alpha_1}\right)^{d-1} \, \left(\frac{\alpha_2}{\alpha_1}\right)^{d},
\end{align*}
where $\beta = \alpha_1 \frac{|G|}{|E_1|}$.
\end{lemma}

\begin{rmk}
If we set $E_1 =E_2$ and $\alpha_1 = \alpha_2$, we obtain \eqref{Eab1} i.e. the restricted weak type $(p_d, q_d)$ for $T$.
\end{rmk}
%\subsubsection*{Proof of Lemma}
First we will find a sequence of subsets of $\Delta$ and their properties, which is essential to construct a band structure of Lemma \ref{band}.

\begin{lemma}\label{setup}
Let $\gamma=\max\{\alpha_{1},\beta\}$, $\nu = \frac{d(d+1)}{4K+2d(d+1)}$ under the assumptions in Lemma \ref{d-trilinear}. There exist a point $y_0$ in $E_1$, a constant $C>0$, and a sequence of subsets $P_1,\dots,P_{2d}$ of $\Delta$ such that
\begin{enumerate}[(i)]
\item $\sigma(P_j) \ge C \beta $ for odd $j$,
\item $\sigma(P_j) \ge C \alpha_1 $ for even $j < 2d$,
\item $\sigma(P_{2d}) \ge C \alpha_2 $,
\end{enumerate}
and $|z_j| \ge ( 4\pi\nu )^{-\nu} \gamma^\nu $ for $ z_j \in P_j$, $1 \le j \le 2d-1$, and $|z_{2d}| \ge ( 4\pi\nu )^{-\nu} \alpha_2^\nu$.

Also there exists a small constant $c>0$ such that
\begin{enumerate}
\item[(iv)] if $z_j \in P_j$ for odd $j$, then $|z_j - z_i| \ge c \beta^{\frac12}\abs{z_{i}}^{\frac{-2K}{d(d+1)}}$, where $i < j < 2d$,
\item[(v)] if $z_j \in P_j$ for even $j<2d$, then $\abs{z_{j}-z_{i}}\ge c \alpha_{1}^{\frac12} \abs{z_{i}}^{\frac{-2K}{d(d+1)}}$, where $i<j$,
\item[(vi)] for $z_{2d} \in P_{2d}$ and $j < 2d$, %$\epsilon >0$, and $2d > j $, $|z_{2d}| \ge c\alpha_2^{\nu}$ and
$\abs{z_{2d} -z_j}\ge c \alpha_2^{\frac12} \abs{z_{2d}}^{\frac{-2K}{d(d+1)}}$ if $\abs{z_j}< \frac12 ( 4\pi\nu )^{-\nu} \alpha_{2}^{\nu}$, and
$\abs{z_{2d}-z_{j}} \ge c \alpha_{2}^{\frac12}\abs{z_{j}}^{\frac{-2K}{d(d+1)}}$ if $\abs{z_j} \ge \frac12( 4\pi\nu )^{-\nu} \alpha_{2}^{\nu}$.
%$| z_{2d} - z_j | \ge c \alpha_2^{\frac12}|z_{2d}|^{\frac{-2k}{d(d+1)}}$
\end{enumerate}
\end{lemma}

\begin{proof}
We begin by showing that we may consider a truncated operator instead of $T$ by following the proof of Lemma 1 in \cite{G}. %??(Is something missing here?)
Let us define $B_{\gamma}=\{z \in \Delta : \abs{z} < ( 4\pi\nu )^{-\nu} \gamma^{\nu} \}$
%We will use the method of Gressman in \cite{G}.
and a truncated operator $\widetilde{T}f(x)=\int_{\Delta\setminus B_{\gamma}} f(x-h(z)) d \sigma(z)$.
%We claim that $\widetilde{T}\chi_{E_1}(x) \ge \frac12 \alpha_1$ for $x \in \widetilde G \subset G$ by choosing sufficiently small $\delta$.
%Also we can see that $|\widetilde G | \ge \frac12 |G|$.

Since
\[
\sigma(B_{\gamma})=\int_{B_\gamma} \abs{z}^{\frac{4K}{d(d+1)}} d \mu (z)
\le  2 \pi \int_{0}^{( 4\pi\nu )^{-\nu}\gamma^{\nu}} s^{\frac{4K}{d(d+1)}+1}d s = 2 \pi \nu (4\pi\nu)^{-1} \gamma = \frac \gamma 2,
\]
we have that
\begin{align*}
\langle \WTchi{E_1}, \chi_{G}\rangle \ge \langle \Tchi{E_1}, \chi_{G}\rangle - \sigma(B_{\gamma})\abs{G} \ge (\alpha_{1}-\gamma/2)\abs{G}.
\end{align*}
If $\gamma=\alpha_{1}$, then $\langle \WTchi{E_1}, \chi_{G}\rangle \ge \frac{1}{2}\alpha_{1}\abs{G} =\frac12 \beta |E_1|$.% by choosing $\delta=(\frac{1}{4\pi \nu})^{\nu}$.

If $\gamma=\beta$, we see that
\begin{align*}
\langle \WTchi{E_{1}}, \chi_{G} \rangle = \langle \chi_{E_{1}}, \WTTchi{G} \rangle \ge \langle \chi_{E_{1}}, \TTchi{G} \rangle -
\sigma(B_{\gamma})\abs{E_{1}} \ge \frac1 2 \beta \abs{E_{1}}
\end{align*}
since $\langle \chi_{E_{1}}, \TTchi{G} \rangle \ge \alpha_1|G| = \beta|E_1|$.
%We also choose $\delta=(\frac{1}{4\pi\nu})^{\nu}$ so that $\langle \WTchi{E_{1}}, \chi_{G} \rangle \ge \frac{1}{2}\beta\abs{E_{1}}$. % \frac 12 \alpha_1|G|$.

Therefore %if $\delta \le (\frac1{4\pi\nu})^\nu$, then
we get that
\begin{align}\label{eq:j=0}
\langle \WTchi{E_{1}}, \chi_{G}\rangle \ge \frac{1}{2} \alpha_{1}\abs{G} = \frac12 \beta |E_1|.
\end{align}

%And we check in the case $E_{2}$ (when $\beta \le \alpha_{1}$):
%\begin{align*}
%\langle \WTchi{E_{2}}, \chi_{G} \rangle &\ge \langle \Tchi{E_{2}}, \chi_{G} \rangle - \sigma(B_{\delta})\abs{G} \\
%&\ge \alpha_{2} \abs{G} - \sigma(B_{\delta})\abs{G} \\
%&=\alpha_{2}\abs{G} - 2\pi \nu \delta^{1/\nu}\alpha_{1} \abs{G} \\
%&\ge \frac{1}{2}\alpha_{2}\abs{G}
%\end{align*}
%when $\delta=(\frac{1}{4\pi \nu})^{\nu}$.

Now we show $(i)$ -- $(vi)$.
By abuse of notation, we will write $T$, instead of $\widetilde{T}$.
Since $ \langle \Tchi{E_1}, \chi_G \rangle = \langle \chi_{E_1}, \TTchi{G}  \rangle   \ge \frac{1}{2}\beta
\abs{E_{1}}$, we can define a set
\begin{align}\label{eq:j=1}
E_1^1 = \{ y \in E_1 : \TTchi G (y) \ge {\beta}/4 \}.
\end{align}
Let us set
\begin{align*}
G^1 = \{ x\in G : \Tchi{E_1^1} (x) \ge {\alpha_1}/8 \},
\end{align*}
considering that
\begin{align*}
\langle \Tchi{E_1^1}, \chi_G \rangle
= \langle \Tchi{E_1}, \chi_G \rangle - \langle \chi_{E_1 \setminus E_1^1},  \TTchi{G} \rangle \ge \frac12 \beta |E_1| - \frac14 \beta |E_1|
=\frac 14 \beta |E_1| = \frac14 \alpha_1 |G|.
\end{align*}

%It is easy to show $G^1$ is not empty. (See the proof of Lemma 1 in \cite{DLW}.) In fact,
%\begin{align*}
%\langle \chi_{E_1^1}, \TTchi {G^1}\rangle
%& = \langle \chi_{E_1^1}, \TTchi{G}  \rangle - \langle \chi_{E_1^1} , \TTchi{G\setminus G^1}\rangle \\
%& = \langle \Tchi{E_1},\chi_G\rangle - \langle \Tchi{E_1\setminus E_1^1},\chi_G \rangle - \langle \chi_{G\setminus G^1}, \Tchi{E_1}\rangle \\
%& \ge \frac{1}{2}\alpha_1 |G| - \frac14 \beta |E_1| - \frac18 \alpha_1 |G|
%= \frac 18  \alpha_1 |G|.
%\end{align*}
%
Continuing this procedure, we can find sequences of sets $E_1^j$ and $G^j$ defined by
\begin{gather*}
E^0_{1} := E_{1},\quad G^0 :=G\\
E_1^j  = \{ y \in E_1^{j-1} : \TTchi{G^{j-1}}(y) \ge \beta / 2^{2j} \},\quad j=1,\dots,d \\
G^j  = \{ x \in G^{j-1} : \Tchi{E_1^j} (x) \ge \alpha_1 / 2^{2j+1} \},\quad j=1,\dots, d-1.
\end{gather*}

We should show that $E_{1}^{j}$ and $G^{j}$ are nonempty. % we will make the stronger property
It suffices to show that
\begin{align}\label{eq:j}
\langle \Tchi{E_{1}^{j}}, \chi_{G^{j}}\rangle \ge \beta \abs{E_{1}}/2^{2j+1}
\textrm{ for } j \ge 0.
%\quad \forall j \ge 0
\end{align}

%Take $j=0$. By \eqref{eq:j=0}, we obtain the following inequality,
%\begin{align}
%\langle \Tchi{E_{1}}, \chi_{G}\rangle \ge \frac{1}{2}\alpha_1 \abs{G} = \frac{1}{2}\beta \abs{E_1}.
%\end{align}

The case $j=0$ is clear by \eqref{eq:j=0}.
%Using the induction, we get the following step:
Let us assume \eqref{eq:j} for some $j$. Then it follows that
\begin{align*}
\langle \Tchi{E_{1}^{j+1}}, \chi_{G^{j+1}} \rangle &= \langle \Tchi{E_{1}^{j+1}}, \chi_{G^{j}} \rangle - \langle \Tchi{E_{1}^{j+1}}, \chi_{G^{j} \setminus
G^{j+1}} \rangle \\
&\ge \langle\chi_{E_{1}^{j+1}}, \TTchi{G^{j}} \rangle -  2^{-2j-3} \alpha_{1}\abs{G} \\
&\ge \langle \chi_{E_{1}^{j}}, \TTchi{G^{j}} \rangle - \langle \chi_{E_{1}^{j} \setminus E_{1}^{j+1}}, \TTchi{G^{j}} \rangle - 2^{-2j-3} \alpha_{1}\abs{G} \\
&\ge \langle \chi_{E_{1}^{j}}, \TTchi{G^{j}} \rangle -  2^{-2j-2} \alpha_{1}\abs{G}-  2^{-2j-3} \alpha_{1}\abs{G} \\
&\ge  2^{-2j-1} \alpha_{1}\abs{G} -  2^{-2j-2} \alpha_{1}\abs{G}-  2^{-2j-3} \alpha_{1}\abs{G} = 2^{-2(j+1)-1} \beta\abs{E_{1}}.
\end{align*}
This gives the claim \eqref{eq:j} by induction.

Now we define
\begin{align}\label{def_H}
H_k(z_1,\dots,z_k) := \sum_{j=1}^{k} (-1)^{j+1} h(z_j), \quad k\ge 1.
\end{align}

Fix $y_0 \in E_1^{d}\subset E_1$ and set
\[
P_1 = \{ z_1 \in \Delta\setminus B_\gamma : y_0 + H_1(z_1) \in G^{d-1}\}, \textrm{ then }
\sigma(P_1) = \TTchi{G^{d-1}}(y_0) \ge \beta /2^{2d}.
\]
For all $z_1 \in P_1$, we also set
\[
P_2 = \{ z_2 \in \Delta\setminus B_\gamma : y_0 + H_2(z_1,z_2) \in E_1^{d-1} \}, \textrm{ then } \sigma(P_{2}) = \Tchi{E_1^{d-1}}(y_0 + H_1(z_1)) \ge
\alpha_1/2^{2d-1}.
\]
Through the iterative process, we obtain that for $k=2,3,\dots,d-1$,
\begin{gather*}
 P_{2k-1} = \{ z_{2k-1} \in \Delta\setminus B_\gamma : y_0 + H_{2k-1}(z_1,\dots,z_{2k-1}) \in G^{d-k} \}, \\
 P_{2k} = \{ z_{2k} \in \Delta\setminus B_\gamma : y_0 + H_{2k}(z_1,\dots, z_{2k}) \in E_1^{d-k} \}.
\end{gather*}
Finally we set
\begin{gather*}
P_{2d-1} = \{ z_{2d-1} \in \Delta\setminus B_\gamma : y_0 + H_{2d-1}(z_1,\dots,z_{2d-1}) \in G \},\\
P_{2d} = \{ z_{2d} \in \Delta \setminus B_{\alpha_2}: y_0 + H_{2d}(z_1,\dots,z_{2d}) \in E_2 \}
\end{gather*}
where $B_{\alpha_2} = \{ z \in \Delta : |z| \le (4\pi\nu)^{-\nu}\alpha_2^\nu \}$.
Then it follows that for $k=2,3,\dots, d-1$
\begin{gather*}
\sigma(P_{2k-1}) = \TTchi{G^{d-k}}(y_0+H_{2k-2}(z_1,\dots,z_{2k-2})) \ge \beta / 2^{2(d-k+1)}\\
\sigma(P_{2k}) = \Tchi{E_1^{d-k}}(y_0+H_{2k-1}(z_1,\dots,z_{2k-1})) \ge \alpha_1  / 2^{2(d-k)+1}
\end{gather*}
provided $z_j \in P_j$.
We also have
\begin{gather*}
\sigma(P_{2d-1}) = \TTchi{G}(y_0 + H_{2d-2}(z_1,\dots,z_{2d-2})) \ge \beta/2^2 ,
\end{gather*}
and
\begin{align*}
\sigma(P_{2d} ) &= \int_{\Delta\setminus B_{\alpha_2}} \chi_{E_2}(y_0 + H_{2d}(z_1,\dots,z_{2d}) ) d\sigma(z_{2d}) \\
& \ge \int_\Delta \chi_{E_2}(y_0 + H_{2d}(z_1,\dots,z_{2d}) ) d\sigma(z_{2d}) - \sigma(B_{\alpha_2}) \ge \frac 12 \alpha_2
\end{align*}
since $y_0 + H_{2d-2}\in E_1^1$ and $y_0 + H_{2d-1} \in G$.
%By this construction we can demonstrate the following lemma, which is similar to Lemma 5 in \cite{S2}.
Hence we get a sequence of sets $P_j$, $1\le j \le 2d$, in which $\abs{z_{i}} \gtrsim \gamma^{\nu}$, $1 \le i  < 2d$ and $\abs{z_{2d}} \gtrsim \alpha_2^\nu$.
Thus we obtain $(i)$--$(iii)$.

To prove $(iv)$--$(vi)$, we will consider subsets of $P_j$'s which maintain the properties $(i)$--$(iii)$.
%\chg{Also we see that $(i)$--$(iii)$ are valid with such $P_j$'s.}
%\chg{So we proved the properties $(i)-(iii)$.}
%Let $\bar{P_{i}}=\{z_i \in P_{i} : \abs{z_{i}} < c\beta^{\nu} \}$ when $i$ is odd.
%$\sigma(\bar{P_{i}})=\int_{P_{i}} \abs{z}^{\frac{4k}{d(d+1)}} d \mu(z) \le 2\pi c^{\nu} \beta$ where $c$ only depends $k, d$. So if $c$ is sufficiently small then $\sigma(P_{i}\backslash \bar{P_{i}}) \ge\frac{1}{2} \beta$ when $i$ is odd.
%Therefore we can replace $P_{i}$ by $P_{i}\backslash \bar{P_{i}}$.
%Now we turn to proving properties $(iv)$--$(vi)$.
%(Why are there no mention of the other cases above?? Need to mention them explicitly.)
%
Let $B_{\beta}(z_j) = \{z \in \Delta : \abs{z-z_{j}} \le c_{0} \beta^{1/2}\abs{z_{j}}^{\frac{-2K}{d(d+1)}} \}$ for a sufficiently small constant $c_0 > 0$.
Suppose $i$ is odd and $j < i < 2d$.
If $\abs{z} \le c \abs{z_{j}}$ for all $z \in P_{i}$ and a constant $c>0$, then
% and a constant $c'$ then
\begin{align*}
\sigma (P_{i} \cap B_{\beta}(z_j) ) =\int_{P_{i}\cap B_{\beta}(z_j)} \abs{z}^{\frac{4K}{d(d+1)}}d \mu (z)
\le (c \abs{z_{j}})^{\frac{4K}{d(d+1)}} \mu( B_{\beta}(z_j) )
%&\le c_1 \abs{z_{j}}^{4k/(d(d+1))} c_{0}^2 \beta \abs{z_j}^{-4k/(d(d+1))} \\
%&=c_1 c_{0}^{2}\beta
 \le c_1 \beta
\end{align*}
where $c_1 = 2 \pi c^{\frac{4K}{d(d+1)}} c_0^2$. %depends $k$ and $d$.
Since $\sigma(P_i) \ge \beta/2^{2d-i+1}$ when $i$ is odd, we have $\sigma(P_{i} \backslash B_{\beta}(z_j) ) \ge \beta/2^{2d-i+1} - c_1 \beta \ge \beta/2^{2d-i+2}$ if
we choose sufficiently small $c_{0}>0$. %Let us denote $P_i \setminus B_\beta(z_j)$ by $P_i$ again.
(Thus $(i)$ still holds for $P_i \setminus B_\beta(z_j)$.)
If $\abs{z} \le c \abs{z_{j}}$ is not valid for all $z \in P_{i}$, we use the fact that $\abs{z_{j}} \ge (4\pi\nu)^{-\nu} \gamma^{\nu}$ for $z_j \in P_j$.
It follows that $(4\pi\nu)^{\frac12} \abs{z_{j}} \ge \gamma^{\frac12}\abs{z_{j}}^{\frac{-2K}{d(d+1)}} \ge \beta^{\frac12}\abs{z_{j}}^{\frac{-2K}{d(d+1)}} $,
and then $\abs{z- z_{j}} < c_{0} \beta^{\frac12} \abs{z_{j}}^{\frac{-2K}{d(d+1)}}< c_{0} (4\pi\nu)^{\frac12} \abs{z_{j}}$ on $B_{\beta}(z_j)$.
In this case, we also obtain $\sigma(P_{i}\setminus B_{\beta}(z_j) ) \ge \beta/ 2^{2d-i+2}$, since $|z| \le (1+c_{0} (4\pi\nu)^{1/2})|z_j|$.
Thus we conclude that $(iv)$ holds in any case.

When $1 < i < 2d$ is even, we can also obtain $(v)$ by replacing $\beta$ with $\alpha_1$.

To show $(vi)$, we consider two cases for each $j < 2d$.
If $| z_j| < \frac12 (4 \pi\nu)^{-\nu}\alpha_2^\nu$, then $|z_{2d} - z_j| > \frac 12 |z_{2d}| \ge \frac12
(4\pi\nu)^{-1/2}\alpha_2^{1/2}|z_{2d}|^{-2K/[d(d+1)]}$, since $|z_{2d}| \ge (4\pi\nu)^{-\nu} \alpha_2^\nu$ and $\nu = d(d+1)/[4K +2d(d+1)]$.
If $| z_j| \ge \frac12 (4 \pi\nu)^{-\nu}\alpha_2^\nu$, i.e. $\frac12 (4\pi\nu)^{-1/2}\alpha_2^{1/2} |z_j|^{-2K/[d(d+1)]} \le |z_j|$, it follows that for $z
\in B_{\alpha_2}(z_j)$, the inequalities $|z- z_j| < c_0 \alpha_2^{1/2} |z_j|^{-2K/[d(d+1)]} < 2 c_0 (4\pi\nu)^{1/2}|z_j|$ imply that $ |z| < (1+2c_0 (4\pi\nu)^{1/2})|z_j|$ for
$z \in B_{\alpha_2}(z_j)$. By this we have that
\begin{align*}
\sigma(B_{\alpha_2}(z_j))
&= \int_{B_{\alpha_2}(z_j)} |z|^{\frac{4K}{d(d+1)}} d\mu(z)\\
& \le (1+2 c_0 (4\pi\nu)^{\frac12})^{\frac{4K}{d(d+1)}} |z_j|^{\frac{4k}{d(d+1)}} \times \mu(B_{\alpha_2}(z_j)) \\
& \le (1+2 c_0 (4\pi\nu)^{\frac12})^{\frac{4K}{d(d+1)}} |z_j|^{\frac{4K}{d(d+1)}} \times  \pi c_0^2 \alpha_2 |z_j|^{-\frac{4K}{d(d+1)}} \\
& = c_1 \alpha_2
\end{align*}
where $c_1 = \pi c_0^2 ( 1+2 c_0 (4\pi\nu)^{\frac12})^{\frac{4K}{d(d+1)}}$.
Choosing sufficiently small $c_0$, we can get $\sigma(P_{2d} \setminus B_{\alpha_2}(z_j)) > \alpha_2/4$.
Once again, $(iii)$ holds for $P_{2d} \setminus B_{\alpha_2}(z_j)$ in place of $P_{2d}$.
Hence $|z_{2d} - z_j| \ge c_0 \alpha_2^{\frac12} |z_j|^{-\frac{2K}{d(d+1)}}$ when $|z_j| \gtrsim \alpha_2^\nu$.
This completes the proof.
\end{proof}

To prove Lemma \ref{d-trilinear}, we consider two cases $\beta \gtrsim \alpha_{1}$ and $\beta \ll \alpha_1$.
First let us assume  $\beta \gtrsim \alpha_{1}$.
Let $\Phi_{2d}=\{(z_{1}, z_{2}, ..., z_{2d}) : z_{i} \in P_{i}\, \text{ where }\, 1 \le i \le 2d \}$.
We define $z_0=(z_{1}, z_{2}, ... ,z_{d}) \in \Phi_{d}$ and $\Phi := \{z \in \mathbb{C}^{d} : (z_{0},z)\in \Phi_{2d}\}$.
Note that $\sigma(\Phi) \sim \alpha_1^{d/ 2 - 1} \beta^{ d/ 2} \alpha_2$ when $d$ is even, and $\sigma(\Phi) \sim \alpha_1^{(d-1)/ 2} \beta^{(d-1)/2} \alpha_2$ when $d$ is odd.
For $d \ge 2$, let
\begin{align*}
a_{d}=
\begin{cases}
1 + 3 + \cdots + (d-1) &\text{if $d$ is even}, \\
2+ 4 + \cdots + (d-1) &\text{if $d$ is odd}.
\end{cases}
\end{align*}

The following lemma gives Lemma \ref{d-trilinear} whenever $\beta \gtrsim \alpha_1$ and $a_{d} \ge (d-1)$.
\begin{lemma}\label{b>a}
Let $d \ge 2$. Assume the hypotheses in Lemma \ref{d-trilinear}. Then
\begin{align*}
\abs{E_{2}} \gtrsim \alpha_{1}^{\frac{d(d+1)}2}\left(\frac{\beta }{\alpha_{1}}\right)^{a_{d}} \left( \frac{\alpha_{2}}{\alpha_{1}} \right)^{d}.
\end{align*}
\end{lemma}

%Lemma \ref{d-trilinear} is deduced from $\beta \gtrsim \alpha_{1}$ and

\begin{proof}
We assume that $d$ is even. The proof is similar when $d$ is odd.
Let $H(z):= y_{0} + H_{2d}(z_{0},z)$. By Lemma \ref{jaco_mono}, we have
\begin{align*}
\abs{E_{2}} &\gtrsim \int_{\Phi}\abs{ { J_{\mathbb R}   H (z) }}d\mu (z)= \int_{\Phi}\abs{J_{\mathbb{C}} H (z)}^{2} d\mu(z)\\
&\gtrsim \int_{\Phi}\max_{d+1 \le i \le 2d}\abs{z_{i}}^{2K}\prod_{d+1 \le i <j \le 2d}\abs{z_{j}-z_{i}}^{2}d\mu(z),
\end{align*}
where $J_{\mathbb R} H $ is the real Jacobian of $H$ in $\mathbb R^{2d}$.

Using $(iv)$ -- $(vi)$ in Lemma \ref{setup}, we obtain that
\begin{align*}
\abs{E_{2}} &\gtrsim \alpha_{1}^{1+3+\cdots+(d-3)}\beta^{2+4+\cdots+(d-2)} \alpha_{2}^{d-1} \times \\
&\qquad\times\int_{\Phi} \max_{d+1 \le i \le 2d}\abs{z_{i}}^{2K} \prod_{i=d+1}^{2d-2}\abs{z_{i}}^{\frac{-4K}{d(d+1)}(2d-i-1)} %\times\\
\prod_{j=d+1}^{2d-1} \abs{z_{j}}^{\frac{-4K}{d(d+1)}\varepsilon_{j}}\abs{z_{2d}}^{\frac{-4K}{d(d+1)}(1-\varepsilon_{j})} d \mu(z), \\
& = \alpha_{1}^{\frac{d(d-1)}{2}} \left(\frac{\beta}{\alpha_1} \right)^{2+4+\cdots+(d-2)} \left( \frac{\alpha_{2}}{\alpha_1} \right)^{d-1}  \times \\
&\qquad\times \int_{\Phi} \max_{d+1 \le i \le 2d}\abs{z_{i}}^{2K} \prod_{i=d+1}^{2d-1}\abs{z_{i}}^{\frac{-4K}{d(d+1)}(2d-i-1 + \varepsilon_i)} %\times\\
 \abs{z_{2d}}^{\frac{-4K}{d(d+1)}\sum_{i=d+1}^{2d-1}(1-\epsilon_{j})}d \mu(z),
\end{align*}
where $\varepsilon_{j}=0$ if $|z_j| <\frac12 (4\pi \nu)^{-\nu} \alpha_2^{\nu}$, and $\varepsilon_{j}=1$ if $|z_j |\ge \frac12 (4\pi \nu)^{-\nu} \alpha_2^{\nu}$, for $d+1 \le j \le 2d-1$.
%(See (3) in Lemma \ref{setup}.) %, when $d$ is even.
It is clear that $\max_{d+1 \le i \le 2d}\abs{z_{i}}^{2K} \ge \prod_{i=d+1}^{2d}\abs{z_{i}}^{2K_{i}}$ when $\sum_{i=d+1}^{2d}K_{i}=K$. We can choose
$K_{i}\ge0$ such that
\begin{align*}
K_{i}=\frac{2K}{d(d+1)}(1+(2d-i-1)+\epsilon_{i})
\end{align*}
where $d+1 \le i \le 2d-1$, and
\begin{align*}
K_{2d}=\frac{2K}{d(d+1)}(1+\sum_{i=d+1}^{2d-1}(1-\epsilon_{i})),
\end{align*}
where $\varepsilon_{i}=0$ or $1$ for $d+1 \le i \le 2d-1$.

Then we get that
\begin{align*}
\abs{E_{2}} &\gtrsim \alpha_{1}^{\frac{d(d-1)}{2}} \left(\frac{\beta}{\alpha_1} \right)^{2+4+\cdots+(d-2)} \left( \frac{\alpha_{2}}{\alpha_1} \right)^{d-1}
\int_\Phi \prod_{i=d+1}^{2d} |z_i|^{\frac{4K}{d(d+1)}} d\mu(z) \\
& = \alpha_{1}^{\frac{d(d-1)}{2}} \left(\frac{\beta}{\alpha_1} \right)^{2+4+\cdots+(d-2)} \left( \frac{\alpha_{2}}{\alpha_1} \right)^{d-1}  \sigma (\Phi)\\
%\alpha_{1}^{1+3+\cdots+d-3}\beta^{2+4+\cdots+d-2} \alpha_{2}^{d-1} \beta^{\frac{d}{2}}\alpha_{1}^{\frac{d-2}{2}} \alpha_{2}\\
%&\sim \alpha_{2}^{d} \left(\frac{\beta}{\alpha_{1}}\right)^{1+3+\cdots +d-1}\alpha_{1}^{\frac{d(d-1)}{2}} \\
& \sim \alpha_{1}^{\frac{d(d-1)}{2}} \left(\frac{\beta}{\alpha_1} \right)^{2+4+\cdots+(d-2)} \left( \frac{\alpha_{2}}{\alpha_1} \right)^{d-1} \alpha_1^{ \frac d 2
- 1} \beta^{ \frac d 2} \alpha_2\\
&\sim \alpha_{1}^{\frac{d(d+1)}{2}} \left(\frac{\beta}{\alpha_{1}}\right)^{a_{d}} \left(\frac{\alpha_{2}}{\alpha_{1}}\right)^{d} ,
\end{align*}
when $d$ is even.
%\chg{Because of $ \beta \gtrsim \alpha_{1}$}, we obtain that \chg{$\left(\frac{\beta}{\alpha_1}\right)^{a_{d}} \ge\left(\frac{\beta}{\alpha_1}\right)^{d-1}$}. So this completes the proof of Lemma \ref{d-trilinear} in the case \chg{$\beta \gtrsim \alpha_{1}$}.
\end{proof}

%Let $\bar{P_{2d}}=\{z_{2d} \in P_{2d}: \abs{z_{2d}}<c \alpha_{2}^{\nu}\}$. Then $\sigma(\bar{P_{2d}}) \le 2\pi c^{\nu} \alpha_{2}$. Likewise if $c$ is sufficiently small then $\sigma(P_{2d} \setminus \bar{P_{2d}}) \ge \frac{1}{2}\alpha_{2}$. We will replace $P_{2d}$ by $P_{2d} \setminus \bar{P_{2d}}.$
%If $\abs{z_{j}}< \epsilon \alpha^{\nu}$ and $\abs{z_{2d}}> c \alpha_{2}^{\nu}$, then $\abs{z_{2d}-z_{j}} \ge \frac{c}{2}\abs{z_{2d}}$ when $\epsilon$ is sufficiently small.

%When the case $\abs{z_{j}}> \epsilon \alpha_{2}^{\nu}$ occurs, then we define $B(z_{j})=\{z \in B: \abs{z-z_{j}} < c_{0}\alpha_{2}^{1/2}\abs{z_{j}}^{\frac{-2k}{d(d+1)}}\}$. If $z \in B(z_{j})$, then $\abs{z -z_{j}}< c_{0}\alpha_{2}^{1/2} \abs{z_{j}}^{\frac{-2k}{d(d+1)}} <c_{0}\alpha_{2}^{\nu} < \abs{z_{j}}$. Since $\abs{z} < 2\abs{z_{j}}$,
%\begin{align*}
%\sigma(B(z_{j}))&=\int_{B(z_{j})}\abs{z}^{\frac{4k}{d(d+1)}}d\mu(z) \\
%& \le \abs{z_{j}}^{\frac{4k}{d(d+1)}}\int_{B(z_{j})}d\mu(z) \\
%& = 2\pi c_{0}^{2}\alpha_{2 }\abs{z_{j}}^{\frac{4k}{d(d+1)}}\abs{z_{j}} ^{\frac{-4k}{d(d+1)}} \\
%& \le c \alpha_{2}
%\end{align*}

%\end{proof}

\subsection{The band structure argument}
To handle the case $\beta \ll \alpha_1$, we modify the original band structure argument due to Christ \cite{Ch1}. We also refer to the works by Stovall
\cite{S2} and Gressman \cite{G}, which treat the degenerate cases.

We begin with recalling some definitions to describe the band structure argument. We shall decompose an index set into subsets called bands. In each band,
the even index or 1 is a `free index'. If there is no even index or 1, the smallest index is the free index. If a band has only two indices, the other is a
`quasi-free' index, and we say that the quasi-free index is quasi-bound to the free index. If a band has more than three indices, the other indices other than the
free index are bound to the free index.

The following lemma is a variant of the real case considered in \cite{S2}. %By this lemma we can choose appropriate $k$ variables which

\begin{lemma}\label{band}
Let $\varepsilon >0$.
There exist parameters $\delta, \delta'$ satisfying $0 < c_{d,\varepsilon} < \delta' <\varepsilon\delta $, a constant $\tilde c$, an integer $d\le k < 2d$, a
set $\omega \subset \mathbb C^{k}$ with $\sigma(\omega) \gtrsim \alpha_1^{ \lceil k/2 \rceil} \beta^{ \lfloor k/2 \rfloor} (\alpha_2/ \alpha_1)$, and a band
structure on $\{2d-k+1,\dots, 2d\}$, such that the following properties hold:
\begin{enumerate}[(i)]
\item There are exactly $d$ free or quasi-free indices. In particular, each even index is free.
\item $|z_i - z_j | >  \delta \alpha_1^{\frac12} |z_i z_j|^{-\frac{K}{d(d+1)}}$, unless $i$ and $j$ lie in the same band.
\item ${\tilde c} \beta^{\frac12} |z_i z_j|^{-\frac{K}{d(d+1)}} < |z_i - z_j| \leq \delta \alpha_1^{\frac12} |z_i z_j|^{-\frac{K}{d(d+1)}}$ whenever $i$ is quasi-bound to
$j$.
\item $\delta' \alpha_1^{\frac12} |z_i z_j|^{-\frac{K}{d(d+1)}} > |z_i - z_j| $ whenever $i$ is bound to $j$.
\end{enumerate}
\end{lemma}

\begin{proof}
It is not enough to arrange the absolute values of $z_1,\dots,z_{2d}$ in order, since two variables with the same size can be separated. Thus our
approach is slightly different at the beginning.
We start by observing that $z_j$'s for even $j$ are separated from each other.
In other words, all the balls centered at $z_j$ for even $j$ can be made mutually disjoint by choosing the radii appropriately.

{\sl Step 1.}
By Lemma \ref{setup}, there exist constants $c$ such that
\begin{align}\label{separate}
|z_i - z_j| > c  \alpha_1^{\frac12}  |z_i z_j|^{-\frac{K}{d(d+1)}}
\end{align}
for all even indices $j\le 2d$ and  for all $i < j$.
%\chg{In fact, let us assume that $|z_j| \sim |z_i|$. Then \eqref{separate} follows by $(v)$ and $(vi)$ in Lemma \ref{setup}.
%When $|z_j | \ll |z_i|$ or $|z_j |\gg |z_i|$, we see that $|z_j - z_i| \gtrsim |z_j|$. Since $|z_j| \gtrsim \alpha_1^\nu$ for all even $j\le 2d-2$ and $|z_{2d} | \gtrsim \alpha_2^\nu \ge \alpha_1^\nu$,
%we obtain that $|z_j|\gtrsim \alpha_1^{\frac12}|z_j|^{-\frac{2K}{d(d+1)}}$ for all even $j$, which implies \eqref{separate}.}
%
In fact, we have that $|z_i - z_j| \gtrsim \alpha_1^{\frac12}   |z_i|^{-\frac{2K}{d(d+1)}}  $ for even $j \le 2d-2$ and $i < j $ by $(v)$ in Lemma \ref{setup}.
If $|z_i| \lesssim |z_j|$,
it follows that $|z_i - z_j| \gtrsim \alpha_1^{\frac12}  |z_i z_j|^{-\frac{K}{d(d+1)}}  $.
If $|z_i | \gg |z_j| $, then we get $|z_i - z_j| \gtrsim |z_i| \gg |z_j| $.
Since $|z_j| \ge (4\pi\nu)^{-\nu}\gamma^\nu$ by Lemma \ref{setup} and $\nu = \frac{d(d+1)}{4K + 2d(d+1)}$,
one can see that $|z_i - z_j| \gtrsim |z_j| \gtrsim \alpha^{\frac12}|z_j|^{-\frac{2K}{d(d+1)}} \gg \alpha^{\frac12} |z_i z_j|^{-\frac{K}{d(d+1)}} $.

Next, let us consider the case of $j=2d$. If $|z_{2d} | \sim |z_i|$ for any $i < 2d$, then $(vi)$ in Lemma \ref{setup} gives $|z_{2d} - z_i| \gtrsim \alpha_1^{\frac12}  |z_i z_{2d}|^{-\frac{K}{d(d+1)}}$. (Note that we are assuming here that $\alpha_1 \le \alpha_2$.)
For the case $|z_{2d}| \ll |z_i|$ or $|z_{2d}| \gg |z_i|$, we see that
$|z_{2d} - z_i| \gtrsim |z_{2d}| = |z_{2d}|^{1 +\frac{K}{d(d+1)}} |z_i|^{\frac{K}{d(d+1)}} |z_{i} z_{2d}|^{-\frac{K}{d(d+1)}}$.
Since $|z_{2d}|\gtrsim \alpha_2^{\nu}\ge \alpha_1^\nu$ and $|z_i| \gtrsim \gamma^{\nu}\ge \alpha_1^\nu$ for $i < 2d$, it follows that $$|z_{2d} - z_i| \gtrsim \alpha_1^{\nu(1 + \frac{2K}{d(d+1)})}  |z_i z_{2d}|^{-\frac{K}{d(d+1)}} = \alpha_1^{\frac12} |z_i z_{2d}|^{-\frac{K}{d(d+1)}} .$$
Therefore, \eqref{separate} is valid for any even $j$ and $i < j$.

Similarly, there exists a constant $\tilde c>0$ such that
\begin{equation}\label{separate-beta}
|z_i - z_j| > \tilde c \beta_1^{\frac12}  |z_i z_j|^{-\frac{K}{d(d+1)}}
\end{equation}
for any odd $j$ and $i < j$.

%For each even $j$, we can take $\min(|z_i||z_j|)^{-\frac{K}{d(d+1)}}$ for all even $i \neq j$, which we denote by $(|z_{j_\circ}||z_j|)^{-\frac{K}{d(d+1)}}$.}

Let us define a ball centered at $z_j$ with small $\delta < c$, to be chosen later, by setting
\begin{align*}
B_{\delta, \alpha_1}(z_j) = \{ z : |z-z_j| \le \delta  \alpha_1^{\frac12}   |z\, z_j|^{-\frac{K}{d(d+1)}}  \}.
\end{align*}

%\chg{Then it is clear that for each even $j$, the $B_{\delta, \alpha_1}(z_j)$ are pairwise disjoint.
%(Indeed $B_{d\delta,\alpha_1}(z_j)$ are pairwise disjoint.)}
% for different choices of $\delta$ \chg{as above!!!!!!!!!!!!!!!!!!!}. ??

%by the choice of $\delta$. ?? (as above?)

For each even index $j$, let $b(j)$ be a subset of $\{1,2,\dots,2d\}$ such that
\begin{itemize}
\item $j \in b(j)$,
\item $j_1 \in b(j)$ if $z_{j_1} \in B_{\delta, \alpha_1}(z_{j})$,
\item $j_{k+1} \in b(j)$ if $z_{j_{k+1}} \in B_{\delta,\alpha_1}(z_{j_k})$ for some $j_k\in b(j)$.
%\item $j_{k+1} \in b(j)$ if $z_{j_{k+1}} \in \{ z: |z- z_{j_k}| \le \delta \alpha_1^{\frac12} \chg{(|z_{j_\circ}||z_j|)^{-\frac{K}{d(d+1)}} } \}$ for some $j_k\in b(j)$.
\end{itemize}

If there is no $j_1$ such that $z_{j_1} \in B_{\delta, \alpha_1}(z_{j})$, then we set $b(j) = \{ j\}$.

First we show how to construct $b(j)$ for each even $j$.
By \eqref{separate}, $b(2d)= \{ 2d \}$ holds.
In fact, we have that
\[
|z_i - z_{2d} | \ge c \alpha_1^{\frac12} |z_i z_{2d}|^{-\frac{K}{d(d+1)}}
\]
for all $i < 2d$.

For $z_{2d-2}$,  \eqref{separate} holds for all odd and even $i < j=2d-2$.
On the other hand, $z_{2d-1}$ may be contained in $B_{\delta,\alpha_1}(z_{2d-2})$.
In general, each odd $i$ may be contained in either exactly one $B_{\delta,\alpha_1}(z_{j})$ satisfying $i > j$ or none of them.
Then each $B_{\delta,\alpha_1}(z_j)$ may contain odd indices greater than $j$.
Those odd indices belong to $b(j)$.
If $B_{\delta,\alpha_1}(z_j)$ has no odd index, then let $b(j) = \{j\}$.

Let $j_{1}$ be one of the odd indices contained in $b(j)$,
%such that $z_{j_1} \in B_{\delta,\alpha_1}(z_j)$ (i.e. $j_1 \in b(j)$),
and consider $B_{\delta,\alpha_1}(z_{j_1})$.
If there  exists  an odd index $\ell$ such that $z_\ell \in B_{\delta,\alpha_1}(z_{j_1})$, then we denote it by $j_{2}$ (of course, $j_{2} \in b(j)$).
For $z_{j_2}$ we consider $B_{\delta,\alpha_1}(z_{j_2})$ and repeat the process as above.
Hence, each $b(j)$ consists of a unique even index $j$ and some odd indices greater than $j$.

Now we consider indices which belong to none of the $b(j)$ for even $j$.
In this case, we choose the smallest index among the remaining indices.
By \eqref{separate} this index must be $1$.
 In fact, $B_{c/3,\alpha_1}(z_1)$ is disjoint from the other $B_{c/3,\alpha_1}(z_j)$ for even $j$.
By choosing sufficiently small $\delta$, it is valid that $B_{\delta, \alpha_1}(z_i) \subset B_{c/3,\alpha_1}(z_j)$ for all $i \in b(j)$.
Then we can construct $b(1)$ in the same manner as above.

If there are still remaining indices, we choose the smallest one and repeat the procedure.
By this we can construct  bands for odd indices.

%If there is an index $m$ which does not belong to any $b(j)$, $b(1)$ or $b(l)$'s, then one can see that $B_{\delta, \alpha_1}(z_m)$ does not contain variables except for $z_m$. In this case, $\{ m\} =:b(m)$ is a band.

Hence the index set $\{1,\dots,2d \}$ can be decomposed into the bands $b(j)$'s.
Here $j$ represents the free index in $b(j)$.
Each $b(j)$ has less than $d$ elements because of \eqref{separate}.
(The case where $b(j)$ has $d$ elements can only occur when $b(1)$ contains all odd indices.)
Also the free index is the smallest in the band.
%, which we denote by $b(j)$ for even $j$ and $b(l)$ for odd $l$.

Now we can check that the  properties $ (ii)$ and $(iii)$ hold.
Let us assume that $j' \in b(j)$ and $k' \in b(k)$ for $j\neq k$,  i.e. $j'$ and $k'$ are in different bands.
From the construction of $b(j)$ and $b(k)$,
it follows that
%it means that
$z_{j'} \notin B_{\delta,\alpha_1}(z_{k'})$ and $z_{k'} \notin B_{\delta,\alpha_1}(z_{j'})$.
This immediately  implies $(ii)$ for $j'$ and $k'$.

If $j'$ is quasi-bound to $j$, i.e. $b(j) = \{j, j'\}$, then $j'$ must be an odd number.
Hence $(iii)$ follows from  the construction of $b(j)$ and \eqref{separate-beta}.

{\sl Step 2.}
First, we need to verify that $|z_{j_k}| \sim |z_{j_{k+1}}|$ whenever $j_k, j_{k+1} \in b(j)$.
Since $z_{j_{k+1}} \in B_{\delta,\alpha_1}(z_{j_k})$ and $|z_j| \ge (4\pi\nu)^{-\nu} \alpha_1^{\nu}$ for all $j$, it follows that
\begin{align*}
|z_{j_k} | &\ge |z_{j_{k+1}}| - |z_{j_k} - z_{j_{k+1}}|
\ge  |z_{j_{k+1}}| - \delta \alpha_1^{\frac12}  |z_{j_k} z_{j_k+1}|^{-\frac{K}{d(d+1)}} \\
&\ge  |z_{j_{k+1}}| - c  \alpha_1^{\frac12} (\alpha_1^{2\nu}) ^{-\frac{K}{d(d+1)}}
= |z_{j_{k+1}}| - c  \alpha_1^{\nu},
\end{align*}
where $c = \delta(4\pi\nu)^{2 K\nu /d(d+1)}$.
Thus we have that $|z_{j_{k+1}}| \le c \alpha_1^\nu + |z_{j_k}| \le (\delta   (4\pi\nu)^{1/2}  +1)|z_{j_k}|$.
If we exchange $z_{j_k}$ and $z_{j_{k+1}}$, then it also holds that
$|z_{j_{k}}|\le (\delta   (4\pi\nu)^{1/2}  +1)|z_{j_{k+1}}|$.
Hence we have that for any $j_k, j_l \in b(j)$
\begin{align}\label{z_jk}
|z_{j_k}| \le (1 + \bar c) |z_{j_{k+1}}| \le (1 +\bar c)^2 |z_{j_{k+2}}| \le \cdots \le (1 +\bar c)^d |z_{j_l}|,
\end{align}
where $\bar c = \delta   (4\pi\nu)^{1/2}$.
(Note that each $b(j)$ has at most $d$ elements.)
By exchanging $z_{j_k}$ for $z_{j_l}$, we finally get $(1+\bar c)^{-d}|z_{j_l}| \le |z_{j_k}| \le (1+\bar c)^d |z_{j_l}|$.
%Let us assume that $j_k,j_l$ are in the same band $b(j)$. Then we have
%\begin{align}
%|z_{j_k}| & = |z_{j_k} - z_{j_{k-1}} + z_{j_{k-1}} - z_{j_{k-2}} + \dots + z_{j_{l+1}} - z_{j_l} + z_{j_l}| \nonumber\\
%& \ge |z_{j_l}| - ( |z_{j_l} -z_{j_{l+1}}| + \dots + |z_{j_{k-1}} - z_{j_k}|) \nonumber\\
%& \ge |z_{j_l}| - d \delta \alpha_1^{\frac12} (|z_{j_\circ}||z_j|)^{-\frac{K}{d(d+1)}}.
%\end{align}
%By Lemma \ref{setup} and the assumption $\alpha_2 \gtrsim \alpha_1$, we have that $|z_j| \ge \alpha_1^{\nu}$ for all $j$.
%Then it follows that
%\begin{align*}
%|z_{j_k}| \ge |z_{j_l}| - d \delta \alpha_1^{\frac12} (\alpha_1^{\nu})^{-\frac{2K}{d(d+1)}} = |z_{j_l}| - d \delta \alpha_1^\nu,
%\end{align*}
%and then $|z_{j_l}| \le (1+d\delta) |z_{j_k}|$.
%Conversely??(You mean 'The reverse inequality'? Why?)
Therefore we obtain that %\chg{choosing sufficiently small $\delta$ provides that }
\begin{align}\label{sameband}
|z_{j_k}| \sim |z_{j_l}|\, \textrm{ whenever $j_k$ and $j_l$ are in the same band.}
\end{align}
(The implicit constant can be adjusted by choosing sufficiently small $\delta$ when we use \eqref{sameband} in the proof of Lemma \ref{slice-jaco}.)

 If $j_k$ is bound to $j$, we see that
\begin{align*}
|z_{j_k} - z_j| & \le |z_{j_k} - z_{j_{k-1}}| + \cdots + |z_{j_1} - z_j| \\
& \le \delta  \alpha_1^{\frac12} |z_{j_k} z_{j_{k-1}}|^{-\frac{K}{d(d+1)}} + \cdots + \delta  \alpha_1^{\frac12} |z_{j_1} z_{j}|^{-\frac{K}{d(d+1)}} \\
& \le \delta  \alpha_1^{\frac12} d  (1+\bar c)^{\frac{K}{d+1}}  |z_{j_k} z_j|^{-\frac{K}{d(d+1)}}
%|z_{j_{k}} - z_{j_{k-1}}| + \cdots +(d-1)\delta \alpha_1^{\frac12}|z_j|^{-\frac{2K}{d(d+1)}}
\end{align*}
by \eqref{z_jk}.
Then $(iv)$ is not guaranteed.
If $(iv)$ holds on a subset $\mathcal Z' \subset \mathcal Z := P_1\times P_2 \times \dots \times P_{2d}$ with $\sigma(\mathcal Z') \ge \sigma(\mathcal Z)/2$,
then we proceed to the next step.

Otherwise, there exist a subset $\mathcal Z'' \subset \mathcal Z$ satisfying $\sigma(\mathcal Z'') \ge \sigma(\mathcal Z)/2$, a band $b(j)$, and an index
$j_0$ such that $|z_{j_0} - z_j| \ge \delta' \alpha_1^{\frac12} |z_{j_0} z_j|^{-\frac{K}{d(d+1)}}$.
Then we replace $\delta$ by $\delta'/d$, with which we repeat {\sl Step 1} % with $\delta'/d$.
%This means that $j_0$ no longer belongs to $b(j)$.
%Then we repeat Step 1 with small $\delta'<\varepsilon \delta$ instead of $\delta$}??
%(was this process mentioned before?)
until we get
\begin{align}\label{bound}
|z_{j_k} - z_j| <
%\delta' \alpha_1^{\frac12} |z_j|^{-\frac{2K}{d(d+1)}} \sim
\delta' \alpha_1^{\frac12}  |z_{j_k} z_j|^{-\frac{K}{d(d+1)}}
\end{align}
for each bound index $j_k \in b(j)$ on a subset $\widetilde{\mathcal Z} \subset \mathcal Z$ satisfying $\sigma(\widetilde{\mathcal Z}) \ge \sigma(\mathcal
Z)/2$.

{\sl Step 3.} %The rest of the procedure is the same as those in \cite{Ch1, S1, S2, G}.
Adopting the notations in \cite{Ch1}, we denote by $\mathcal M$, $\mathcal N$ the number of free and quasi-free indices, respectively.
We have at least $d+1$ free indices, which are even indices and 1, from the previous steps.
Using a projection repeatedly, we will reduce the value $\mathcal M +\mathcal N$ to $d$, which yields $(i)$.
First, we discard the index 1 by fixing $z_1 \in P_1$ and classify $\{2,\dots,2d\}$ as free, quasi-free and bound indices.
Then the number $\mathcal M +\mathcal N$ can decrease by 1.
Of course, it can be unchanged or increased by 1 when 1 was the free index of a band with two or three elements.
After discarding indices $\{1,2,\dots,2d-k\}$ appropriately, we obtain a band structure on $\{2d-k+1,\dots,2d\}$ with $\mathcal M + \mathcal N =d$.
Note that discarding an index does not affect properties $(ii)-(iv)$.
We denote by $\omega \in \mathbb C^k$ a set of $(z_{2d-k+1},\dots,z_{2d}) \in P_{2d-k+1} \times \cdots \times P_{2d}$ such that $\sigma(\omega) \geq c
\sigma(P_{2d-k+1} \times \cdots \times P_{2d})$ is valid for some constant $c<1$. Thus we obtain $\sigma(\omega) \sim \sigma( P_{2d-k+1} \times \cdots \times
P_{2d}) \gtrsim \alpha_1^{ \lceil k/2 \rceil} \beta^{ \lfloor k/2 \rfloor} (\alpha_2/ \alpha_1)$.
\end{proof}

\subsection{The slice argument}
Let $\Lambda=\{\lambda_1,\dots,\lambda_d\} \subset \{2d-k+1, \dots, 2d\}$ be a set of all indices which are free or quasi-free.
Also we set $\tau = (\tau_1,\dots,\tau_d) \in \mathbb C^d$ with $\tau_i = z_{\lambda_i}$, $\lambda_i \in \Lambda$.
For $\lambda'_j \in \Lambda' := \{2d-k+1,\dots,2d\}\setminus\Lambda$, $1\le j\le k-d$, there is a free index $\lambda_i\in\Lambda$ to which $\lambda'_j$ is
bound.
Let $s_j = (z_{\lambda'_j} - z_{\lambda_i})   z_{\lambda_i}^{\frac{2K}{d(d+1)}}$ and $s = (s_1,\dots, s_{k-d}) \in \mathbb C^{k-d}$.
It is clear that the map $z := (z_{2d-k+1},\dots,z_{2d}) \mapsto (\tau, s)$ is a diffeomorphism, so its inverse $z(\tau,s)$ exists and is differentiable.
Set $x_0 + H_{2d} (z_1,\dots,z_{2d-k},z) =: H(z(\tau,s))$, then $H(\omega) \subset E_2$ for $\omega$ as in Lemma \ref{band}. (Recall that $H_{2d}$ is defined in \eqref{def_H}.)

For any $s \in \mathbb C^{k-d}$, we consider a slice $\omega_s = \{ \tau : z(\tau, s) \in \omega \} \subset \mathbb C^d$. By the construction of $P_{2d}$ in
the proof of Lemma \ref{setup} ($z_{2d}$ is clearly one exponent of $\tau$), $H(\omega_s)$ is contained in $E_2$ for each $s$.
By B\'ezout's theorem, for each $s \in \mathbb C^{k-d}$ there are finitely many preimages under the map $\tau\mapsto H(z(\tau,s))$. Thus we have
\begin{align}\label{E2}
|E_2| \gtrsim \int_{\omega_s} \Big|\det\left(\frac{\partial H(z(\tau,s)) }{\partial \tau}\right) (\tau,s)\Big|^2 d\mu(\tau).
\end{align}

The following lemma gives a lower bound for the integrand in \eqref{E2}.
\begin{lemma}\label{slice-jaco}
For a given $\varepsilon >0$, there exist
a band structure on $\{2d-k+1,\dots,2d\}$ and a set $\omega$ which satisfy Lemma \ref{band}.
Then for all $(\tau,s) \in \omega$, and for some $C >0$,
\begin{align}\label{slice-jaco-bound}
\Big|\det\left(\frac{\partial H(z(\tau,s)) }{\partial \tau}\right) \Big|^2 \ge C \alpha_1^{\frac{d(d-1)}{2}} \left( \frac{\beta }{\alpha_1} \right)^{\mathcal N}
\left(\frac{\alpha_2}{\alpha_1}\right)^{ d-1 } \prod_{ i=1}^{d} |\tau_i|^{\frac{4K}{d(d+1)}}
\end{align}
holds.
Here $\mathcal N$ is the number of quasi-free indices associated to $\omega$, and $\varepsilon$ and $C$ depend only on $d, \delta$ by Lemma \ref{band}.
\end{lemma}
We postpone its proof for a moment and prove Lemma \ref{d-trilinear}.

\subsubsection*{Proof of Lemma \ref{d-trilinear} }
By integrating both sides of \eqref{E2}, and by \eqref{slice-jaco-bound}, we see that
\begin{align*}
\int_{\{s : z(\tau,s)\in\omega\}}  |E_2| d\mu(s) \gtrsim \alpha_1^{\frac{d(d-1)}{2}} \left( \frac{\beta }{\alpha_1} \right)^{\mathcal N}
\left(\frac{\alpha_2}{\alpha_1}\right)^{d-1} \int_{\{s : z(\tau,s)\in\omega\}}\int_{\omega_s}d\sigma(\tau)d\mu(s).
\end{align*}
For any $s=(s_1,\dots,s_{k-d}) \in \mathbb C^{k-d}$ satisfying $z(\tau,s)\in\omega$,
\[
|s_j| = |z_{\lambda'_j} - z_{\lambda_i} ||z_{\lambda_i}|^{\frac{2K}{d(d+1)}}
\le \delta'\alpha_1^{\frac12}  |z_{\lambda'_j}  z_{\lambda_i}|^{-\frac{K}{d(d+1)}}|z_{\lambda_i}|^{\frac{2K}{d(d+1)}} \lesssim \delta' \alpha_1^{\frac12}
\]
by \eqref{bound} and \eqref{sameband}.
It follows that $\int_{\{s : z(\tau,s)\in\omega\}} d\mu(s) \lesssim \alpha_1^{k-d}$.

By reversing the change of variables $(\tau,s) \mapsto z$, we obtain that
\begin{align*}
\int_{\{s : z(\tau,s)\in\omega\}}\int_{\omega_s}d\sigma(\tau)d\mu(s) &= \int_{\{s : z(\tau,s)\in\omega\}}\int_{\omega_s} \prod_{i=1}^d
|\tau_i|^{\frac{2K}{d(d+1)}} d\mu(\tau)d\mu(s) \\
&\sim \int_\omega \prod_{j=1}^{k-d} |z_{\lambda'_j}|^{\frac{2K}{d(d+1)}}\prod_{i=1}^d |z_{\lambda_i}|^{\frac{2K}{d(d+1)}} d\mu(z)
=\sigma(\omega)
\end{align*}
since $| \det(\frac{\partial(\tau,s)}{\partial z}) |  \sim  \prod_{ j=1}^{k-d} |z_{\lambda'_j}|^{\frac{2K}{d(d+1)}} $ by \eqref{sameband}.
%where we use \eqref{sameband} for the last equality.
Thus we conclude that
\begin{align*}
|E_2| \gtrsim \alpha_1^{\frac{d(d-1)}{2}} \left( \frac{\beta }{\alpha_1} \right)^{\mathcal N} \left(\frac{\alpha_2}{\alpha_1}\right)^{d-1} \alpha_1^{d-k}
\alpha_1^{\lceil k/2 \rceil } \beta^{\lfloor k/2 \rfloor } \left(\frac{\alpha_2}{ \alpha_1}\right)
 = \alpha_1^{\frac{d(d+1)}{2}} \left( \frac{\beta_1}{\alpha_1} \right)^{\mathcal N + \lfloor k/2 \rfloor} \left( \frac{\alpha_2}{\alpha_1} \right)^d.
\end{align*}
Since $\mathcal N + \lfloor k/2 \rfloor \le d-1$ and $\beta_1 \ll \alpha_1$, this gives Lemma \ref{d-trilinear}. \qed

\subsubsection*{Proof of Lemma \ref{slice-jaco} }

Recall that $\tau_i = z_{\lambda_i}$ for $\lambda_i \in \Lambda$ and $s_j = (z_{\lambda'_j} - z_{\lambda_i})z_{\lambda_i}^{\frac{2K}{d(d+1)}} =
(z_{\lambda'_j} - \tau_i) \tau_i ^{\frac{2K}{d(d+1)}} $. Then
\begin{align*}
& H(z(\tau,s))= \\& = x_0 + H_{2d}(z_1,\dots,z_{2d-k},z(\tau,s)) \\
& = x_0 + H_{2d-k}(z_1,\dots,z_{2d-k}) + \sum_{1\le i \le d} \left(  (-1)^{\lambda_i+1} h(\tau_i) + \sum_{j \Rightarrow i} (-1)^{\lambda'_j+1} h(s_j \tau_i^{-\frac{2K}{d(d+1)}}
+ \tau_i) \right) \\
& = x_0 + H_{2d-k}(z_1,\dots,z_{2d-k}) + \sum_{1\le i \le d} \left( \theta_i h(\tau_i) + \sum_{j \Rightarrow i} (-1)^{\lambda'_j+1} ( h(s_j \tau_i^{-\frac{2K}{d(d+1)}}
+ \tau_i) - h(\tau_i) \right),
\end{align*}
where $j \Rightarrow i$ means $\lambda'_j \in \Lambda'$ is bound to $\lambda_i \in \Lambda$, and $\theta_i = (-1)^{\lambda_i+1} + \sum_{ \lambda'_j \Rightarrow i}
(-1)^{j+1}$.
Note that $\theta_i$ cannot be $0$. In fact, for each $\lambda_i$, the number of indices which is bound to $\lambda_i$ is $0$ or at least $2$.
Also, all the indices bound to each $\lambda_i$ are odd.
Thus, $\sum_{ \lambda'_j \Rightarrow i} (-1)^{j+1}$ should be at least $2$ (or $0$ if there is no
index bound to $i$). Hence $\theta_i$ cannot be $0$.

For fixed $s$, each column of $\frac{\partial H(z(\tau,s)) }{ \partial \tau}$ is given by $\theta_i h'(\tau_i) $ if there is no $j$ such that $j \Rightarrow
i$, or
\begin{align*}
&\theta_i h'(\tau_i)  + \sum_{j\Rightarrow i} (-1)^{\lambda'_j+1} \times\\ 
&\times \left( h'(s_j \tau_i^{-\frac{2K}{d(d+1)}} +\tau_i) - h'(\tau_i) - \frac{2K}{d(d+1)}
\left(\frac{s_j}{\tau_i}\right) \tau_i^{-\frac{2K}{d(d+1)}} h'(s_j \tau_i^{-\frac{2K}{d(d+1)}} +\tau_i) \right).
\end{align*}
Then by multilinearity, we have
\[
\det \left(\frac{\partial H(z(\tau,s)) }{ \partial \tau} \right)
= C \, J_d(\tau) + \textrm{error terms},
\]
where $C = \prod_{1\le i \le d} \theta_i $, and $J_d(\tau) = \det (h(\tau_1),\dots,h(\tau_d))$ is the determinant of the \emph{complex} Jacobian of the map
$(\tau_1,\dots,\tau_d) \mapsto \sum_{1\le i \le d}h(\tau_i)$.

Our claim is that the error terms can be bounded by $O(\varepsilon) \times |J_d(\tau)|$. If it is proven, we can see that $| \det (\frac{\partial
H(z(\tau,s)) }{ \partial \tau} ) | \gtrsim |J_d(\tau)|$ by choosing sufficiently small $\varepsilon$.
Recall that
\begin{align*}
|J_d(\tau)| \gtrsim \max_{1\le i \le d}|\tau_i|^K \prod_{1 \le i < j \le d} |\tau_i - \tau_j|
\end{align*}
from Lemma \ref{jaco_mono}.
We may assume that $\tau_d = z_{2d}$.
One can see that, by Lemma \ref{band} and Lemma \ref{setup},
\begin{align*}
&\prod_{1 \le i < j \le d} |\tau_i - \tau_j| \\
& \gtrsim \alpha_1^{\frac{d(d-1)}{4}} \left( \frac{\beta }{\alpha_1} \right)^{\frac{\mathcal N}{2}} \left( \frac{\alpha_2}{\alpha_1} \right)^{\frac{d-1}{2}}
\prod_{ i=1 }^{d-2} |\tau_i|^{-\frac{2K}{d(d+1)}({d-1-i})}
\prod_{j=1}^{d-1}|\tau_j|^{-\frac{2K}{d(d+1)}\varepsilon_j}|\tau_{d}|^{-\frac{2K}{d(d+1)}(1-\varepsilon_j)}\\
% &=\alpha_1^{\frac{d(d-1)}{2}} \left( \frac{\beta }{\alpha_1} \right)^{\mathcal N} \left( \frac{\alpha_2}{\alpha_1} \right)^{d-1} \prod_{1\le i \le d} |\tau_i|^{-\frac{K}{d+1}}\\
&= \alpha_1^{\frac{d(d-1)}{4}} \left( \frac{\beta }{\alpha_1}  \right)^{\frac{\mathcal N}{2}} \left( \frac{\alpha_2}{\alpha_1} \right)^{\frac{d-1}{2}}
\prod_{ i=1 }^{d-2} |\tau_i|^{-\frac{2K}{d(d+1)}(d-1-i+\varepsilon_i)}
|\tau_{d-1}|^{-\frac{2K}{d(d+1)}\varepsilon_{d-1}}
\prod_{j=1}^{d-1}|\tau_d|^{-\frac{2K}{d(d+1)}(1-\varepsilon_j)},
\end{align*}
where $\varepsilon_j= 0$ if $|\tau_j| < \frac12 (4 \pi \nu)^{-\nu}\alpha_2^\nu$, and $\varepsilon_j =1$ otherwise. (See $(vi)$ in Lemma \ref{setup}.)
%Note that the quantity $\varepsilon$ is derived from (3) in Lemma \ref{setup}, depending on whether  $|\tau_j| \ge \varepsilon\alpha_2^\nu$ or not.

Also it is obvious that
\begin{align*}
\max_{1\le i \le d}|\tau_i|^{K} \ge \prod_{1\le i\le d}|\tau_i|^{ K_i },
\end{align*}
for some $K_i > 0 $ satisfying $\sum_{i=1}^d K_i = K$.
By choosing appropriate $K_i$'s  to cancel out the exponents  $\varepsilon_j$'s, we can obtain \eqref{slice-jaco-bound}.
In fact, one can choose $K_i = \frac{2K}{d(d+1)}(d-i +\varepsilon_i)$ for $1\le i \le d-2$, $K_{d-1} = \frac{2K}{d(d+1)}(1 +\varepsilon_{d-1})$, and $K_{d} =
\frac{2K}{d(d+1)}(d-\sum_{j=1}^{d-1}\varepsilon_j)$ satisfying $\sum_{i=1}^{d} K_i = K$.
Hence this gives the desired inequality
\eqref{slice-jaco-bound}.

Now we turn to the error terms.
It suffices to consider two types of error terms, which are
\begin{align}\label{1st-type}
\det( h'(\tau_1),\dots,h'(\tau_{i-1}),h'(u_{j(i)} + \tau_i) - h'(\tau_i),\dots,h'(u_{j(d)} + \tau_d) - h'(\tau_d)),
\end{align}
and
\begin{align}\label{2nd-type}
\det( h'(\tau_1),\dots,h'(\tau_{i-1}),\frac{u_{j(i)}}{\tau_i}h'(u_{j(i)} + \tau_i),\dots,\frac{u_{j(d)}}{\tau_d} h'(u_{j(d)} + \tau_{d} )),
\end{align}
where $u_{j(i)} = s_j \tau_i^{-\frac{2K}{d(d+1)}}$ for some $j \Rightarrow i$.

\bigskip

{\sl An estimate for the second type \eqref{2nd-type}.}
First, we shall find an upper bound of $\eqref{2nd-type}$.
Note that $u_{j(i)} + \tau_i = u_{j(i)} + z_{\lambda_i}= z_{\lambda'_j}$ such that $\lambda'_j$ is bound to $\lambda_i$. (See the definition of $s_j$ at the beginning of the proof.) Using $(iv)$ in Lemma \ref{band} or \eqref{bound}, we observe
that
\begin{align}\label{uji}
|u_{j(i)}| = | z_{\lambda'_j} - z_{\lambda_i} | < \varepsilon \delta \alpha_1^{\frac12}
 |z_{\lambda'_j} z_{\lambda_i}|^{-\frac{ K}{d(d+1)}} \lesssim \varepsilon\delta
\alpha_1^{\frac12}\alpha_1^{-\frac{2K}{d(d+1)}\nu} = \varepsilon\delta\alpha_1^\nu
\end{align}
since we assume that $|z_i| \gtrsim \gamma^\nu = \max\{\alpha_1,\beta_1\}^\nu$ for $1\le i \le 2d-1$, and $|z_{2d}| \gtrsim \alpha_2^\nu \gtrsim
\alpha_1^\nu$ in Lemma \ref{setup}. (Recall that $\nu = d(d+1)/(4K + 2d(d+1))$.) It follows that
\begin{align}\label{ujii}
|u_{j(i)}| \lesssim \varepsilon \delta |z_{\lambda_i}| = \varepsilon \delta |\tau_{i}|
\end{align}
whenever $j \Rightarrow i$.

Thus we see that
\begin{align*}
 | \eqref{2nd-type} | \lesssim (\varepsilon\delta)^{d-i+1} |\det (h'(\tau_1),\dots, h'(\tau_{i-1}), h'(u_{j(i)} +\tau_i),\dots,h'(u_{j(d)} +\tau_d) )|.
\end{align*}
%where $z_{l'} = u_{j(l)} + \tau_l = z_{\lambda'_j}$, $i\le l \le d$.
% Here $\lambda'_j \in \Lambda'$ which is bound to $\lambda_l \in \Lambda$ is denoted by $l'$ to clarify the relation between free index and its bound index.
By \eqref{sameband}, we see that
\begin{align}\label{sim}
 | u_{j(l)} + \tau_l | =| z_{\lambda'_j} | \sim |z_{\lambda_l}| = |\tau_l|
\end{align}
for $i \le l \le d$ and $j\Rightarrow i$.

Recall that if $h(z) = (z,z^2,\dots,z^{d-1},z^N)$, then
\begin{align}\label{jaco}
\det(h'(z_1),\dots,h'(z_d)) = N\, (d-1)! \, V(z_1,\dots,z_d) \, Q_{N-d}(z_1,\dots,z_d),
\end{align}
where $V(z_1,\dots,z_d) = \prod_{1\le i < j \le d} (z_j - z_i)$ is the complex Vandermonde determinant and $Q_m$ is a homogeneous monic polynomial of degree
$m$ defined by
\[  Q_m(z_1, \cdots, z_d) =
\sum_{a_1 + \cdots +a_{d} = m} z_1^{a_1} \cdots z_d^{a_d} .\]
We refer to Section 3 in \cite{BH} for further details of \eqref{jaco}.

If we set $z_{l'} : = z_{\lambda'_j} = u_{j(l)} + \tau_l$,  we obtain that
\begin{align*}
|\eqref{2nd-type}|\lesssim (\varepsilon\delta)^{d-i+1} |V(\tau_1,\dots,\tau_{i-1},z_{i'},\dots,z_{d'})|
\times |Q_{N-d}(\tau_1,\dots,\tau_{i-1},z_{i'},\dots,z_{d'})|.
\end{align*}

We first show that
\begin{align}\label{Vandermonde}
|V(\tau_1,\dots,\tau_{i-1},z_{i'},\dots,z_{d'})| \lesssim |V(\tau_1,\dots,\tau_d)|.
\end{align}
To see this, we will show that $|\tau_i - z_{l'}|\lesssim |\tau_i - \tau_l|$ and $|z_{i'} - z_{l'}| \lesssim |\tau_i - \tau_l|$ for $i\neq l$.
By the triangle inequality, it suffices to show that  $|\tau_l - z_{l'}| \lesssim \varepsilon |\tau_i -\tau_l|$ for $i \neq l$.

If $|\tau_i| \ll |\tau_l|$ or $|\tau_l| \ll |\tau_i|$, we have $|\tau_i - \tau_l| \gtrsim |\tau_l|$.
By \eqref{ujii} we also have that $|\tau_l - z_{l'}| \lesssim \varepsilon \delta |\tau_l|$. Hence we obtain that $|\tau_l - z_{l'}| \lesssim \varepsilon \delta |\tau_i - \tau_l|$ as desired.
%, which implies that $|\tau_i - z_{l'}|\lesssim |\tau_i - \tau_l|$ as desired.

If $|\tau_i| \sim |\tau_l|$, then \eqref{bound} and \eqref{sim} gives
\begin{align*}
|\tau_l - z_{l'}|
& \le  \varepsilon \delta \alpha_1^{\frac12} |\tau_l z_{l'}|^{-\frac{K}{d(d+1)}}
\sim  \varepsilon \delta \alpha_1^{\frac12}(|\tau_l|)^{-\frac{2K}{d(d+1)}}
\sim \varepsilon \delta \alpha_1^{\frac12} |\tau_l \tau_i|^{-\frac{K}{d(d+1)}} \\
& <  \varepsilon |\tau_i - \tau_l|.
\end{align*}
Here the last inequality holds by $(ii)$ in Lemma 3.5 since $i$ and $l$ are in different bands.
Hence we obtain \eqref{Vandermonde}.
%It follows that $|V(\tau_1,\dots,\tau_{i-1},z_{i'},\dots,z_{d'})| \lesssim |V(\tau_1,\dots,\tau_d)|$. %if $\varepsilon$ is small.

Also, we obtain from \eqref{sim} that
\begin{align*}
| Q_{N-d}(\tau_1,\dots,\tau_{i-1},z_{i'},\dots,z_{d'}) |
&\le \sum_{a_1+\cdots+a_d = N-d} |\tau_1|^{a_1}\cdots|\tau_{i-1}|^{a_{i-1}} |z_{i'}|^{a_i} \cdots |z_{d'}|^{a_d} \\
& \sim \sum_{a_1+\cdots+a_d = N-d} |\tau_1|^{a_1}\cdots|\tau_{d}|^{a_{d}}
 \lesssim \max_{1\le i \le d} |\tau_i|^{N-d}.
\end{align*}

Therefore it follows that
\begin{align*}
|\eqref{2nd-type}| \lesssim (\varepsilon\delta)^{d-i+1}\max_{1\le i \le d} |\tau_i|^{N-d} |V(\tau_1,\dots,\tau_d)|
\lesssim (\varepsilon\delta)^{d-i+1} |J_d(\tau)|.
\end{align*}
This finishes the error estimate for the second type.

\bigskip

{\sl An estimate for the first type \eqref{1st-type}.}
We can write
\begin{align*}
\eqref{1st-type} = \int_{\tau_i}^{u_{j(i)} + \tau_i}\cdots \int_{\tau_d}^{u_{j(d)}+ \tau_d} \prod_{l=i}^d \frac{\partial}{\partial \tau_l}\bigg\vert_{\tau_l
=\zeta_k} \det(h'(\tau_1),\dots,h'(\tau_d))\, d\zeta_i\cdots d\zeta_d.
\end{align*}

By \eqref{jaco}, we get
\begin{align*}
\prod_{l=i}^d &\frac{\partial}{\partial \tau_l}\bigg\vert_{\tau_l =\zeta_l}  \det(h'(\tau_1),\dots,h'(\tau_d))\\
&= N \,(d-1)! \, Q_{N-d}(\tau_1,\dots,\tau_{i-1},\zeta_i,\dots,\zeta_d) \prod_{l=i}^d \frac{\partial}{\partial \tau_l}\bigg\vert_{\tau_l =\zeta_l}
V(\tau_1,\dots,\tau_d) + \\
&\qquad+ N\,(d-1)!\,V(\tau_1,\dots,\tau_{i-1},\zeta_i,\dots,\zeta_d) \prod_{l=i}^d \frac{\partial}{\partial \tau_l}\bigg\vert_{\tau_l =\zeta_l}
Q_{N-d}(\tau_1,\dots,\tau_d) \\
&=: I + II.
\end{align*}

First, we consider the following product of derivatives
of the Vandermonde determinant:
\[
\prod_{l=i}^d \frac{\partial}{\partial \tau_l} V(\tau_1,\dots,\tau_d),
\]
which is given by a finite sum of terms
$
 V(\tau_1,\dots,\tau_d) / \prod_{l=i}^d (\tau_l - \tau_{m(l)})
$.
Here, $m(l)$ is an index strictly less than $l$.
Then $|I|$ is bounded by terms such as
\begin{align*}
\frac{ N\, (d-1)! \, |Q_{N-d}(\tau_1,\dots,\tau_{i-1},\zeta_i,\dots,\zeta_d) V(\tau_1,\dots,\tau_{i-1},\zeta_i \dots,\zeta_d)| }{\prod_{l=i}^d |\zeta_l -
\zeta_{m(l)}|}.
\end{align*}
%where $\tau_{m_l} = \zeta_{m_l}$ for $ m_l \ge i$.
Note that $\zeta_{m(l)} = \tau_{m(l)}$ if $m(l) \le i-1$.
To handle the denominator, we need to make some observations.
Since $\zeta_l$ is on the line segment between $\tau_l$ and $u_{j(l)} + \tau_l$, we see that
$ |\tau_l -\zeta_l | \le |u_{j(l)}|$. If $|\tau_l| \ll |\tau_{m(l)}|$ or $|\tau_{m(l)}| \ll |\tau_l|$, it follows from \eqref{ujii} that
\[
|\tau_l -\zeta_l | \le |u_{j(l)}| \le \varepsilon \delta |\tau_l| \lesssim \varepsilon \delta |\tau_l - \tau_{m(l)}|.
\]
If $|\tau_l| \sim |\tau_{m(l)}|$, then \eqref{uji} (with \eqref{sameband}) and $(ii)$ in Lemma \ref{band} gives that
\[
|\tau_l -\zeta_l | \le |u_{j(l)}| \le \varepsilon \delta \alpha_1^{\frac12} |\tau_l|^{-\frac{2K}{d(d+1)}}
\sim \varepsilon \delta \alpha_1^{\frac12} |\tau_l  \tau_{m(l)}|^{-\frac{K}{d(d+1)}} < \varepsilon |\tau_l - \tau_{m(l)}|.
\]
Hence we obtain that $|\tau_l -\zeta_l | < \varepsilon |\tau_l - \tau_{m(l)}|$.
In the same way, one can see that $|\tau_{m(l)} -\zeta_{m(l)} | < \varepsilon |\tau_l - \tau_{m(l)}|$ for $m(l) \ge i$.
By this and the triangle inequality, it follows that
\[
|\zeta_l - \zeta_{m(l)} | \gtrsim |\tau_l - \tau_{m(l)}|
\]
for $1 \le m(l) < l \le d$.

Once again we consider two cases.
If $|\tau_l| \ll |\tau_{m(l)}|$ or $|\tau_{m(l)}| \ll |\tau_l|$,
\[
|\zeta_l - \zeta_{m(l)} | \gtrsim |\tau_l - \tau_{m(l)}| \gtrsim |\tau_l| \gtrsim \alpha^{\frac12}|\tau_l|^{-\frac{2K}{d(d+1)}},
\]
where the last inequality is from the fact that $|\tau_l|\gtrsim \alpha^\nu = \alpha^{\frac{d(d+1)}{4K + 2d(d+1)}}$.
If $|\tau_l| \sim |\tau_{m(l)}|$, it follows from $(ii)$ in Lemma \ref{band} that
\[
|\zeta_l - \zeta_{m(l)} | \gtrsim |\tau_l - \tau_{m(l)}| \gtrsim \delta \alpha^{\frac12} |\tau_l \tau_{m(l)}|^{-\frac{K}{d(d+1)}} \sim \delta \alpha^{\frac12}|\tau_l|^{-\frac{2K}{d(d+1)}}.
\]
%
%First we consider terms $|\zeta_l - \tau_{m(l)}|$ for $ i \le l \le d$ and $ 1 \le m(l) \le i-1$.
%Since $\zeta_l$ is on the line segment between $\tau_l$ and $u_{j(l)} + \tau_l$, it follows from \eqref{diffband} and \eqref{bound} that
%Since $\zeta_l$ is on the line segment between $\tau_l$ and $u_{j(l)} + \tau_l$, we see that for $i \le l \le d$ and $1 \le m(l) \le i-1$,
%\[
%|\zeta_l - \tau_{{m(l)}}| \ge |\tau_l - \tau_{{m(l)}} | - |\tau_l - \zeta_l|
%\gtrsim \delta \alpha_1^{\frac12} |\tau_l|^{-\frac{2K}{d(d+1)}} - \varepsilon\delta\alpha_1^{\frac12}|\tau_l|^{-\frac{2K}{d(d+1)}}
%\gtrsim \delta \alpha_1^{\frac12} |\tau_l|^{-\frac{2K}{d(d+1)}}.
%\]
%by \eqref{diffband} and \eqref{bound}.

%Similarly the terms $|\zeta_l - \zeta_{m(l)} |$, for $1 \le m(l) < l \le d$, can be bounded by
%the fact that $\zeta_l$ and $\zeta_{m(l)}$ are close to $\tau_l$ and $\tau_{m(l)}$ respectively gives
%\[
%|\zeta_l - \zeta_{m(l)} | \gtrsim |\tau_l - \tau_{m(l)} | \gtrsim \delta \alpha_1^{\frac12} |\tau_l|^{-\frac{2K}{d(d+1)}},
%\]
%since $\zeta_l$ and $\zeta_{m(l)}$ are close to $\tau_l$ and $\tau_{m(l)}$, respectively.
%
Thus we obtain that
\[
\prod_{l=i}^d |\zeta_l -\tau_{m(l)}| \gtrsim \prod_{l=i}^d \delta \alpha_1^{\frac12} |\tau_l|^{-\frac{2K}{d(d+1)}}.
\]

Also, the numerator can be bounded by $|J_d(\tau)|$ by following the same argument as the second type estimate.

As a result, we obtain that
\begin{align*}
\Big | \int_{\tau_i}^{u_{j(i)} + \tau_i}\cdots\int_{\tau_d}^{u_{j(d)}+\tau_d}\, I \, d\zeta_i\cdots d\zeta_d \Big|
&\lesssim |J_d(\tau)|\, \frac{\prod_{l=i}^d |u_{j(l)}|}{\prod_{l=i}^d \delta \alpha_1^{\frac12} |\tau_l|^{-\frac{2K}{d(d+1)}} } \\
&\lesssim |J_d(\tau)|\, \frac{\prod_{l=i}^d \varepsilon \delta \alpha_1^{\frac12} |\tau_l| ^{-\frac{2K}{d(d+1)}} }{\prod_{l=i}^d \delta \alpha_1^{\frac12}
|\tau_l|^{-\frac{2K}{d(d+1)}}} \\
&= \varepsilon^{d-i+1} |J_d(\tau)|.
\end{align*}
The last inequality is valid by \eqref{uji} and \eqref{sim}.

Now we consider an estimate for $II$.
Using the definition of the monic polynomial, we see that
\begin{align*}
\bigg| \prod_{l=i}^d \frac{\partial}{\partial \tau_l} \,Q_{N-d}(\tau_1,\dots,\tau_d) \bigg|
& = \bigg| \sum_{\substack{a_1+\cdots+a_d = N-d;\\a_i,\dots,a_d \ge 1}} \left(\prod_{l=i}^d a_l \right)
\tau_1^{a_1}\dots\tau_{i-1}^{a_{i-1}}\tau_i^{a_i-1}\dots\tau_d^{a_d-1} \bigg| \\
& \lesssim \frac{ \max_{1\le i \le d} |\tau_i|^{N-d} }{\prod_{l=i}^d |\tau_l|}.
\end{align*}
Also, the fact that $|V(\tau_i,\dots,\tau_{i-1},\zeta_i,\dots,\zeta_d)| \lesssim |V(\tau_i,\dots,\tau_d)|$ is already obtained.

It follows that
\[
|II| \lesssim \frac{ N\, (d-1)! \, |V(\tau_1,\dots,\tau_d)| \max_{1\le i \le d} |\tau_i|^{N-d} }{ \prod_{l=i}^d |\tau_l|}
\lesssim \frac{|J_d(\tau)|}{\prod_{l=i}^d |\tau_l|} ,
\]
and then
\begin{align*}
\Big | \int_{\tau_i}^{u_{j(i)} + \tau_i}\cdots\int_{\tau_d}^{u_{j(d)}+\tau_d}\, II \, d\zeta_i\cdots d\zeta_d \Big|
\lesssim |J_d(\tau)|  \prod_{l=i}^d \frac{ |u_{j(l)}|}{|\tau_l|} \lesssim (\varepsilon\delta)^{d-i+1} |J_d(\tau)|
\end{align*}
by \eqref{ujii}. Hence the error terms of both types can be bounded by $O(\varepsilon)\times |J_d(\tau)|$.
This completes the proof.
\qed

%--------------------------------------------------------

\begin{section}{Lemmas for optimal Lorentz space inequalities}
In this section, we prove Lemma \ref{trilinearF} which is crucial to show the (nearly) optimal Lorentz boundedness of $T$. (See Lemma A.2.)

\begin{lemma}\label{trilinearF}
Let $F_1, F_2, E \subset \mathbb R^{2d}$ be measurable sets with finite measure. Suppose that
\begin{align*}
\TTchi {F_1} (y) \ge \beta_1 \textrm{   and   } \TTchi {F_2} (y) \ge \beta_2
\end{align*}
for all $y \in E$ and $\beta_1 \le \beta_2$.
%Moreover, suppose that $\beta_i \gtrsim \eta \langle \TTchi{F_i},\chi_E \rangle /|E| $ for some $\eta >0$.
%Also define
Suppose that $\alpha_i$ for $i=1,2$ such that $\langle \TTchi{F_i},\chi_E \rangle /|F_i| \ge \alpha_i $ and $\alpha_2 \le \alpha_1$.
%And
%\begin{align}
%\Tchi D (x) \ge \alpha
%\end{align}
%for all $x \in F$.
Then there exists a constant $C >0$, depending on $N$ and $d$, such that
\begin{align*}
| F_2| \ge C \alpha_1^{r_1} \alpha_2^{r_2} \beta_1^{s_1} \beta_2^{s_2},
%\, \textrm{ or }\, \alpha_1^{\frac{d(d+1)}{2}} \left(\frac{\beta_1}{\alpha_1} \right)^d\left(\frac{\beta_2}{\beta_1} \right)^2
\end{align*}
where $r_1 +r_2 = \frac{d(d-1)}{2}$, $s_1 +s_2 = d$, and $  \frac{s_2}{q_d'} - \frac{r_2}{q_d} -1 >0 $.
%$\left< \Tchi{E}, \chi_{F}\right>= $\alpha_1 \abs{F_1}=\beta_1\abs{E}$.
\end{lemma}

\begin{proof}[Proof of Lemma \ref{trilinearF}]
Similarly to Lemma \ref{setup}, we need the following to prove Lemma \ref{trilinearF}.
\begin{lemma}\label{setup:F}
Let $\gamma_1=\max\{\alpha_{1},\beta_1 \}$, $\nu = \frac{d(d+1)}{4K+2d(d+1)}$. There exist a point $y_0$ in $E$, a constant $C>0$, and a sequence of sets $P_1,\dots,P_{2d-1}$
in $\Delta$ such that
\begin{enumerate}[(i)]
\item $\sigma(P_j) \ge C \beta_1  $ for odd $j < 2d-1$,

\item $\sigma(P_j) \ge C \alpha_1 $ for even $j < 2d-1$,

\item $\sigma(P_{2d-1}) \ge C \beta_2 $,
\end{enumerate}
and $|z_j| \ge ( 4\pi\nu )^{-\nu} \gamma_1^\nu $ for $ z_j \in P_j$ for $1 \le j \le 2d-2$, $|z_{2d-1}| \ge ( 4\pi\nu )^{-\nu} \beta_2^\nu$.

Also there exists  a positive small constant $c$ such that
\begin{enumerate}
\item[(iv)] if $z_j \in P_j$ for odd $j<2d-1$, then $|z_j - z_i| \ge c \beta_1^{\frac12}\abs{z_{i}}^{\frac{-2K}{d(d+1)}}$, where $i < j $,
\item[(v)] if $z_j \in P_j$ for even $j<2d-1$, then $\abs{z_{j}-z_{i}}\ge c \alpha_{1}^{\frac12} \abs{z_{i}}^{\frac{-2K}{d(d+1)}}$, where $i<j$,
\item[(vi)] for $z_{2d-1} \in P_{2d-1}$ and $j < 2d-1$,
%$\epsilon >0$, and $2d > j $, $|z_{2d}| \ge c\alpha_2^{\nu}$ and
$\abs{z_{2d-1} -z_j}\ge c \beta_2^{\frac12} \abs{z_{2d-1}}^{\frac{-2K}{d(d+1)}}$ if $\abs{z_j}< \frac12 ( 4\pi\nu )^{-\nu} \beta_{2}^{\nu}$, and
$\abs{z_{2d-1}-z_{j}} \ge c \beta_{2}^{\frac12}\abs{z_{j}}^{\frac{-2K}{d(d+1)}}$ if $\abs{z_j} \ge \frac12( 4\pi\nu )^{-\nu} \beta_{2}^{\nu}$.
%$| z_{2d} - z_j | \ge c \alpha_2^{\frac12}|z_{2d}|^{\frac{-2k}{d(d+1)}}$
\end{enumerate}
\end{lemma}

We begin with the easy case $\beta_1 \gtrsim \alpha_1$.
Let $\Phi_{2d-1}=\{(z_{1}, z_{2}, ..., z_{2d-1}) : z_{i} \in P_{i}\, \text{ where }\, 1 \le i \le 2d-1 \}$.
We define $z_0=(z_{1}, z_{2}, ... ,z_{d-1}) $ and $\Phi := \{z \in \mathbb{C}^{d} : (z_{0},z)\in \Phi_{2d-1}\}$.
Note that $\sigma(\Phi) \sim \alpha_1^{\lfloor d/ 2 \rfloor} \beta_1^{ \lceil d/  2 \rceil } (\beta_2/\beta_1)$.
For $d \ge 2$, let us define
\begin{align*}
a'_{d}=
\begin{cases}
1 + 3 + \cdots + d &\text{if $d$ is odd}, \\
2+ 4 + \cdots + d &\text{if $d$ is even}.
\end{cases}
\end{align*}

Then the following lemma is obtained by the same argument as in the proof of Lemma \ref{b>a}.
\begin{lemma}%\label{b>a}
Let $d \ge 2$. Assume the hypotheses in Lemma \ref{trilinearF}. Then
\begin{align*}
\abs{F_{2}} \gtrsim \alpha_{1}^{\frac{d(d+1)}2}\left(\frac{\beta_1 }{\alpha_{1}}\right)^{a'_{d}} \left( \frac{\beta_{2}}{\beta_{1}} \right)^{d}.
\end{align*}
\end{lemma}

Since $\beta_1 \gtrsim \alpha_1$, this implies Lemma \ref{trilinearF}. In fact, $(\beta_1/\alpha_1)^{a'_d} \geq (\beta_1 / \alpha_1)^{d}$, and
$(\beta_2/\beta_1)^d \geq (\beta_2/\beta_1)^2$ by the assumption $\beta_2 \ge \beta_1$ of Lemma \ref{trilinearF}.

Now suppose that $\beta_1 \ll \alpha_1$.
In this case we obtain the following lemma similar to Lemma \ref{band}.
%The proof is the same except for the fact that $\alpha_2$ is substituted by $\alpha_1$.

\begin{lemma}\label{band_F}
Let $\varepsilon >0$. Then there exist parameters $\delta, \delta'$ satisfying $c_{d,\varepsilon} < \delta' <\varepsilon\delta <\varepsilon c$, a positive constant
$c_0$, an integer $d\le k < 2d$, an element $z_0$, a set $\omega \subset \mathbb C^{k}$ with
\begin{equation}\label{omega}
\sigma(\omega) \sim \alpha_1^{\lfloor \frac k2 \rfloor} \beta_1^{\lceil \frac k2 \rceil } \left( \frac{\beta_2}{\beta_1} \right),
\end{equation}
and a band structure on $\{2d-k,\dots, 2d-1\}$, such that the following properties hold:
\begin{enumerate}[(i)]
\item There are exactly $d$ free or quasi-free indices. In particular, each even index is free.
\item $|z_i - z_j | >  \delta \alpha_1^{\frac12}  |z_i z_j|^{-K/d(d+1)}$, unless $i$ and $j$ lie in the same band.
\item $  c_0  \beta_1^{\frac12} |z_i z_j|^{-K/d(d+1)} < |z_i - z_j| \lesssim \delta \alpha_1^{\frac12}  |z_i z_j|^{-K/d(d+1)}$ whenever $i$ is
quasi-bound to $j$.
\item $\delta' \alpha_1^{\frac12}  |z_i z_j|^{-K/d(d+1)} > |z_i - z_j| $ whenever $i$ is bound to $j$.
\end{enumerate}
\end{lemma}
Note that \eqref{omega} may be deduced from $(i)-(iii)$ in Lemma \ref{setup:F}.

Now suppose that $\beta_2 \gtrsim \alpha_1$.
It follows that $2d-1$ must be a free index without quasi-bound and bound indices after carrying out Lemma \ref{band_F}.
Thus \eqref{slice-jaco-bound} will be modified as follows:
\begin{align*}%\label{jaco_F}
\Big| \det\left(\frac{\partial H(z(\tau,s)) }{\partial \tau}\right) \Big|^2
 \gtrsim \alpha_1^{\frac{d(d-1)}{2}} \left( \frac{\beta_1}{\alpha_1} \right)^{\mathcal N}\left( \frac{\beta_2}{\alpha_1} \right)^{d-1} \prod_{i=1}^{d} |\tau_i|^{\frac{4K}{d(d+1)}}
%&\gtrsim \alpha_1^{\frac{d(d-1)}{2}} \left( \frac{\beta_1}{\alpha_1} \right)^{\mathcal N}
%\prod_{i=1}^{d} |\tau_i|^{\frac{4K}{d(d+1)}},
\end{align*}
where $\mathcal N$ is the number of quasi-free indices.
Similarly to the proof of Lemma \ref{d-trilinear}, it follows by \eqref{omega} that
\begin{align*}
|F_2|
&\gtrsim \alpha_1^{\frac{d(d-1)}{2}} \left( \frac{\beta_1}{\alpha_1} \right)^{\mathcal N} \left( \frac{\beta_2}{\alpha_1} \right)^{d-1} \alpha_1^{d-k} \alpha_1^{\lfloor \frac k2 \rfloor} \beta_1^{\lceil
\frac k2 \rceil } \left( \frac{\beta_2}{\beta_1} \right)\\
& = \alpha_1^{\frac{d(d+1)}{2}} \left( \frac{\beta_1}{\alpha_1} \right)^{\mathcal N  + \lceil \frac k2 \rceil } \left( \frac{\beta_2}{\alpha_1} \right)^{d-1} \left( \frac{\beta_2}{\beta_1} \right).
\end{align*}
Since $2d-k$, $2d-1$, and all even indices between $2d-k$ and $2d-1$ are free indices, the number of free indices is at least $\lfloor \frac k2 \rfloor + 2$.
Hence $\mathcal N + \lceil \frac k2 \rceil  \le d-1$.
Since we have $\beta_2 \gtrsim \alpha_1$, we conclude that $|F_2| \gtrsim \alpha_1^{\frac{d(d-1)}{2}} \beta_1^{d-2} \beta_2^2$.
This satisfies the relations in Lemma \ref{trilinearF}.

We assume that $\beta_2 \ll \alpha_1$. Then the index $2d-1$ may not be free.
The number of free indices is at least $\lfloor \frac k2 \rfloor +1$, which means $\mathcal N + \lceil \frac k2 \rceil \le d$.
One can see that this is not enough for the desired bound.
So we follow the argument using a two-stage band structure due to Stovall \cite{S1}.

Let $\mathscr B(2d-1)$ be the band containing $2d-1$ after carrying out Lemma \ref{band_F}.
Now, we decompose $\mathscr B(2d-1)$ into sub-bands as follows.
%{\sl Two-stage band structure with $\alpha_2$, $\beta_2$.}
%We will do similarily the band structure about $B(2d-1)$.
For $\varepsilon>0$ and $c_{d,\varepsilon}$ in Lemma \ref{band_F}, there exist $\rho$ and $\rho'$ such that $c_{d,\varepsilon}<\rho' <\varepsilon \rho < \delta'$, and a subset $\omega'$ of $\omega$ satisfying that $\sigma(\omega') \sim \sigma(\omega)$. Then the following properties hold:
\begin{enumerate}[(i)]
\item $\abs{z_{i}-z_{j}} > \rho \gamma_{2}^{1/2} | z_{i} z_{j} |^{-K/d(d+1)}$ unless $i$ and $j$ lie in the same band.
\item $c_{0}\beta_{1}^{1/2}| z_{i} z_{j} |^{-K/d(d+1)} < \abs{z_{i} - z_{j}} \le \rho\gamma_{2}^{1/2} | z_{i} z_{j} |^{-K/d(d+1)}$ whenever $i$ is quasi-bound to $j$.
\item $\rho' \gamma_{2}^{1/2} | z_{i} z_{j} |^{-K/d(d+1)} > \abs{z_{i} - z_{j}}$ whenever $i$ is bound to $j$,
\end{enumerate}
for $i, j \in \mathscr B(2d-1)$, $\gamma_{2}=\max \{\alpha_{2}, \beta_{2}\}$ and some constant $c_{0}>0$.

After this step, the number of free and quasi-free indices in $\{2d-k,\dots, 2d-1\}$ may increase.
Then we repeat {\sl Step 3} in the proof of Lemma \ref{band} (the step of eliminating some indices) until we get exactly $d$ free and quasi-free indices in $\{ 2d-k',\dots, 2d-1\}$ for some integer $k'$.
By abuse of notation we will write $k$ instead of $k'$.

Let $F_1$ and $Q_1$ be the number of free and quasi-free indices which are contained in $\{2d-k, \dots, 2d-1\} \backslash \mathscr B(2d-1)$.
Also, let $F_2$ and $Q_2$ be the number of free and quasi-free indices which are contained in $\mathscr B(2d-1)$.
Note that $F_1 +Q_1 +F_2+Q_2 = d$.
We set $M=F_{2}+Q_2$, and the number of elements of $\mathscr B(2d-1)$ is denoted by $N$. Then $N-M$ denotes the number of bound indices in $\mathscr B(2d-1)$.

The case when $\mathscr B(2d-1) ={2d-1}$ is the  same as the case  $\beta_2 \gtrsim \alpha_1$ above.
Hence we consider the following three cases:
\begin{enumerate}
\item $2d-1$ is free and there is at least one free index other than $2d-1$ in $\mathscr B(2d-1)$.
\item $2d-1$ is quasi-free.
\item $2d-1$ is bound to some $j$ in $\mathscr B(2d-1)$.
\end{enumerate}

{\sl Case (1)}. Since we  are in  the case that $2d-1$ is free, we have $F_1 + F_2 \ge \lfloor \frac k 2 \rfloor +2$.  In this case, we get the lower bound of Jacobian \eqref{slice-jaco-bound} as follows:
\[
\Big|\det\left(\frac{\partial H(z(\tau,s)) }{\partial \tau}\right) \Big|^2 > C \alpha_{1}^{\frac{d(d-1)}{2}} \left(\frac{\gamma_{2}}{\alpha_{1}}\right)^{\frac{M(M-1)}{2}}  \left(\frac{\beta_{1}}{\alpha_{1}}\right)^{Q_{1}}  \left(\frac{\beta_{1}}{\gamma_{2}}\right)^{Q_{2}}
\prod_{i=1}^{d} |\tau_i|^{\frac{4K}{d(d+1)}}.
\]

In addition, note that $\sigma(\{ s : z(\tau,s) \in \omega  \})$
is bounded above by $\alpha_{1}^{k-d} \left( \frac{\gamma_{2}}{\alpha_{1}}\right)^{N-M}$ in this case.
Thus, by combining these, we obtain that
\begin{align*}
\abs{F_{2}} &\gtrsim \alpha_{1}^{\frac{d(d-1)}{2}} \left(\frac{\gamma_{2}}{\alpha_{1}}\right)^{\frac{M(M-1)}{2}} \left(\frac{\beta_{1}}{\alpha_{1}}\right)^{Q_{1}}  \left(\frac{\beta_{1}}{\gamma_{2}}\right)^{Q_{1}} \times \alpha_{1}^{d-k} \left( \frac{\gamma_{2}}{\alpha_{1}}\right)^{M-N} \times \alpha_1^{\lfloor \frac k2 \rfloor} \beta_1^{\lceil \frac k2 \rceil } \left( \frac{\beta_2}{\beta_1} \right) \\
&= \alpha_{1}^{\frac{d(d+1)}{2}}\left(\frac{\beta_{1}}{\alpha_{1}}\right)^{Q_{1}+ \lceil \frac{k}{2}\rceil}  \left(\frac{\gamma_{2}}{\alpha_{1}}\right)^{\frac{M(M-1)}{2}} \left(\frac{\beta_{1}}{\gamma_{2}}\right)^{Q_{2}} \left(\frac{\gamma_{2}}{\alpha_{1}}\right)^{M-N} \left(\frac{\beta_{2}}{\beta_{1}}\right) \\
%&=\alpha_{1}^{\frac{d(d+1)}{2}}\left(\frac{\beta_{1}}{\alpha_{1}}\right)^{Q_{1}+Q_{2}+ \lceil \frac{k}{2} \rceil -1} \times \left(\frac{\gamma_{2}}{\alpha_{1}}\right)^{\frac{M(M+1)}{2} - N - Q_{2}}\left(\frac{\beta_{2}}{\alpha_{1}}\right) \\
&\ge \alpha_1^{\frac{d(d+1)}{2}}  \left( \frac{\beta_1}{\alpha_1} \right)^{Q_1 + Q_2 + \lceil \frac k 2 \rceil - 1  } \left(   \frac{\gamma_2}{\alpha_1} \right)^{\frac{M(M-1)}{2} - Q_2} \left( \frac{\beta_2}{\alpha_1}\right) .
\end{align*}
The last inequality holds, because $N-M \ge 0$ and $\gamma_2 \ll \alpha_1$.
Since $2d-1$ is free, we have that $Q_{1}+Q_{2} + \lceil \frac{k}{2} \rceil -1 \le Q_1 +Q_2 +F_1+F_2 -2 =d-2$.

%%%%%%%%%%%%%%
If $\frac{M(M-1)}2 \le d - Q_1 -\lceil \frac k2 \rceil$, then we may write $ ( \gamma_2 /\alpha_1 )^{\frac{M(M-1)}{2} - Q_2} \ge (\beta_2 / \alpha_1)^{d - Q_1 -Q_2 -\lceil \frac k 2 \rceil}$.
%Also we see that $( {\beta_1}/{\alpha_1} )^{Q_1 + Q_2 + \lceil \frac k 2 \rceil - 1} \ge  ( {\beta_1}/{\alpha_1} )^{d-2}$ since $Q_1 + Q_2 +\lceil \frac k 2 \rceil \le d -1$.
%\chg{ Moreover it is vaild that $(\beta_2/\alpha_2)^{d-Q_1 - Q_2 -\lceil \frac k2 \rceil +1} \ge (\beta_2 /\alpha_2)^2$.}
Then we have that $|F_2| \gtrsim \alpha_1^{r_1} \alpha_2^{r_2} \beta_1^{s_1} \beta_2^{s_2}$ with
$r_1 = d(d-1)/2$, $r_2 =0$ and
\begin{align*}
s_1   = Q_1 + Q_2 + \Big\lceil \frac k 2 \Big\rceil -1, \quad 
s_2   = d- Q_1 - Q_2 -\Big\lceil \frac k 2 \Big\rceil +1.
\end{align*}
It is easy to verify that the relations of $r_1,r_2,s_1,s_2$ in Lemma \ref{trilinearF} are valid by the fact that $s_2 \ge 2$ implies $s_2/ q_d' >1$.

If $\frac{M(M-1)}2 > d - Q_1 -\lceil \frac k2 \rceil$, we may write $\gamma_2^{\frac{M(M-1)}2 -Q_2} \ge \beta_2^{d - Q_1 -Q_2 -\lceil \frac k 2 \rceil} \alpha_2^{\frac{M(M-1)}2  + Q_1 + \lceil \frac k 2 \rceil - d}$.
Then $|F_2| \gtrsim \alpha_1^{r_1} \alpha_2^{r_2} \beta_1^{s_1} \beta_2^{s_2}$ holds with
\begin{align*}
r_1 & = \frac{d(d+1)}{2} -\frac{M(M-1)}{2} - Q_1 - \Big\lceil \frac k 2 \Big\rceil,  \\
r_2 & = \frac{M(M-1)}{2} + Q_1 + \Big\lceil \frac k 2 \Big\rceil - d,  \\
s_1 & = Q_1 + Q_2 + \Big\lceil \frac k 2 \Big\rceil -1, \\
s_2 & = d- Q_1 - Q_2 -\Big\lceil \frac k 2 \Big\rceil +1.
\end{align*}
It is easy to check that $r_1+r_2 = d(d-1)/2$ and $s_1 + s_2 = d$.
Now let us verify $s_2 /q_d' - r_2 /q_d -1 > 0$.
Since $2d-1$ is free and is the largest index, we see that $2d-1$ has no quasi-bound index, hence $Q_2 \le (M-1)/2$ . Also it is valid that $Q_1 + \lceil k/2 \rceil \le d - M +1$, since $F_1$ is at least $\lceil k/2 \rceil - 1$ and $d = F_1 + Q_1 + F_2 + Q_2 = F_1 + Q_1 +M$.
Hence we see that $ s_2 /q_d' - r_2 /q_d -1 \ge (M-1)/2 - 1/q_d - (M-1)^2/2q_d$, which is positive for all $2 \le M \le d/2 +1$. Here, we may assume that $M \ge 2$, because $\mathscr B (2d-1)$ contains at least two elements.
Then the least element and $2d-1$ are free.
%Since $2d-1$ is free, $M$ is at least 1.
Also, we get $M \le d/2 +1$ from the fact that $Q_1 + \lceil k/2 \rceil \le d - M +1$ implies $ M \le d - \lceil k/2 \rceil +1 \le d/2 +1$.  By the concavity, $(M-1)/2 - 1/q_d - (M-1)^2/2q_d >0 $ if it is positive when $M = 2$ and $M = d/2+1$.
For the case $M=2$, the condition $d \ge 4$ is necessary.
Thus we conclude that $s_2 /q_d' - r_2 /q_d -1 > 0$ holds for all possible $M$ whenever $d \ge 4$.

When $d= 3$, one can see that the only possible cases are $k=3$ and all indices $3,4,5$ are free.
(Note that $4,5$ should be free by the assumptions.)
Thus $Q_1 = 0$ and then $M(M-1)/2 = d -Q_1 -\lceil \frac k2  \rceil$ since $M=2$.
%We  refer to Section 5. %?? (for what?)
We will consider more general $\mathbb{C}^3$ curves in Section 5.

{\sl Case (2)}. When $2d-1$ is quasi-free, we have that $F_1 + F_2 \ge \lfloor \frac k 2 \rfloor +1$. In this case, the lower bound of Jacobian \eqref{slice-jaco-bound} can be modified as
\[
\Big| \det\left( \frac{\partial H(z(\tau,s)) }{\partial \tau}\right) \Big|^2 \ge C \alpha_1^{\frac{d(d-1)}{2}} \left( \frac{\gamma_2}{\alpha_1} \right)^{\frac{M(M-1)}{2}} \left( \frac{\beta_1}{\alpha_1} \right)^{Q_1} \left( \frac{\beta_1}{\gamma_2} \right)^{Q_2-1} \left(\frac{\beta_2}{\gamma_2}\right) \prod_{i=1}^{d} |\tau_i|^{\frac{4K}{d(d+1)}}.
\]
Then it follows that
\begin{align*}
|F_2| &\gtrsim \alpha_1^{\frac{d(d-1)}{2}} \left( \frac{\gamma_2}{\alpha_1} \right)^{\frac{M(M-1)}{2}} \left( \frac{\beta_1}{\alpha_1} \right)^{Q_1} \left( \frac{\beta_1}{\gamma_2} \right)^{Q_2-1} \left(\frac{\beta_2}{\gamma_2}\right)  \alpha^{d-k} \left( \frac{\gamma_2}{\alpha_1} \right)^{M-N} \alpha_1^{\lfloor \frac k2 \rfloor} \beta_1^{\lceil \frac k2 \rceil} \left( \frac{\beta_2}{\beta_1} \right) \\
& = \alpha_1^{\frac{d(d+1)}{2}}  \left( \frac{\beta_1}{\alpha_1} \right)^{Q_1 + Q_2 + \lceil \frac k 2 \rceil - 2  } \left(   \frac{\gamma_2}{\alpha_1} \right)^{\frac{M(M-1)}{2} - (N-M) - Q_2} \left( \frac{\beta_2}{\alpha_1}\right)^2  \\
&\ge \alpha_1^{\frac{d(d+1)}{2}}  \left( \frac{\beta_1}{\alpha_1} \right)^{Q_1 + Q_2 + \lceil \frac k 2 \rceil - 2  } \left(   \frac{\gamma_2}{\alpha_1} \right)^{\frac{M(M-1)}{2} - Q_2} \left( \frac{\beta_2}{\alpha_1}\right)^2 .
\end{align*}
The last inequality holds, because $N-M \ge 0$ and $\gamma_2 \ll \alpha_1$.

If $\frac{M(M-1)}2 \le d - Q_1 -\lceil \frac k2 \rceil$, then we may write $ ( \gamma_2 /\alpha_1 )^{\frac{M(M-1)}{2} - Q_2} \ge (\beta_2 / \alpha_1)^{d - Q_1 -Q_2 -\lceil \frac k 2 \rceil}$.
Also we see that $( {\beta_1}/{\alpha_1} )^{Q_1 + Q_2 + \lceil \frac k 2 \rceil - 2} \ge  ( {\beta_1}/{\alpha_1} )^{d-2}$, since $Q_1 + Q_2 +\lceil \frac k 2 \rceil \le Q_1+Q_2+F_1+F_2\le d$.
Then we have that $r_1 = d(d-1)/2$, $r_2 =0$, $s_1 = d-2$, and $s_2 = 2$. It is easy to check the relations of $r_1,r_2,s_1,s_2$ in Lemma \ref{trilinearF}.

If $\frac{M(M-1)}2 > d - Q_1 -\lceil \frac k2 \rceil$, we may write $r_2^{\frac{M(M-1)}2 -Q_2} \ge \beta_2^{d - Q_1 -Q_2 -\lceil \frac k 2 \rceil} \alpha_2^{\frac{M(M-1)}2  + Q_1 + \lceil \frac k 2 \rceil - d}$. Then we have that
\begin{align*}
r_1 & = \frac{d(d+1)}{2} -\frac{M(M-1)}{2} - Q_1 - \Big\lceil \frac k 2 \Big\rceil,  \\
r_2 & = \frac{M(M-1)}{2} + Q_1 + \Big\lceil \frac k 2 \Big\rceil - d , \\
s_1 & = Q_1 + Q_2 + \Big\lceil \frac k 2 \Big\rceil -2 ,\\
s_2 & = d- Q_1 - Q_2 -\Big\lceil \frac k 2 \Big\rceil +2.
\end{align*}
It is easy to check that $r_1+r_2 = d(d-1)/2$ and $s_1 + s_2 = d$.
Now let us verify $s_2 /q_d' - r_2 /q_d -1 > 0$.
Since $Q_2 \le M/2$ and $Q_1 + \lceil k/2 \rceil \le d - M +1$, we see that
$ s_2 /q_d' - r_2 /q_d -1 \ge M/2 - 2/q_d - M(M-2)/2q_d$, which is positive for all $2 \le M \le d/2 +1$. Since $2d-1$ is quasi-free, $M$ is at least 2. Also, $Q_1 + \lceil k/2 \rceil \le d - M +1$ implies $ M \le d - \lceil k/2 \rceil +1 \le d/2 +1$. Thus we conclude that $s_2 /q_d' - r_2 /q_d -1 > 0$ holds for all possible $M$ whenever $d \ge 4$.

When $d=3$, we can check that $\frac{M(M-1)}{2} = d - Q_1 -\lceil \frac k2 \rceil$ holds.
In fact, the only possible cases are that $5$ is quasi-bound to $3$ or $4$ with $k=3$. Then $Q_1=0$ and $M = F_2 + Q_2 =2$.

{\sl Case (3)}. Finally we consider the case that $2d-1$ is bound to some $j \in \mathscr B$. In this case, we have that
\[
\Big| \det\left( \frac{\partial H(z(\tau,s)) }{\partial \tau}\right) \Big|^2 \ge C \alpha_1^{\frac{d(d-1)}{2}} \left( \frac{\gamma_2}{\alpha_1} \right)^{\frac{M(M-1)}{2}} \left( \frac{\beta_1}{\alpha_1} \right)^{Q_1} \left( \frac{\beta_1}{\gamma_2} \right)^{Q_2} \prod_{i=1}^{d} |\tau_i|^{\frac{4K}{d(d+1)}}.
\]
Thus one can see that
\begin{align*}
|F_2| &\gtrsim \alpha_1^{\frac{d(d-1)}{2}} \left( \frac{\gamma_2}{\alpha_1} \right)^{\frac{M(M-1)}{2}} \left( \frac{\beta_1}{\alpha_1} \right)^{Q_1} \left( \frac{\beta_1}{\gamma_2} \right)^{Q_2} \alpha^{d-k} \left( \frac{\gamma_2}{\alpha_1} \right)^{M-N} \alpha_1^{\lfloor \frac k2 \rfloor} \beta_1^{\lceil \frac k2 \rceil} \left( \frac{\beta_2}{\beta_1} \right) \\
& = \alpha_1^{\frac{d(d+1)}{2}}  \left( \frac{\beta_1}{\alpha_1} \right)^{Q_1 + Q_2 + \lceil \frac k 2 \rceil - 1  } \left(   \frac{\gamma_2}{\alpha_1} \right)^{\frac{M(M-1)}{2} - Q_2 - (N-M) } \left( \frac{\beta_2}{\alpha_1}\right).
\end{align*}

Since $2d-1$ is bound to some $j$, there are at least two indices which are bound to $j$.
Then we have $N-M \ge 2$, which implies $(\gamma_2 /\alpha_1)^{-(N-M)} \ge (\gamma_2 /\alpha_1)^{-2}$.
Thus we obtain that
\begin{align*}
|F_2| \gtrsim \alpha_1^{\frac{d(d+1)}{2}}  \left( \frac{\beta_1}{\alpha_1} \right)^{Q_1 + Q_2 + \lceil \frac k 2 \rceil - 1  } \left(   \frac{\gamma_2}{\alpha_1} \right)^{\frac{M(M-1)}{2} - Q_2 - 2 } \left( \frac{\beta_2}{\alpha_1}\right).
\end{align*}

Note that $d - Q_1 - Q_2 -\lceil \frac k2 \rceil \ge 0$ and we may assume $\gamma_2 = \alpha_2 > \beta_2$. (If $\alpha_2  \le \beta_2$, then $2d-1$ should be free.)
Hence we may write
\[
 \gamma_2^{\frac{M(M-1)}{2} - Q_2 - 2 } \ge   \beta_2^{d-Q_1 - Q_2 - \lceil \frac k2 \rceil } \alpha_2^{\frac{M(M-1)}{2} +Q_1 +\lceil \frac k2 \rceil - d- 2 } .
\]
It follows that $|F_2| \gtrsim \alpha_1^{r_1}\alpha_2^{r_2}\beta_1^{s_1}\beta_2^{s_2}$ holds with
\begin{align*}
r_1 & = \frac{d(d+1)}{2} - Q_1 - \Big\lceil \frac k 2 \Big\rceil -\frac{M(M-1)}{2} +2, \\
r_2 & = \frac{M(M-1)}{2} + Q_1 + \Big\lceil \frac k 2 \Big\rceil - d -2, \\
s_1 & = Q_1 + Q_2 + \Big\lceil \frac k 2 \Big\rceil -1, \\
s_2 & = d- Q_1 - Q_2 -\Big\lceil \frac k 2 \Big\rceil +1.
\end{align*}

It is easy to check that $r_1 + r_2 = d(d-1)/2$ and $s_1 + s_2 =d$.
Now let us verify $s_2 /q_d' - r_2 /q_d -1 > 0$ for all $d \ge 3$. Since $Q_2 \le (M-1)/2$ and $Q_1 + \lceil k/2 \rceil \le d - M +1$, it follows that
$ s_2 /q_d' - r_2 /q_d -1 \ge (M-1)/2 + 1/q_d - (M-1)^2/2q_d$, which is positive for all $1 \le M \le d/2 +1$.
Since the assumption that $2d-1$ is bound implies $M \ge 1$, and the condition $Q_1 + \lceil k/2 \rceil \le d - M +1$ implies $M \le d/2+1$ for all , we can conclude that $s_2 /q_d' - r_2 /q_d -1 > 0$ for all possible $M$ and for all $d \ge 3$.
This completes the proof.
\end{proof}

%%%%%%%%%%%%%%%%%%%%%%%%%%%%%%%%%%%%%%%%%%%%%%%%%%%%%%%%%%%%%%%%%%%%%%%%%%%%%%%%%%%
\section{The polynomial curves in $\mathbb C^3$}

In this section, we consider  polynomial curves of simple type in $\mathbb C^3$.
Let $h(z) = (z,z^2,\phi(z))$ with a complex polynomial $\phi(z)$ of degree at most $N$.
Then $d\sigma(z) = |\phi'''(z)|^{1/3} d\mu(z)$.
By Lemma \ref{jaco_poly}, it is enough to consider a restricted domain $B := D \cap B_\ell$ for some $1 \le \ell \le M$.
By \eqref{assum:poly}, we also see that $|\phi'''(z)| \sim |z|^k$ whenever $z \in B_\ell \subset S\cap \Delta\cap E_k$ with $E_k = G_k$ or $D_k$.

%We shall obtain a restricted weak type estimate at first.
Thus we can set
\begin{align}\label{Tf}
T f(x) = \int_{B} f(x-h(z)) \, d\sigma(z) = \int_B f(x-h(z)) \,|z|^k d\mu(z).
\end{align}

This section is actually redundant, because \eqref{Tf} is the same as \eqref{operator}. However, we will show that the $d=3$ case can be treated directly without the band structure.
Also, the complex version of the band structure in Section 3 was motivated by the proof of Lemma \ref{trilinear_3d}.

Again, we will show that two refinements of the restricted weak type estimates hold:
\begin{align}\label{Eab}
|E| \gtrsim \alpha^4 \beta^2 \quad \textrm{ or }\quad |F| \gtrsim \alpha^3 \beta^3.
\end{align}

First, we prove the following, by which one can obtain the strong type $(2,3)$ estimate.
(See Lemma \ref{weak_ver} and the beginning of Appendix A.)
\begin{lemma}\label{trilinearE}
Let $E_1,E_2,G \subset \mathbb R^6$ be measurable sets with finite measure. Suppose that
\begin{align*}
\Tchi{E_1}(x) \ge \alpha_1 \textrm{   and   } \Tchi{E_2} (x) \ge \alpha_2
\end{align*}
for all $x \in G$ and $\alpha_1 \le \alpha_2$. Then
\begin{align*}
| E_2| \gtrsim \alpha_1 \, \alpha_2^3 \,\beta^2,
\end{align*}
where $\beta = \alpha_1 \frac{|G|}{|E_1|}$.
\end{lemma}

The proof is similar to the real case, since we make a comparison between the absolute values of complex variables. (See \cite{DLW}.)

\begin{proof}
Since $ \langle \TTchi G , \chi_{E_1} \rangle = \langle \chi_G, \Tchi {E_1} \rangle \ge \alpha_1 |G| = \beta |E_1|$, we can define a set
\begin{align*}
E_1^1 = \{ y \in E_1 : \TTchi G(y) \ge \beta/2 \}.
\end{align*}
It follows that
\begin{align*}
\langle \Tchi{E_1^1},\chi_G \rangle = \langle \chi_{E_1^1}, \TTchi G \rangle = \langle \chi_{E_1} , \TTchi G \rangle - \langle \chi_{E_1\setminus E_1^1} ,
\TTchi G \rangle \ge \alpha_1 |G| - \frac{\beta}{2} |E_1| = \frac {\alpha_1} 2 |G|.
\end{align*}
Thus we can define
\begin{align*}
G^1 = \{ x\in G : \Tchi{E_1^1} (x) \ge \alpha_1/4 \}.
\end{align*}
One can see that $G^1$ is not empty. %(See the proof of Lemma 1 in \cite{DLW}.)

For $x_0 \in G^1$, we set
\[
P = \{ z_1 \in B : x_0 - h(z_1) \in E_1^1\} \textrm{ and then }
\sigma(P) = \Tchi{E_1^1}(x_0) \ge \alpha_1/4.
\]
For all $z_1 \in P$, we also set
\[
Q_{z_1} = \{ z_2 \in B : x_0 - h(z_1)+h(z_2) \in G \} \textrm{ and then } \sigma(Q_{z_1}) = \TTchi{G}(x_0 - h(z_1)) \ge \beta/2.
\]
For $z_1 \in P$ and $z_2 \in Q_{z_1}$, we set
\[
R_{z_1,z_2} = \{ z_3 \in B : x_0 -h(z_1) + h(z_2) -h(z_3) \in E_2\}.
\]
Then we see that $\sigma(R_{z_1,z_2}) = \Tchi{E_2}(x_0 -h(z_1)+h(z_2))  \ge \alpha_2$.

Let $\mathcal S = \{ (z_1,z_2,z_3) : z_1 \in P ,\, z_2 \in Q_{z_1},\, z_3 \in R_{z_1,z_2} \}$.
And if we set
\begin{align}\label{Fun_Phi}
\Phi_h(z_1 , z_2, z_3) = - h(z_1)+ h(z_2) - h(z_3),
\end{align}
then $x_0 + \Phi_h(\mathcal S) \subset E_2$.
From Lemma \ref{jaco_poly} and B\'ezout's theorem, we have that
\begin{align*} \label{Jacobian_E}
|E_2| & \ge C \iiint_{\mathcal S} |J_{\mathbb R} \Phi_h(z_1,z_2,z_3)|\, d\mu(z_1)d\mu(z_2)d\mu(z_3) \notag \\
& = \iiint_{\mathcal S} |J_{\mathbb C} \Phi_h(z_1,z_2,z_3) |^2 \, d\mu(z_1)d\mu(z_2)d\mu(z_3)\notag \\
& \ge C \iiint_{\mathcal S} \max\{|z_1|,|z_2|,|z_3|\}^{2k} |z_2-z_1|^2|z_3-z_1|^2|z_3-z_2|^2 d\mu(z_1)d\mu(z_2)d\mu(z_3)
\end{align*}
whenever $z_1,z_2,z_3 \in \mathcal S$. The last integrand is obtained from \eqref{jaco:poly-mono}.

To obtain a lower bound of the last integral, we follow the argument in \cite{DLW}.
Let us define a set
\[
B_\alpha = \{ z\in B : |z| \le ( 16 \pi \nu)^{-\nu} \alpha^{\nu} \},
\]
for $\nu = \frac 3{k+6}$ and  a fixed $k$ as in the definition of $T$.
Then we see that
\[
\sigma(B_\alpha) = \int_{B_\alpha} |z|^{\frac k 3} d\mu(z)
 \le 2 \pi \int_{0}^{( 16\pi\nu )^{-\nu}\alpha^{\nu}} r^{\frac{k}{3}+1} d r = 2 \pi \nu (16\pi\nu)^{-1} \alpha = \frac \alpha 8.
\]
%We will choose $\delta >0$ to be sufficiently small so that $\sigma( U \setminus B_\alpha) > \frac{c_0}{2} \alpha$ whenever $U $ is a subset of $B$ satisfying $\sigma(U) \ge c_0\alpha$.
Thus we may assume that $|z_1| \ge ( 16 \pi \nu)^{-\nu} \alpha_1^{\nu}$ on $P$.
In fact, we can replace $P$ by $P \setminus B_{\alpha_1}$ since $\sigma(P \setminus B_{\alpha_1}) \ge \sigma(P) - \sigma(B_{\alpha_1}) \ge \frac {\alpha_1}
8$.
We also assume that $|z_2| \ge ( 16 \pi \nu)^{-\nu} \beta^{\nu}$ on $Q_{z_1}$ and $|z_3| \ge ( 16 \pi \nu)^{-\nu} \alpha_2^{\nu}$ on $R_{z_1,z_2}$ in the
same way.

The following lemma implies that we may assume $z_1,z_2$ and $z_3$ are separated from each other.
\begin{lemma}\label{separation}
There exists a small constant $c > 0$ such that for $(z_1,z_2,z_3) \in \mathcal S$
\begin{enumerate}[(i)]
\item $| z_2 - z_1 | \ge c \beta^{\frac 12} |z_2|^{-\frac k 6 \varepsilon_1 } |z_1|^{-\frac k 6 (1-\varepsilon_1)}$, where $\varepsilon_1 = 0$ if $z_1 \in
B_{\beta/2}$, or $\varepsilon_1 = 1$ if $ z_1 \notin B_{\beta/2}$.
\item $| z_3 - z_1 | \ge c \alpha_2^{\frac 12} |z_3|^{-\frac k 6 \varepsilon_2 } |z_1|^{-\frac k 6 (1-\varepsilon_2)}$, where $\varepsilon_2 = 0$ if $ z_1 \in
B_{\alpha_2/2}$, and $\varepsilon_2 = 1$ if $ z_1 \notin B_{\alpha_2/2}$.
\item $| z_3 - z_2 | \ge c \alpha_2^{\frac 12} |z_3|^{-\frac k 6 \varepsilon_3 } |z_2|^{-\frac k 6 (1-\varepsilon_3)}$, where $\varepsilon_3 = 0$ if $ z_2 \in
B_{\alpha_2/2}$, and $\varepsilon_3 = 1$ if $ z_2 \notin B_{\alpha_2/2}$.
\end{enumerate}
\end{lemma}
\begin{proof}
We will show $(i)$. The remaining cases can be shown in a similar way. First, we consider the case $z_1  \in B_{\beta/2}$, where $|z_1 | \le
(16\pi\nu)^{-\nu} (\beta/2)^\nu < |z_2|/ 2^\nu$. Then we have that
\[
|z_2 - z_1| \gtrsim |z_2| > (16\pi\nu)^{-\frac12} \beta^{\frac12} |z_2|^{-\frac k 6}
\]
since $ |z_2 | > (16 \pi \nu)^{-\nu}\beta^\nu = (16 \pi \nu)^{-3/(k+6)}\beta^{3/(k+6)}$.

If $ z_1  \notin B_{\beta/2}$, then $|z_1| > (16 \pi \nu)^{-\nu}(\beta/2)^\nu$, which also implies $|z_1| > (32\pi\nu)^{-\frac12} \beta^{\frac12}
|z_1|^{-\frac k 6}$. Let us define a set $B_{\beta}(w) $ for $ w \in B$ and $c_0 >0$ by
\[
B_{\beta}(w)  = \{ z \in B : |z-w| \le c_0 \,\beta^{\frac12}|w|^{-\frac k6} \}.
\]
When $z \in B_\beta(z_1)$ for $ z_1 \notin B_{\beta/2}$, we see that
\[
|z - z_1| \le c_0 \,\beta^{\frac12}|z_1|^{-\frac k6}
\le c_0 (32\pi\nu)^{\frac12} |z_1|,
\]
and then
\[
|z| \le |z - z_1| + |z_1| \le(1+ c_0 (32\pi\nu)^{\frac12} ) |z_1|.
\]
It follows that
\begin{align*}
\sigma(B_{\beta}(z_1))
= \int_{B_{\beta}(z_1)} |z|^{\frac{k}{3}} d\mu(z)
 \le \left((1+ c_0 (32\pi\nu)^{\frac12}) |z_1| \right)^{\frac{k}{3}} \times \pi c_0^2 \beta |z_1|^{-\frac k 3}.
% =: c_1 \beta.
\end{align*}
By choosing sufficiently small $c_0$, we have that $\sigma(Q_{z_1} \setminus B_\beta(z_1)) \gtrsim \beta$. Thus we can regard $Q_{z_1} \setminus
B_\beta(z_1)$ as $Q_{z_1}$, and we can say (i) holds for $z_1 \in P$ and $z_2 \in Q_{z_1}$.
\end{proof}

Now we turn to obtaining a lower bound of $E'$. We may assume that $\mathcal S$ can be replaced by a suitable subset of $\mathcal S$, where $z_1,z_2,z_3$
satisfy the observations above. By the above lemma, the Vandermonde determinant $|z_2-z_1| |z_3 - z_1| |z_3-z_2|$ can be treated in each case. In fact, we
have
\[
|z_2- z_1|^2|z_3-z_1|^2|z_3-z_2|^2 \gtrsim \beta \alpha_2^2
|z_1|^{-\frac k 3(2- \varepsilon_1 -\varepsilon_2)} |z_2|^{-\frac k 3 ( \varepsilon_1 + 1-\varepsilon_3)} |z_3|^{ -\frac k 3 (\varepsilon_2 +
\varepsilon_3)}.
\]

Also it is obvious that
\begin{align}\label{torsion}
\max \{|z_1|,|z_2|,|z_3|\}^{2k} \ge |z_1|^a |z_2|^b |z_3|^c
\end{align}
for positive constants $a$,$b$,$c$ satisfying $a + b+c=2k$.

Therefore, if we set
\begin{align*}
a  = \frac k 3  + \frac k 3 ( 2- \varepsilon_1 - \varepsilon_2),
 \quad b  = \frac k 3 + \frac k 3 ( \varepsilon_1 +  1-\varepsilon_3) ,
\textrm{ and } \,c  = \frac k 3 + \frac k 3 (\varepsilon_2 + \varepsilon_3),
\end{align*}
we have
\begin{align*}
|E_2|
&\ge C \iiint_{\mathcal S} \max\{|z_1|,|z_2|,|z_3|\}^{2k} |z_2-z_1|^2|z_3-z_1|^2|z_3-z_2|^2 d\mu(z_1)d\mu(z_2)d\mu(z_3)\\
&\gtrsim \beta \alpha_2^2
\int_P |z_1|^{\frac k3} \int_{Q_{z_1}} |z_2|^{\frac k3} \int_{R_{z_1,z_2}} |z_3|^{\frac k3} d\mu(z_3) d\mu(z_2) d\mu(z_1) \\
&\gtrsim \alpha_1 \beta^2 \alpha_2^3.
\end{align*}
This completes the proof.
\end{proof}

To obtain sharper Lorentz space esimates, we need another refinement of \eqref{Eab} as follows.
\begin{lemma}\label{trilinear_3d}
Let $F_1,F_2, G \subset \mathbb R^6$ be measurable sets with finite measure. Suppose that
\begin{align*}
\TTchi{F_1}(y) \ge \beta_1 \textrm{   and   } \TTchi{F_2} (x) \ge \beta_2
\end{align*}
for all $y \in G$ and $\beta_1 \le \beta_2$. Then
\begin{align*}
| F_2 | \gtrsim \alpha^3 \,\beta_1\, \beta_2^2,
\end{align*}
where $\alpha = \beta_1 \frac{|G|}{|F_1|}$.
\end{lemma}

\begin{proof}
Following the proof of Lemma \ref{trilinearE}, we have that
\begin{align*}
F_1^1 &= \{ x \in F_1 : \Tchi{G} (x) \ge  \alpha/2 \} \\
G^1  & = \{ y \in G : \TTchi{F_1^1} (y) \ge  \beta_1/4 \}.
\end{align*}
Then we can construct the sets contained in $B$ as follows.
For $y_0 \in G^1$, we define
\begin{align*}
&P' = \{z_1 \in B : y_0 + h(z_1) \in F_1^1 \},  &(\sigma(P') \ge  \beta_1/4), \\
&Q'_{z_1} = \{ z_2 \in B : y_0 + h(z_1) - h(z_2) \in G \}, &(\sigma(Q'_{z_1}) \ge  \alpha/2), \\
&R'_{z_1,z_2} = \{ z_3 \in B : y_0 + h(z_1) - h(z_2) + h(z_3) \in F_2 \}, &(\sigma(R'_{z_1,z_2}) \ge \beta_2).
\end{align*}
If we set $\mathcal S' = \{ (z_1,z_2,z_3) : z_1 \in P', z_2 \in Q'_{z_1}, z_3 \in R'_{z_1,z_2} \}$ and $\Phi'_h (z_1,z_2,z_3) = h(z_1) - h(z_2) + h(z_3)$,
then $y_0 + \Phi'_h (\mathcal S') \subset F_2 $. Again we have that
\[
|F_2| \ge C \iiint_{\mathcal S'} \max\{|z_1|,|z_2|,|z_3|\}^{2k} |z_2-z_1|^2|z_3-z_1|^2|z_3-z_2|^2 d\mu(z_1)d\mu(z_2)d\mu(z_3)
\]

We also have an analogue of Lemma \ref{separation}:
\begin{enumerate}[(i)]
\item $|z_1| \ge (16\pi\nu)^{-\nu}\beta_1^\nu$, $|z_2| \ge (16\pi\nu)^{-\nu}\alpha^\nu$, and $|z_3| \ge (16\pi\nu)^{-\nu}\beta_2^\nu$.
\item $| z_2 - z_1 | \ge c \alpha^{\frac 12} |z_2|^{-\frac k 6 \varepsilon_1 } |z_1|^{-\frac k 6 (1-\varepsilon_1)}$, where $\varepsilon_1 = 0$ if $ z_1 \in
B_{\alpha_1/2}$, or $\varepsilon_1 = 1$ if $ z_1 \notin B_{\alpha_1/2}$.
\item $| z_3 - z_1 | \ge c \beta_2^{\frac 12} |z_3|^{-\frac k 6 \varepsilon_2 } |z_1|^{-\frac k 6 (1-\varepsilon_2)}$, where $\varepsilon_2 = 0$ if $ z_1 \in
B_{\beta_2/2}$, and $\varepsilon_2 = 1$ if $ z_1 \notin B_{\beta_2/2}$.
\item $| z_3 - z_2 | \ge c \beta_2^{\frac 12} |z_3|^{-\frac k 6 \varepsilon_3 } |z_2|^{-\frac k 6 (1-\varepsilon_3)}$, where $\varepsilon_3 = 0$ if $ z_2 \in
B_{\beta_2/2}$, and $\varepsilon_3 = 1$ if $ z_2 \notin B_{\beta_2/2}$.
\end{enumerate}
where $\nu = 3/ (k+6)$.

We consider two cases: $\beta_2 \gtrsim \alpha$ and $\beta_2 \ll \alpha$.

{\sl Case 1 :  $\beta_2 \gtrsim \alpha$.}
In this case, we follow the proof of Lemma \ref{trilinearE}. Then we obtain
\[
|F_2| \gtrsim \alpha^2 \beta_1 \beta_2^3 \gtrsim \alpha^3 \beta_1 \beta_2^2
\]
by (i)--(iv) above.

{\sl Case 2 : $\beta_2 \ll \alpha$.} In this case, we assume that $z_1 \in B_{\alpha/2}$. Then $|z_2 - z_1 |\ge c \alpha^{\frac12} |z_1|^{-\frac k 6}$.
(When $z_1 \notin B_{\alpha/2}$, we start with $ |z_2 - z_1 |\ge c \alpha^{\frac12} |z_2|^{-\frac k 6}$.)
We consider two balls given by
\begin{align*}
B(z_1) &= \{ z : |z-z_1| < \frac13 c \alpha^{\frac12} |z_1|^{-\frac k 6} \} \\
B(z_2) &= \{ z : |z-z_2| < \frac13 c \alpha^{\frac12} |z_1|^{-\frac k 6} \}.
\end{align*}
Then $z_3$ can be located in $B(z_1)$, $B(z_2)$, or $(B(z_1)\cup B(z_2))^c$. If $z_3 \in B(z_1)$, then  (iv)  can be replaced by
\[
|z_3 - z_2| \geq \frac13 c \alpha^{\frac12} |z_1|^{-\frac k 6}.
\]
By \eqref{torsion} with $a = k + \frac {k}{3}(1-\varepsilon_2)$, $b= \frac k 3$, $c=\frac k3+\frac {k}{3}\varepsilon_2$, and by (i)--(iii) above, it follows
that
\begin{align*}
|F_2| &\ge C \iiint_{\mathcal S'} \max\{|z_1|,|z_2|,|z_3|\}^{2k} |z_2-z_1|^2|z_3-z_1|^2|z_3-z_2|^2 d\mu(z_1)d\mu(z_2)d\mu(z_3) \\
& \gtrsim \iiint_{\mathcal S'} |z_1|^a |z_2|^b |z_3|^{c}
\alpha^2 \beta_2 |z_1|^{-\frac{2k}{3} - \frac k 3(1-\varepsilon_2)} |z_3|^{-\frac k 3 \varepsilon_2} d\mu(z_1)d\mu(z_2)d\mu(z_3) \\
& \gtrsim \alpha^2 \beta_2 \iiint_{\mathcal S'} |z_1|^{\frac k 3} |z_2|^{\frac k 3} |z_3|^{\frac k 3} d\mu(z_1)d\mu(z_2)d\mu(z_3) \\
& \gtrsim \alpha_1^3  \beta_1 \beta_2^2.
\end{align*}

If $z_3 \in B(z_2)$, then we replace (iii) by $|z_3 - z_1| \geq \frac13 c \alpha^{\frac12} |z_1|^{-\frac k 6}$. By choosing appropriate $a,b,c$, we obtain
$|F_2| \gtrsim \alpha^3 \beta_1 \beta_2$ again.

If $z_3 \in (B(z_1)\cup B(z_2))^c$, we have $|z_3 - z_1| \geq \frac13 c \alpha^{\frac12} |z_1|^{-\frac k 6}$ and $|z_3 - z_2| \geq \frac13 c \alpha^{\frac12} |z_1|^{-\frac k 6}$.
Thus we see that $|F_2|\gtrsim \alpha^4 \beta_1 \beta_2 \gtrsim \alpha^3 \beta_1 \beta_2^2$ as desired.
\end{proof}

%______________________________
\appendix
\section{Proof of Theorems \ref{thm:monomial} and \ref{thm:poly}}
In this section,  for the sake of completeness  we present a detailed proof of \eqref{main_inq:mono} assuming Lemmas \ref{d-trilinear} and  \ref{trilinearF}.
(Theorem 1.2 is implied by Lemmas \ref{trilinearE} and \ref{trilinear_3d}.)
We  closely  follow the argument due to Stovall \cite{S1}. (See also \cite{Ch2}.)
%To obtain Lorentz type estimates we first establish a weaker version that $T : L^{p,u}(\mathbb R^{2d}) \rightarrow L^{q,\infty}(\mathbb R^{2d})$ as follows:
We begin with establishing a weaker version $T : L^{p,u}(\mathbb R^{2d}) \rightarrow L^{q,\infty}(\mathbb R^{2d})$, which implies the weak type $(p,q)$. By the argument in \cite{Ch2}, the weak type $(p,q)$ gives the Lorentz space boundedness of $T : L^{p}(\mathbb R^{2d}) \rightarrow L^{q,p+\epsilon}(\mathbb R^{2d})$ for any $\epsilon >0$.
\begin{lemma}\label{weak_ver}
Let $ 1 \leq u < q = \frac{d(d+1)}{2(d-1)}$ and $p = \frac{d+1}{2} $. Also let $F$ be a (Borel) measurable set and $f \in L^{p,u}$. Then there exists a
constant $C > 0$, depending only on $p,q,u$, such that
\begin{equation*}
\left< Tf, \chi_F \right> \leq C \|f\|_{L^{p, u}} |F|^{\frac{1}{q'}}.
\end{equation*}
Here, $|F|$ is the Lebesgue measure of $F$ on $\mathbb R^{2d}$ and $1/q' = 1 - 1/q$.
\end{lemma}
\begin{proof}
Let us set $f=\sum_{k=-\infty}^\infty 2^k \chi_{E_k}$  where the $E_k$'s are pairwise disjoint measurable sets in $\mathbb R^{2d}$. Let us assume that
\[
\nnorm{f}_{L^{p ,u}} \sim  \left(\sum_{k=-\infty}^\infty (2^{k } \abs{E_k}^{\frac{1}{p}})^u \right)^{\frac{1}{u}} = 1.
\]
% $\nnorm{\sum_{k=-\infty}^\infty 2^k \chi_{E_k}}_{L^{2,u}}\sim  (\sum_{k=-\infty}^\infty (2^k \abs{E_k}^{1/2})^u)^{1/u} \sim 1 $.
Then it suffices to show that
\begin{equation}\label{weaktype}
\sum_{k \in \mathbb Z} 2^k \left< \Tchi{E_k},\chi_F \right> \lesssim |F|^{\frac{1}{q'}}.
\end{equation}
%where $1/q' = 1- 1/q$.

Let $\mathscr T(E_k,F)$ denote $\left< \Tchi{E_k}, \chi_F \right> $.
We classify $E_k$ depending on the restricted weak type estimate obtained in Section 3 and the normalization of $f$ above.
For nonnegative integers $m$ and $n$, we define
\begin{equation}\label{k1}
\mathscr K_0^F=\{ k \in \mathbb Z : \mathscr T(E_k ,F)=0\},
\end{equation}
\begin{equation}\label{k2}
\mathscr K_m^F=\{ k \in \mathbb Z : C 2^{-m-1} \abs{E_k}^{\frac{1}{p}} \abs{F}^{\frac{1}{q'}} \le \mathscr T(E_k ,F) \le C 2^{-m} \abs{E_k}^{\frac{1}{p}}
\abs{F}^{\frac{1}{q'}} \},
\end{equation}
\begin{equation}\label{k3}
  \mathscr K_{m, n}^F = \{k \in \mathscr K_m^F : 2^{-n-1} <(2^{k} \abs{E_k}^{\frac{1}{p}})^{u} \le 2^{-n} \}.
\end{equation}
Here the constant $C$ arose from the restricted weak type estimates for $T$.% \eqref{TchiEF}.

We split $\mathscr K_{m,n}^F$ into $A m$-separated sets so we define $\{ \mathscr K_{m, n, i}^F \}_{i=1}^{[A m]}$ to be a partition of $\mathscr K_{m,
n}^F$.
Here $A$ is some constant that will be determined later.
Note that $|k -k'|\ge A m$ for any $k,k' \in  \mathscr K_{m, n, i}^F$.
Fix $i$ and set $ \mathscr K = \mathscr K_{m, n, i}^F$ for convenience.

By the normalization of $f$, we have $\sum_{k \in \mathscr K} (2^k \abs{E_k}^{1/p})^{u} \le 1$.
We may assume that $\#\mathscr K>0$.
  By \eqref{k3}, we see that  $\sum_{k \in \mathscr K} 2^{-n}\lesssim 1$, so we obtain $\# \mathscr K \lesssim 2^{n}$.
Now we will find two upper bounds for $\sum{2^k \mathscr T(E_k, F)}$.

Firstly, we obtain that
\begin{align}\label{bound_1}
\sum_{k \in \mathscr K}2^k \mathscr T(E_k,F)  \lesssim  & \sum_{k \in \mathscr K}2^k 2^{-m} \abs{E_{k}}^{\frac{1}{p}}\abs{F}^{\frac{1}{q'}} \\
 \lesssim & \sum_{k \in \mathscr K}  2^{-m} 2^{-\frac{n}{u}} \abs{F}^{\frac{1}{q'}} \nonumber\\
 \lesssim & (\# \mathscr K) 2^{-\frac{n}{u}-m}\abs{F}^{\frac{1}{q'}} \nonumber\\
 \lesssim & 2^{n(1-\frac 1u) -m }\abs{F}^{\frac{1}{q'}}\nonumber
\end{align}
since $\mathscr K$ is a subset of \eqref{k2} and \eqref{k3}.

Secondly, we consider a subset of $F$ related to an average of $\Tchi{E_k}$ on $F$.
%we define a special set $G_k$ for given $k \in \mathbb{Z}$ which is similar to $F$.
For each $k \in \mathscr K$, let
\begin{equation}\label{G_k}
G_{k} =\{x \in F : T \chi_{E_k}(x) \geq \frac{\mathscr T(E_{k},F)}{2\abs{F}}\}.
\end{equation}
It follows that
\begin{align*}
\mathscr T(E_{k}, F \backslash G_{k})
< \frac{\mathscr T(E_k,F)}{2|F|} |F \setminus G_k|
\le \frac12 \mathscr T(E_k,F),
\end{align*}
and then $\mathscr T(E_{k},F) \sim \mathscr T(E_{k},G_{k})$.
%= \int_{\chi_{F \backslash G_{k}}}T \chi_{E_{k}}(x) \leq \frac{1}{2} \int_{\chi_{F \backslash G_{k}}}\frac{\mathscr T(E_{k}, F)}{\abs{F}}
%\leq \frac{1}{2} \int_{\chi_{F}}\frac{\mathscr T(E_{k}, F)}{\abs{F}} =\frac{1}{2}\mathscr T(E_{k},F).
%\end{align*}

%So we get that $\mathscr T(E_{k},F) \sim \mathscr T(E_{k},G_{k})$.

Since
\[
C2^{-(m-1)}|E_k|^{\frac{1}{p}} |F|^{\frac{1}{q'}} \le \mathscr T(E_k,F)  \sim \mathscr T(E_{k},G_{k}) \le  C|E_k|^{\frac{1}{p}}|G_k|^{\frac{1}{q'}},
\]
we have
\begin{align}\label{FG_k}
 2^{-mq'} |F| \lesssim |G_k|.
\end{align}
Also we can observe that
\[
\sum_{k \in \mathscr K} |G_k| = \sum_{k \in \mathscr K} \int_F \chi_{G_k} \le \left( \int_F \Big(\sum_{k \in \mathscr K} \chi_{G_k}\Big)^2 \right)^{\frac12}
|F|^{\frac12}.
\]
It follows that
\[
\left( \frac{1}{|F|} \sum_{k \in \mathscr K} |G_k| \right)^2 \le \frac{1}{|F|} \int_F (\sum_{k \in \mathscr K} \chi_{G_k})^2 \le \frac{1}{|F|}\sum_{k \in
\mathscr K} |G_k| + \frac{1}{|F|}\sum_{k\neq\ell}|G_k\cap G_\ell|.
\]
Therefore we have two cases, i.e.
\[
\frac{1}{|F|} \sum_{k \in \mathscr K} |G_k| \lesssim 1 \quad\textrm{or}\quad
\left(  \frac{1}{|F|} \sum_{k \in \mathscr K} |G_k| \right)^2 \lesssim \frac{1}{|F|}\sum_{k\neq\ell}|G_k\cap G_\ell|.
\]
From the latter inequality and \eqref{FG_k}, it holds that
\begin{align} \label{contradiction}
|G_k \cap G_\ell| \gtrsim  2^{-2mq'} |F|
\end{align}
for some $k\neq\ell$. In fact, one can see that $$|F|^{-1} \sum_{k\neq\ell} |G_k \cap G_\ell| \le |F|^{-1} (\# \mathscr K)^2 \max_{k\neq\ell}|G_k \cap
G_\ell|$$ and $$(|F|^{-1} \sum_{ k \in \mathscr K} |G_k|)^2 \gtrsim ( |F|^{-1}  2^{-mq'} |F| (\# \mathscr K))^2$$ by \eqref{FG_k}. So we obtain
$\max_{k\neq\ell}|G_k \cap G_\ell| \gtrsim  2^{-2mq'} |F|$.

Our claim is that \eqref{contradiction} yields a contradiction by Lemma \ref{trilinearF} or \ref{trilinearE}. We postpone the proof for a moment.
%until the last part of this section using Lemma \ref{b>a}.

%By the Lemma \ref{trilinearE},
Then we may assume that $\sum_{k \in \mathscr K} \abs{G_k} \lesssim \abs{F}$. By this inequality, the second bound for $\sum 2^k \mathscr T(E_k,F)$ can be
derived from
\begin{equation}\label{bound_2}
\begin{aligned}
\sum_{k \in \mathscr K} 2^{k}\mathscr T(E_{k},G_{k}) \lesssim & \sum_{k \in \mathscr K} 2^{k} \abs{E_k}^{\frac{1}{p}} \abs{G_{k}}^{\frac{1}{q'}} \\
\lesssim & \left(\sum_{k \in \mathscr K}(2^{k} \abs{E_k}^{\frac{1}{p}})^{q} \right)^{\frac{1}{q}} \left(\sum_{k \in \mathscr K}\abs{G_{k}} \right)^{\frac{1}{q'}}
\\
\lesssim & (\# \mathscr K 2^{-\frac{ nq }{u}})^{\frac{1}{q}}\abs{F}^{\frac{1}{q'}} \\
\lesssim & 2^{n(\frac{1}{q}-\frac 1 u)}\abs{F}^{\frac{1}{q'}}.
\end{aligned}
\end{equation}
By \eqref{bound_1} and \eqref{bound_2}, we obtain that
\begin{equation}\label{twobounds}
\sum_{k \in \mathscr K} {2^{k}\mathscr T(E_{k},F)} \lesssim \min\left(2^{n(1-\frac 1u) -m }, 2^{n(\frac{1}{q}-\frac 1 u)}\right)\abs{F}^{\frac{1}{q'}}.
\end{equation}
Since $1 \leq u < q$, we have
\begin{equation}\label{bound_final}
\begin{aligned}
\sum_{k \in \mathscr K_{m}^{F}} 2^{k} \mathscr T(E_{k},F)= &
\sum_{n=0}^{\infty} \sum_{i=1}^{\lceil A m\rceil} \sum_{k \in \mathscr K}2^{k} \mathscr T (E_{k}, F) \\
\lesssim & Am \sum_{n=0}^{\infty} \min\left(2^{n(1-\frac 1u) -m }, 2^{n(\frac{1}{q}-\frac 1u)}\right)\abs{F}^{\frac{1}{q'}}  \\
\lesssim & Am \left( \sum_{n > \lceil m q' \rceil} 2^{n(\frac{1}{q}-\frac 1u)} + \sum_{n \le \lceil mq' \rceil} 2^{n(1-\frac 1u) -m } \right)
\abs{F}^{\frac{1}{q'}} \\
\lesssim & Am\, 2^{-mq'(\frac{1}{u} - \frac 1q) } \abs{F}^{\frac{1}{q'}}.
\end{aligned}
\end{equation}
Let us set $0 \leq \varepsilon < q'(\frac{1}{u} - \frac 1q)$. Since there exists a constant $c$ such that $m \le 2^{ \varepsilon m}$ for $m \ge c$, we sum
\eqref{bound_final} over $0 \le m < \infty$ as follows:
\begin{align*}
\sum_{m=0}^\infty\sum_{k \in \mathscr K_{m}^{F}} 2^{k} \mathscr T(E_{k},F)=\sum_{ 0 \le m < c} Am\, 2^{-mq'(\frac{1}{u} - \frac 1q) } \abs{F}^{\frac{1}{q'}}
+ \sum_{ m \ge c} A 2^{-m( q'(\frac{1}{u} - \frac 1q) -\varepsilon) } \abs{F}^{\frac{1}{q'}} \lesssim |F|^{\frac 1 {q'}}
\end{align*}
This gives the desired inequality \eqref{weaktype}.

Now we turn to proving that a contradiction occurs if we assume \eqref{contradiction}.
%For the contradiction we assume that \eqref{contradiction} is valid.
Let us denote $G = G_k \cap G_\ell$ where $k>l$, $E_1=E_k$, $E_2=E_\ell$, $\alpha_1 = 2^{-m} |E_1|^{\frac{1}{p}} |F|^{-\frac{1}{q}}$, $\beta = \alpha_1
|G||E_1|^{-1}$, and $\alpha_2 = 2^{-m} |E_2|^{\frac{1}{p}}|F|^{-\frac{1}{q}}$.
Then, by Lemma \ref{d-trilinear}(or Lemma \ref{trilinearE}), %\eqref{FG_k}
and \eqref{contradiction}, we obtain that
\begin{align*}
| E_2| & \gtrsim \alpha_1^{\frac{d(d+1)}{2}} \,\left( \frac{\beta}{\alpha_1}\right)^{d-1} \, \left(\frac{\alpha_2}{\alpha_1}\right)^{d}  \\
&\gtrsim (2^{-m} |E_1|^{\frac 1{p}} |F|^{-\frac{1}{q}})^{\frac{d(d+1)}{2}} ( 2^{-2m\,q'}  |F| |E_1|^{-1})^{d-1} (|E_2|^{\frac{1}{p}} |E_1|^{-\frac{1}{p}})^d\\
& \gtrsim 2^{-m ( \frac{d(d+1)}{2}+ 2 (d-1) q'  )} |E_1|^{1-\frac{d}{p}} |E_2|^{\frac{d}{p}}.
\end{align*}
It follows that $|E_1| \gtrsim 2^{-m(d+1)(q +  2q' ) } |E_2|$.
Together with the fact $|E_1| \sim 2^{-np/u-kp}$ and $|E_2| \sim 2^{-np/u-\ell p}$ by \eqref{k3},
we have that $2^{m(d+1)(q +  2q' )} \gtrsim 2^{(k-\ell)p}$.
Since $k>\ell$, and the case $k <\ell$ can be obtained in a similar way, we finally obtain that $|k-\ell| \lesssim m(d+1)(q +  2q' )/p$. Since we can take the
constant $A$ to be sufficiently large, this contradicts our construction of $\mathscr K$. This completes the proof.
\end{proof}

%\begin{corollary}
%dd
%\end{corollary}

Recall that $p = \frac{d+1}{2}$, $q  = \frac{d(d+1)}{2(d-1)}$. The following lemma implies Theorem \ref{thm:monomial} and Theorem \ref{thm:poly}.
In fact, the other three cases ($u \le p < v \le q$, $p \le u < q \le v$, and $u \le p < q \le v$) follow from the next lemma and the fact that $\|f \|_{L^{p,u}} \le \| f \|_{L^{p,u'}}$ whenever $ u' \le u$.

%The necessary conditions of Theorem \ref{main_inq:mono} and \ref{main_inq:poly} are little complicated like $1 \leq u < q_d$, $p < v \leq \infty$, and $u<v$.
%So We use $p$ and $q$ instead of $p$ and $q_d$ such that $1 \le p \le u <v \le q \le \infty$.
%By the property of Lorenz space norm, $\nnorm{f}_{L^{p,}}$

\begin{lemma}
Let $1 \le p \le u <v \le q \le \infty$. For $f \in L^{p,u}$ there exists a constant $C >0$ such that
\begin{equation*}
\| T f \|_{L^{q,v}(\mathbb R^{2d})} \leq C \|f\|_{L^{p,u}(\mathbb R^{2d})}.
\end{equation*}
\end{lemma}

%{\sl Strong type estimates.} Finally, we will do similarly the previous method to get the result  \eqref{main_inq:mono} or \eqref{main_inq:poly}.
\begin{proof}
Let $f= \sum_{k}2^{k} \chi_{E_{k}}$ and $g=\sum_{j}2^{j}\chi_{F_{j}}$ where $E_{k}$ and $F_{j}$ are pairwise disjoint measurable sets. We also assume that
\[
\nnorm{f}_{L^{p ,u}} \sim \left(\sum_{k=-\infty}^\infty (2^{k } \abs{E_k}^{\frac{1}{p}})^u \right)^{\frac{1}{u}} = 1
\]
and
\[
\nnorm{g}_{L^{q' ,v'}} \sim \left(\sum_{j=-\infty}^\infty (2^{j } \abs{F_j}^{\frac{1}{q'}})^{v'} \right)^{\frac{1}{v'}} = 1.
\]
%$T$ maps the Lorentz space $L^{2,u}$ boundedly into $L^{3,v}$ and $L^{3/2, v'}$ into $L^{2, u'}$ whenever $1 \le u<3$, $2<v \le \infty$, and $u<v$.
Then it suffices to show that
\[
\sum_{j\in \mathbb Z} 2^j \sum_{k\in \mathbb Z} 2^k \mathscr T(E_k, F_j) \lesssim 1,
\]
where $\mathscr T(E_k, F_j) =  \langle T\chi_{E_k}, \chi_{F_j} \rangle $.

%We will prove that
%\begin{equation}
%\left< Tf, g \right> \lesssim \nnorm{f}_{L^{p,u}}\nnorm{g}_{L^{q', v'}},
%\end{equation}
%where $1 \leq p \leq u < v \leq q \leq \infty$

Fix a nonnegative integer $m$, and we set $\mathscr K_0^{F_j}$ as above and
\begin{align*}
\mathscr K_{m}^{F_{j}} = \{ k \in \mathbb Z : c 2^{-m-1} \abs{E_{k}}^{\frac{1}{p}}\abs{F_{j}}^{\frac{1}{q'}} < \mathscr T(E_{k}, F_{j}) \le c 2^{-m}
\abs{E_{k}}^{\frac{1}{p}}\abs{F_{j}}^{\frac{1}{q'}} \}
\end{align*}
for some constant $c >0$.
Instead of $\mathscr K_{m,n}^{F_j}$, we define
\begin{align}\label{J_n}
\mathscr J_{n} = \{j \in \mathbb Z : 2^{-n-1} < (2^j \abs{F_j}^{1/q'} )^{v'} \le 2^{-n} \}
\end{align}
from the normalization of $g$.
We will divide $\mathscr J_{n}$ into $\lceil A'{m}\rceil$ subsets, where $A'$ will be specified later.
For $1 \le i \le [A' {m}]$, let us denote by $\mathscr J_{n,i}^m $ an $A' m$-separated subset of $\mathscr J_{n}$.
Let $\mathscr J=\mathscr J_{n,i}^{m}$ for convenience.
Using the method used to get $\# \mathscr K$, one can see that $\# \mathscr J \leq 2^{n}$ for each $\mathscr J$.

%We will calculate two bounds for $\sum_{j \in \mathscr J}2^{j} \sum_{k \in \mathscr K_{m}^{F_{j}}}\mathscr T(E_{k}, F_{j})$ similarly.
%We use \eqref{bound_final} to get the first bound like this,
By \eqref{bound_final}, we have that
\begin{equation*}
\begin{aligned}
\sum_{j \in \mathscr J}2^{j} \sum_{k \in \mathscr K_{m}^{F_{j}}}2^{k}\mathscr T (E_{k}, F_{j}) &\lesssim Am\, 2^{-mq'(\frac 1u - \frac 1q) } \sum_{j \in
\mathscr J}2^j \abs{F_{j}}^{1/q'}\\
&\lesssim Am\,2^{-mq'(\frac 1u - \frac 1q) } 2^{-n/v'} \# \mathscr J \\
&\lesssim Am\, 2^{-mq'(\frac 1u - \frac 1q) } 2^{n(1- \frac 1{v'} )}.
\end{aligned}
\end{equation*}
%for $a=\frac{1}{q -1}(\frac{q}{u}-1)-\frac{\varepsilon}{m}>0$.

For the second bound, we consider
\begin{align}\label{E_kj}
E_{k,j}= \{x \in E_{k} : T^\ast{\chi_{F_{j}}}(x) \ge \frac{\mathscr T(E_{k}, F_{j})}{2\abs{E_{k}}}\},
\end{align}
which implies that
\begin{align*}
\mathscr T (E_k \backslash E_{k,j}, F_{j}) \le \langle \chi_{E_k \setminus E_{k,j}}, T^\ast \chi_{F_j} \rangle <  \frac{1}{2} \mathscr T (E_{k,j}, F_{j}).
\end{align*}
Therefore, $\mathscr T(E_{k}, F_{j}) \sim \mathscr T (E_{k,j},F_{j})$.
It follows that $2^{-mp} |E_k| \lesssim |E_{k,j}|$, which is similar to \eqref{FG_k}. By the same argument, we get two cases:
\[
\frac{1}{|E_k|} \sum_{j \in \mathscr J} |E_{k,j}| \lesssim 1 \quad\textrm{or}\quad
\left(  \frac{1}{|E_k|} \sum_{j \in \mathscr J} |E_{k,j}| \right)^2 \lesssim \frac{1}{|E_k|}\sum_{j\neq l}|E_{k,j}\cap E_{k,l}|.
\]

Again, we may assume that $\sum_{j \in \mathscr J} \abs{E_{k,j}} \le \abs{E_{k}}$ for each $k \in \mathscr K_m^{F_j}$. In fact, the second inequality and
$2^{-mp}|E_k| \lesssim |E_{k,j}|$ lead to $ |E_{k,j} \cap E_{k,l} | \gtrsim 2^{-2mp} |E_k|$, which implies a contradiction to the definition of $\mathscr
J$. %using the same argument as obtaining \eqref{contradiction}.
We will check this at the end.

By the assumption $u/p \geq 1$, it follows that
\begin{align*}
\nonumber \sum_{j \in \mathscr J} \abs{E_{k,j}}^{\frac up} \le
 \left( \sum_{j \in \mathscr J} \abs{E_{k,j}}\right)^{\frac u p} \leq \abs{E_{k}}^{\frac up},
 \end{align*}
and then
\begin{equation}\label{j-sum}
\sum_{j \in \mathscr J}\sum_{k \in \mathscr K_{m}^{F_j}}2^{uk} \abs{E_{k,j}}^{\frac up} \le \sum_{k \in \mathbb Z} 2^{uk} \abs{E_{k}}^{\frac up}.
\end{equation}
By Lemma \ref{weak_ver}, H\"{o}lder's inequality, and \eqref{j-sum}, we thus obtain that
\begin{align*}
\sum_{j \in \mathscr J}2^j \sum_{k \in \mathscr K_{m}^{F_{j}}}2^{k} \mathscr T (E_{k}, F_{j}) &\sim \sum_{j \in \mathscr J}2^j \sum_{k \in \mathscr
K_{m}^{F_{j}}}2^{k} \mathscr T (E_{k,j}, F_{j}) \\
& \lesssim \sum_{j \in \mathscr J}2^{j} \left(\sum_{k \in \mathscr K_{m}^{F_{j}}} \left(2^{k}\abs{E_{k,j}}^{\frac 1p}\right)^{u} \right)^{\frac
1u}\abs{F_{j}}^{\frac 1{q'}}. \\
& \lesssim \left(\sum_{j\in\mathscr J} \sum_{k \in \mathscr K_{m}^{F_{j}}} 2^{ku} \abs{E_{k,j}}^{\frac u p}  \right) ^{\frac 1u} \left(\sum_{j\in
\mathscr J}(2^{j}\abs{F_{j}}^{\frac1{q'}})^{u'} \right)^{\frac1{u'}} \\
& \lesssim \left( \sum_{k\in \mathbb Z} 2^{ku}\abs{E_{k}}^{\frac up}\right)^{\frac 1u} (\# \mathscr J 2^{-\frac{n }{v'}u'}
)^{\frac 1{u'}}\\
% & \lesssim \left( \sum_{k\in \mathbb Z} 2^{ku}\abs{E_{k}}^{\frac up}\right)^{\frac 1u} 2^{n(\frac1{u'}-\frac1{v'})} \\
 & \lesssim 2^{n(\frac1{u'}-\frac1{v'})}.
\end{align*}

Now, let $a = q' (1/u - 1/q)$. Then we have
\begin{align*}
\left< Tf, g \right> &= \sum_{j,k}2^{j}2^{k}\mathscr T (E_{k}, F_{j}) \\
&= \sum_{n=0}^{\infty}\sum_{m=0}^{\infty}\sum_{i=1}^{A' m} \sum_{j \in \mathscr J}2^{j} \sum_{k \in \mathscr K_{m}^{F_{j}}} 2^{k} \mathscr T(E_{k}, F_{j}) \\& \lesssim\sum_{n=0}^{\infty}\sum_{m=0}^{\infty}\sum_{i=1}^{A' m} \min \{Am\, 2^{-am  } 2^{n(1-\frac 1{v'})}, 2^{n(\frac1{u'}-\frac1{v'})} \} \\
& \lesssim \sum_{m=0}^{\infty} \left(\sum_{n >\lceil aum \rceil}A A'm^2 \,2^{-am + n(1-\frac 1{v'})}+ \sum_{n \leq \lceil aum \rceil} A'm
\,2^{n(\frac1{u'}-\frac1{v'})} \right)\\
& \lesssim \sum_{m=0} AA'm^2 2^{au(\frac1{u'}-\frac1{v'})m} \\
& \lesssim \sum_{m \le \lceil c \rceil} AA'm^2 2^{au(\frac1{u'}-\frac1{v'})m} + \sum_{m \le \lceil c \rceil} AA' 2^{au(\frac1{u'}-\frac1{v'})m + \varepsilon
m } \\
&\lesssim 1
\end{align*}
where we can take a constant $c >0$ such that $m^2 \le 2^{\varepsilon m}$ when $m \ge c$ and $0 \le \varepsilon < au(1/v' - 1/u')$.

%:

Now we show that $ |E_{k,j} \cap E_{k,l} | \gtrsim 2^{-2mp} |E_k|$ %deduce (?? gives?)
leads to  a contradiction in view of Lemma \ref{trilinearF}.
Let us assume $j > l$ and $m\ge1$.
(The case $m=0$  is  excluded in the construction of the set $\mathscr J$.)
To apply Lemma \ref{trilinearF}, let us set $E := E_{k,j} \cap E_{k,l}$, $F_1:= F_j$, $F_2 := F_l$.
By \eqref{E_kj} and by the definition of $\mathscr K_{m}^{F_{j}}$, we have that
\[
\TTchi{F_i} (x) \ge \frac{\mathscr T(E_{k}, F_{i})}{2\abs{E_{k}}} \gtrsim 2^{-m-2} |E_k|^{-\frac{1}{p'}} |F_i|^{\frac{1}{q'}} \,\textrm{ for }\, x \in E,
\]
where $i =1,2$.
Let us set $\beta_i = 2^{-3m} |E_k |^{-1/p'}|F_i|^{1/q'}$  and $\alpha_i : = \beta_i |E| |F_i|^{-1}$ for $i =1,2$, which satisfy the assumption of Lemma \ref{trilinearF}. (Note that $2^{-m-2} \ge 2^{-3m}$ for $m\ge1$. )
%We also set  for $i =1,2$, which satisfy the assumption of Lemma \ref{trilinearF} since $\TTchi{F_i}(x) \ge \beta_i$ on $E$.}

%

Then, by Lemma \ref{trilinearF}, and by the assumption $ |E|=|E_{k,j} \cap E_{k,l} | \gtrsim 2^{-2mp} |E_k|$, we obtain that
\begin{align*}
| F_2| & \gtrsim \alpha_1^{r_1} \alpha_2^{r_2} \beta_1^{s_1} \beta_2^{s_2}\\
& \gtrsim (2^{-3m }|E_k|^{-\frac{1}{p'}} |F_1|^{\frac{1}{q'}})^{r_1+s_1} (2^{-3m }|E_k|^{-\frac{1}{p'}} |F_2|^{\frac{1}{q'}})^{r_2+s_2} |E|^{r_1 + r_2} |F_1|^{-r_1} |F_2|^{-r_2}  \\
& \gtrsim  2^{-\frac{3md(d+1)}{2}-mpd(d-1)} |E_k|^{-\frac{d(d+1)}{2p'} +\frac{d(d-1)}{2}} |F_1|^{\frac{ r_1+s_1}{q'} -r_1} |F_2|^{\frac{r_2 + s_2}{q'} - r_2} \\
& = 2^{-\frac{md(d+1)(d+2)}{2}} |F_1|^{\frac{r_2}{q} - \frac{s_2}{q'} +1} |F_2|^{\frac{s_2}{q'} - \frac{r_2}{q}},
\end{align*}
using the conditions $r_1 + r_2 = d(d-1)/2$, $s_1 +s_2 = d$ in Lemma \ref{trilinearF} and $p=(d+1)/2$, $q = d(d+1)/2(d-1)$.

It follows that
\[
|F_1|^{\frac{s_2}{q'}- \frac {r_2} {q}-1  } \gtrsim 2^{-\frac{md(d+1)(d+2)}{2}}  |F_2|^{\frac{s_2}{q'}- \frac {r_2} {q}-1  } .
\]
Since  we have $|F_1| \sim 2^{-nq'/v'-j q'}$ and $|F_2| \sim 2^{-nq'/v'-l q'}$ by \eqref{J_n}, we have that $2^{C m} \gtrsim 2^{j- l}$.
Here, $C = d(d+1)(d +2)/(2q'(s_2/q' - r_2/q +1))$.
Since the case $l < j $ can be obtained in a similar way, we finally obtain that $|j-l| \lesssim Cm$.
If we take the constant $A'$ to be sufficiently large, this contradicts our construction of $\mathscr J$. %\rpl{This completes the proof.}
\end{proof}

%\end{section}
%%

\section{ The necessary conditions}
We use the notation and terminology in \cite{Ch1, S1}.
We will only treat the nondegenerate case %(?? like that??)
when $h(z) = (z,z^2,\dots,z^{d-1},z^d)$, which is an analogue of the moment curve in the real case.
Then $d \sigma (z) \sim d \mu(z) = dudv$ and we may assume that $\mathcal A f (x)= \int_D f(x - h(z) ) d\mu(z)$.

Let $D_{r}$ be an anisotropic scaling in $\mathbb R^{2d}$ given by
\[
D_r (x_1,y_1,x_2,y_2,\dots,x_d,y_d) = (r x_1, r y_1, r^2 x_2, r^2 y_2,\dots, r^d x_d, r^d y_d).
\]
We also define a ball $B(x,\varepsilon)$ of radius $\varepsilon$ centered at $x$ in $\mathbb R^{2d}$.
Then $B(x,\varepsilon) - h(z)$ for $z \in D$ is an $\varepsilon$-neighborhood
of $-h(z)$, translated by $x$.
Hence let us set $N(x, \varepsilon) := B(x,\varepsilon) - h(z)$ for $z \in D$ so that $ y - h(z) \in N(x, \varepsilon)$ whenever $y \in B(x,\varepsilon)$.

%{\sl Restricted weak type.}
First we show that the restricted weak type $(p,q)$ of $\mathcal A$ may hold only for $(p,q) \in \mathcal R$, where $\mathcal R$ is a trapezoid with vertices $(0,0)$, $(1,1)$, $(1/p_d,1/q_d)$, and $(1- 1/q_d, 1-1/p_d)$.
%Note that we consider the nondegenerate curve $h(z) = (z,z^2,\dots,z^{d-1},z^d)$ and then $d\sigma (z) \sim d\mu(z) =dudv$.
%This area is derived from the nondegenerate case $h(z) = (z,z^2,\dots,z^d)$ i.e. $d\sigma(z) = C\, d\mu(z)$ with $C = d!$.

Let $0< r <1$ be a small constant.
We consider $D_r N(0,1)$ in $\mathbb R^{2d}$.
For each $y \in \mathbb R^{2d}$, one can see that $y - h(rz) = y - D_r h(z) \in D_r N(0,1)$ whenever $y \in D_r B(0,1)$ and $|z| < r^{-1}$.
%In fact, $| (x_i, y_i) - z^i| < \delta^i /2 + \delta^i/2 < \delta^i$.
Hence it follows from the dilation $z \mapsto rz$ that
\begin{align*}
\mathcal A \chi_{D_r N(0,1)} (y)  &= \int_{|z| <1 } \chi_{D_r N(0,1)} (y - h(z)) d\mu(z) \\
& = r^2 \int_{|z| < r^{-1}} \chi_{D_r N(0,1)} ( y - h(r z)) d\mu(z)\\
&\gtrsim r^2 \chi_{D_r B(0,1)}(y),
\end{align*}
which implies $\| \mathcal A \chi_{D_r N(0,1)} \|_{L^{q,\infty}} \gtrsim r^{2+ d(d+1)/q}$.
Since $|D_r N(0,1)|^{1/p} \sim r^{d(d+1)/p}$, we obtain $r^{2+ d(d+1)/q} \lesssim r^{d(d+1)/p}$ from the restricted weak type $(p,q)$ for $\mathcal A$.
Since $0< r <1 $ we get $1 + d(d+1)/(2q) \ge d(d+1)/(2p)$.

Now for $0 < \varepsilon <1$ we consider $B(0,\varepsilon)$ and $N(0,\varepsilon)$ so that $ \| \mathcal A \chi_{N(0,\varepsilon)} \|_{L^{q,\infty}} \gtrsim |B(0,\varepsilon)|^{1/q} = \varepsilon^{2d/q}$.
Also we have $|N(0,\varepsilon)| \sim \varepsilon^{2(d-1)}$.
By the restricted weak type $(p,q)$ for $\mathcal A$, we get $\varepsilon^{2d/q } \lesssim \varepsilon^{2(d-1)/p}$.
Hence, $d/q \ge (d-1)/p$, and duality gives $1+ (d-1)/q \ge d/p$.

Finally, the condition $q \ge p$ follows by the fact that $\mathcal A$ is translation invariant.
In fact, if any nonzero linear operator which is translation invariant is bounded from $L^p(\mathbb R^d)$ to $L^q(\mathbb R^d)$, then $q \ge p$ is necessary. (See Section 2.5.3 in \cite{Gr}.)
% $\| \mathcal A\|_{L^p \rightarrow L^q}$ is translation invariant),\chg{?? (as I changed above?)}

As a result, we can see that $\mathcal A$ is of
restricted weak type $(p,q)$ only if $(p,q) \in \mathcal R$.

{\sl Lorentz space estimates.}
Now we show that if \eqref{main_inq:mono} holds, then $u \le q_d$, $ p_d \le v$, and $u \le v$, where $(p_d,q_d) = (\frac{d+1}2, \frac{d(d+1)}{2(d-1)})$.

$\bullet$ $u \le v$ :
%Let $B(x_j,\delta^j)$ be a $\delta^j$-ball centered at $x_j$ in $\mathbb R^{2d}$.
%Also we set $N (x_j,\delta^j) = \{ y : y \in h(z) + B (x_j,\delta^j)\}$ which is a $\delta^j$-neighborhood of $h(z)$.
%Note that $|N (x_j,\delta^j)| \sim \delta^{2(d-1)j}$.
For a positive integer $M$, we choose $0 < \varepsilon <1$ and $x_j$ for $j=1,\dots,M$ such that $B (x_j,\varepsilon^j)$ are pairwise disjoint and $N (x_j,\varepsilon^j)$ are also
pairwise disjoint.

Let $f(y) = \sum_{j=1}^M \varepsilon^{-2(d-1)j/p_d} \chi_{N (x_j,\varepsilon^j)}(y)$. Then one can see that
\[
\|f\|_{L^{p_d,u}} \sim \left( \sum_{j=1}^M \varepsilon^{-\frac{2(d-1)u}{p_d}j} |N (x_j,\varepsilon^j)|^{\frac u{p_d}} \right)^{\frac1u} \sim M^{\frac 1 u}.
\]
Since $\mathcal A \chi_{N (x_j,\varepsilon^j)} (y) \gtrsim \chi_{B (x_j,\varepsilon^j)}(y)$ and $B(x_j,\varepsilon^j)$ are pairwise disjoint, we obtain that
\begin{align*}
\| \mathcal A f \|_{L^{q_d,v}}
& \gtrsim \Big\| \sum_{j=1}^M \varepsilon^{-\frac {2(d-1)j}{p_d}}  \chi_{B(x_j,\varepsilon^j)} \Big\|_{L^{q_d,v}} \sim \left( \sum_{j=1}^M \varepsilon^{-\frac {2 v(d-1)} {p_d}   j} |B(x_j,\varepsilon^j)|^{\frac v {q_d}}    \right)^{\frac 1v}  \\
& \sim \left( \sum_{j=1}^M \varepsilon^{-2v(\frac {d-1} {p_d} - \frac{d }{q_d}) j}    \right)^{\frac 1v} = M^{\frac 1 v}.
\end{align*}
For the last equality we use the fact that $(p_d,q_d)$ satisfies $d/q_d = (d-1)/p_d$. Therefore \eqref{main_inq:mono} gives
$M^{1/v} \lesssim M^{1/u}$ for any positive integer $M \ge 2$. Hence we obtain $u \le v$.

$\bullet$ $u \le q_d$ : Let $N_{\varepsilon,r}(x) = D_r N(x,\varepsilon)$ and $B_{\varepsilon,r}(x) = D_r B(x,\varepsilon)$. We begin with observing $\mathcal A \chi_{N_{\varepsilon,r}(x)}(y) \ge r^2 \chi_{B_{\varepsilon,r}(x)}(y)$ for some $x$ and $0<\varepsilon,r < 1$. Since $h(rz) = D_r h(z)$, we see that
\begin{align*}
\mathcal A \chi_{N_{\varepsilon,r}(x)}(y) & = \int_{|z| < 1} \chi_{N_{\varepsilon,r}(x)} (y -h(z) ) d\mu(z) = r^2 \int_{|z| < r^{-1}} \chi_{N_{\varepsilon,r}(x)} ( y - h(rz)) d\mu (z) \\
& = r^2 \int_{|z| < r^{-1}} \chi_{N (x,\varepsilon)} (D_r^{-1} y - h(z) ) d\mu(z) \gtrsim r^2 \chi_{B_{\varepsilon, r}(x)}(y).
\end{align*}

Let us set $\varepsilon_j = 2^{-(M +j)p_d}$, $r_j = 2^{-j}$, and choose $x_j$ so that $B_{\varepsilon_j, r_j}(x_j)$ are pairwise disjoint and also $N_{\varepsilon_j, r_j}(x_j)$ are pairwise disjoint.
%(For the procedure ??
(This choice of $\varepsilon_j$ and $r_j$ is borrowed from Section 3 in \cite{S1}.)
If we set $f = \sum_{j=1}^M 2^{2j} \chi_{N_{\varepsilon_j,r_j}}$, it follows that $\|f \|_{L^{p_d,u}} \sim M^{1/u} 2^{-2M(d-1)}$. Also, we have that
\[
\| \mathcal A f\|_{L^{q_d,v}} \gtrsim \| \sum_{j=1}^M 2^{2j} 2^{-2j} \chi_{B_{\varepsilon_j,r_j}(x_j)} \|_{L^{q_d,v}}
= |\bigcup_{j=1}^M B_{\varepsilon_j,r_j(x_j)}|^{\frac 1 {q_d}} =M^{\frac 1 {q_d}} 2^{-\frac{2dMp_d}{q_d}}.
\]
Thus \eqref{main_inq:mono} implies that $ M^{1/q_d} 2^{-2dMp_d/q_d} \lesssim M^{1/u} 2^{-2M(d-1)}$. Since $2^{2 p_d M(d/q_d - (d-1)/p_d)} = 1$, we get $u \le q_d$ whenever $M >1$. Note that if $M=1$, one can obtain $1 + d(d+1)/(2q) \ge d(d+1)/(2p)$.

$\bullet$ $p_d \le v$ : In this case, we make use of $\mathcal A^\ast \chi_{B_{\varepsilon,r}} \gtrsim \varepsilon^2 r^2$ on $N_{\varepsilon,r}$, which can be obtained by the same calculation as above.
For a positive integer $M >1 $ and $-M \le j \le -1$, we set $\varepsilon_j = 2^{-(j+M)q'_d}$, $r_j = 2^{j q'_d/q_d}$, and choose $x_j$ so that $B_{\varepsilon_j, r_j}(x_j)$ are pairwise disjoint and $N_{\varepsilon_j, r_j}(x_j)$ are also pairwise disjoint.

Since $2^{2 j } |B_{\varepsilon_j,r_j}(x_j)|^{1/q'_d} = 2^{2 j } 2^{-2d(j+M) + d(d+1)j /q_d} = 2^{-2dM }$ for $q_d  = \frac {d(d+1)}{2(d-1)} $. If we set $f =
\sum_{j=-1}^{-M} 2^{2j } \chi_{E_j}$, it follows that $\|f \|_{L^{q'_d,v'}} \sim M^{1/v'} 2^{-2dM}$.

Also, we have that
\begin{align*}
\| \mathcal A^\ast f \|_{L^{p'_d,u'}} & \gtrsim  \| \sum_{j=-1}^{-M} 2^{2j} 2^{-2(j+M)q'_d + 2jq'_d/q_d}  \chi_{N_{\varepsilon_j,r_j}(x_j)} \|_{L^{p'_d,u'}} \\
& = 2^{-2Mq'_d} |\bigcup_{j=-1}^{-M} N_{\varepsilon_j,r_j(x_j)}|^{\frac 1 {p'_d}} = M^{\frac 1 {p'_d}} 2^{-2Mq'_d -\frac{2M(d-1)q'_d }{p'_d}}.
\end{align*}
If \eqref{main_inq:mono} holds, then we get $
M^{\frac 1 {p'_d}} 2^{-2Mq'_d -\frac{2M(d-1)q'_d }{p'_d}} \lesssim M^{1/v'} 2^{-2dM}$. Since $(p_d,q_d) = (\frac{d+1}2, \frac{d(d+1)}{2(d-1)})$ satisfies $\frac d{q_d} = \frac{d-1}{p_d}$ or $1 + \frac{d-1}{p'_d} = \frac d {q'_d}$, it follows that $v' \le p'_d$ for any positive integer $M>1$.
\end{section}

%%%%%%%%%%%%%%%%%%%%%%%%%%%%%%%%%%%%
%
\bibliographystyle{plain}

% ----------------------------------------------------------------

\end{document}